\documentclass[11pt, twoside, letterpaper]{article}
\usepackage{color,amsmath,amssymb,graphicx,alltt,xy,amsthm,setspace,cancel}
\usepackage{amsmath,amssymb,graphicx,alltt,xy,amsthm,}
\usepackage{amsmath,amssymb,graphicx,alltt,amsthm}
\usepackage{upgreek}
\usepackage{natbib}
\usepackage{endnotes}
\usepackage{enumerate}
\usepackage{enumitem}
\usepackage{titlesec}
\usepackage{etoolbox}

\usepackage[colorlinks=true,breaklinks=true,bookmarks=true,urlcolor=blue,
citecolor=blue,linkcolor=blue,bookmarksopen=false,draft=false]{hyperref}

\providecommand{\keywords}[1]{\begin{small} \emph{\textit{Key words:}} #1\end{small}}

\titleformat{\section}[runin]
{\normalfont\bf}{\thesection.}{.5em}{}
\titlespacing{\section}
{\parindent}{*2}{\wordsep}

\titleformat{\subsection}[runin]
{\normalfont\bf}{\thesubsection.}{.5em}{}
\titlespacing{\subsection}
{\parindent}{*2}{\wordsep}

\titleformat{\subsubsection}[runin]
{\normalfont\bf}{\thesubsubsection}{.5em}{}
\titlespacing{\subsubsection}
{\parindent}{*2}{\wordsep}

\AtBeginEnvironment{matrix}{\setlength{\arraycolsep}{.1em}}

\setlist{nosep}
\setcitestyle{numbers,square}

\usepackage{caption}
\usepackage{bbm}
\usepackage{subfloat}

\usepackage{placeins}

\usepackage{verbatim}

\definecolor{refkey}{rgb}{1,0.5,0.5}
\definecolor{labelkey}{rgb}{0.5,1,0.5}
\definecolor{MyDarkGreen}{rgb}{0,0.45,0}
\definecolor{MyDarkBlue}{rgb}{0,0,0.75}
\definecolor{MyDarkRed}{rgb}{0.9,0,0}

\usepackage[margin=1in]{geometry}

\date{}
\input xy
\xyoption{all}
\newtheorem{theorem}{Theorem}
\newtheorem{lemma}{Lemma}

\newtheorem{definition}[theorem]{Definition}

\newtheorem{corollary}{Corollary}

\newtheorem{example}[lemma]{Example}

\newcounter{example}

\newcommand{\Remark}{\vspace{0mm} \parindent=0pt
	{\bf Remark.} \hspace{0mm} \parindent=3ex}
\newcommand*{\rom}[1]{\expandafter\@slowromancap\romannumeral #1@}

\newcommand{\Question}{\vspace{0mm} \parindent=0pt
	{\bf Question.} \hspace{0mm} \parindent=3ex}

\newcommand{\rev}{}
\newcommand{\newrev}{}

\newcommand{\E}{{{\bf E}}}

\newcommand{\Mod}[1]{\mathrm{mod}\ #1}

\newcommand{\tworow}{\genfrac{}{}{0pt}{}}

\usepackage{tikz}

\title{Diffusion of new products with heterogeneous consumers}
\author{Gadi Fibich and Amit Golan}
\date{\small{Department of Applied Mathematics, Tel Aviv University, \href{mailto:fibich@tau.ac.il}{fibich@tau.ac.il}, \href{mailto:amitgolan33@gmail.com}{amitgolan33@gmail.com}} \\
	\normalsize{\today}}

\makeatletter
\let\inserttitle\@title
\makeatother
\usepackage{fancyhdr}
\begin{document}

\maketitle
	\begin{abstract}
		Does a new product spread faster among heterogeneous or homogeneous consumers? We analyze this question using the stochastic discrete Bass model,
		in which consumers may differ in their individual external influence rates $\{p_j \}$ 
		and in their individual internal influence rates $\{ q_j \}$.
		When the network is complete and the heterogeneity is only manifested in $\{p_j \}$ or only in $\{ q_j \}$, 
		it always slows down the diffusion, compared to the corresponding homogeneous network.
		When, however, consumers are heterogeneous in both~$\{p_j\}$ and~$\{q_{j}\}$, heterogeneity slows down the diffusion in some cases, but accelerates it in others. Moreover, 
		the dominance between the heterogeneous and homogeneous adoption levels is global in time
		in some cases, but changes with time in others. Perhaps surprisingly, global dominance between two networks is not always preserved under ``additive transformations'', 
		such as adding an identical node to both networks.  When the network is not complete, the effect of heterogeneity  depends also on its spatial distribution within the network. 
	\end{abstract}
		

	\section{Introduction.}
\label{sec:introduction}

The study of the diffusion of innovations started in the sociology literature~\cite{de1903laws}
and expanded over the years~\cite{Rogers-03}. 
More generally, 
diffusion in social networks has attracted the attention of researchers  
in physics, mathematics, biology, computer science, social sciences, economics, and management science,
as it concerns the spreading of ``items'' ranging from diseases and computer viruses to rumors, information, opinions, technologies, and 
innovations~\cite{Albert-00,Anderson-92,Jackson-08,Pastor-Satorras-01,Strang-98}. 
In marketing, diffusion of new products plays a key role, with applications in retail service, industrial technology, agriculture, and in educational, pharmaceutical, and consumer-durables markets \cite{Mahajan-93}.

The first quantitative model of the diffusion of new products was proposed in 1969 by 
Bass~\cite{Bass-69}. 
In this model, we consider a population of size $M$, and denote by $n = n(t)$ the number of individuals who adopted the product by time $t$. Each of the 
$(M-n)$ nonadopters may adopt the product due to  {\em external influences} by mass media
at a constant rate of~$p$, and due to {\em internal influences} by individuals who already adopted the product, at the rate of~$\frac{q}{M}n$. Thus, the rate of internal influences increases linearly with the number of adopters. 
The individual adoption rate of each nonadopter is the sum of her external and internal adoption rates, i.e.  
\begin{equation}
	\label{eq:Bass_intro-rate}
	p+\frac{q}{M}n(t).
\end{equation}
Therefore, the rate of change of the number of adopters is 
\begin{equation}
	\label{eq:Bass_intro}
	n'(t)= \left(M-n(t)\right)\left(p+\frac{q}{M}n(t)\right).
\end{equation}

The Bass model~\eqref{eq:Bass_intro} inspired a huge amount of follow-up research; in~2004 it was chosen as one of the most cited papers in the fifty-year history of \textit{Management Science}~\cite{Hopp-04}.
From a modeling perspective, it is a {\em compartmental model}. Thus, the population is divided into two {\em compartments} (groups), adopters and nonadopters, and eq.~\eqref{eq:Bass_intro} provides the rate at which individuals move between these two compartments.
Most of the extensions of the Bass model have also been compartmental models;  
given by a deterministic ODE or ODEs.  As a result, they are relatively easy to analyze.  
Compartmental models, however, make two implicit assumptions, whose validity is highly  questionable: 
\begin{enumerate}
	\item[{\bf (A1)}] All individuals within the population are equally-likely to influence each other. In other words, the underlying social network is a {\em complete graph}. 
	\item[{\bf (A2)}] All individuals within the population are {\em homogeneous}, i.e., they all have the same $p$ and $q$.
\end{enumerate} 

To check the consequences of these assumptions, one needs to go back to the more fundamental  discrete model for the stochastic adoption of each individual in the population~\cite{Rand-11}. 
For example, the discrete analogue of the compartmental model~\eqref{eq:Bass_intro} is, cf.~\eqref{eq:Bass_intro-rate},
\begin{equation}
	\label{eq:Bass-model-heterogeneous}
	\text{Prob}\left(\tworow{j~\text{adopts in}}{(t,t+\Delta t)} \, \bigg| \,
	\tworow{j~\text{did not adopt}}{\text{by time t}}\right) = \left( p_j +  \frac{q_j}{M_j}N_j(t) \right)\Delta t,
\end{equation}
where $p_j$ and $q_j$ are the rates of external and internal influences on~$j$, $M_j$ is the number of peers (the {\em degree}) of $j$, and  $N_j(t)$ is the number of adopters at time $t$ among her $M_j$~peers.

Discrete stochastic Bass models are considerably harder to analyze than compartmental models. They enable us, however, to relax the assumptions of a complete network and of homogeneity.
{\em Most of the analysis of the discrete Bass model so far has been concerned with the role of the network structure.}
Niu~\cite{Niu-02} showed that as $M\to\infty$, the discrete Bass model on a {\em homogeneous  complete network} approaches the compartmental Bass model~\eqref{eq:Bass_intro}. 
Fibich and Gibori~\cite{OR-10} analyzed the discrete Bass model on {\em Cartesian networks}. Fibich, Levin, and Yakir~\cite{Bass-boundary-18} analyzed the effect of {\em boundary conditions} in Cartesian networks. Fibich~\cite{Bass-SIR-model-16, Bass-SIR-analysis-17} analyzed the discrete {\em Bass-SIR model},
in which adopters eventually recover and no longer influence others to adopt, on various networks. Fibich and Levin~\cite{FIBICH2020123055} analyzed the
{\em percolation} of new products on various networks, from which a
fraction of the nodes is randomly removed.

All of the above studies analyzed the discrete Bass model on {\it homogeneous} networks, i.e., 
when all individuals have the same $p$ and $q$. 
Goldenberg et al.~\cite{GLM-01} studied numerically the discrete Bass model on complete networks
with heterogeneous consumers, and observed that heterogeneity  has a small effect on the aggregate diffusion. 
To the best of our knowledge, {\em analysis} of the effect of heterogeneity in the  discrete Bass model 
was only done in~\cite{PNAS-12}. In that study, Fibich, Gavious, and Solan used the {\em averaging principle} to estimate the quantitative difference between heterogeneous and homogeneous networks. Specifically, they showed that if the network is translation-invariant
and the heterogeneity is mild,  
the difference between the aggregate diffusion in the heterogeneous and  
the corresponding homogeneous networks scales as $\epsilon^2$, where $\epsilon$ is  
the level of heterogeneity of~$\{p_j \}$ and~$\{q_j \}$.

Several studies used compartmental models to study 
the effect of heterogeneity. 
Bulte and Joshi divided the population into two groups: The influentials with $p = p_1$ and $q = q_1$, and the imitators with $p=0$ and $q = q_2$. Their numerical results revealed that heterogeneity in $p$ and $q$ can change the qualitative behavior of the diffusion~\cite{Bulte2007NewPD}. Chaterjee and Eliashberg constructed a compartmental diffusion model which allowed for  heterogeneity in consumers' initial perceptions and price hurdles. 
While their study did not directly analyze heterogeneity in $p$ and $q$ within the framework of the discrete Bass model,
it also showed that heterogeneity can alter the qualitative behavior of aggregate adoption~\cite{Chatterjee-Eliashberg-90}.

{\em  This paper provides the first-ever analysis of the qualitative effect of heterogeneity in the stochastic discrete Bass model}.  We show that heterogeneity in~$p$ and~$q$ can speed up or slow down the diffusion, compared with the corresponding homogeneous network. This result is surprising, since  heterogeneity only in~$p$ or only in~$q$ always slows down the diffusion. 
In some cases, the dominance between the heterogeneous and homogeneous networks is global in time; in others the dominance flips after some time. When the network is not complete, the effect of heterogeneity also depends on the way in which it is spatially distributed in the network.

{\em The methodological contribution} of this paper consists of several novel analytical tools: The {\em master equations} for heterogeneous networks, {\em explicit expressions for the first 3 derivatives} of the expected adoption on heterogeneous networks at $t=0$,  and a {\em CDF dominance condition} for comparing the diffusion on two networks. 
We also use the {\em dominance principle} for heterogeneous networks, which was introduced in~\cite{Bass-boundary-18}.

From a more general perspective, the vast majority of models in marketing and in economics
assume that all individuals are homogeneous. This assumption is made not because it is believed to hold, but simply because heterogeneous models are typically an order of magnitude harder to analyze than their homogeneous counterparts. This paper thus adds to the relatively thin literature on heterogeneous models in marketing and in economics.

The paper is organized as follows. In Section~\ref{sec:discrete_Bass}, we introduce the heterogeneous discrete Bass model. 
We then review some results for homogeneous complete networks 
(Section~\ref{subsec:homogeneous}), for one-sided and two-sided homogeneous circlar networks (Section~\ref{subsec:homogeneous_circles}), and the {\em dominance principle} for heterogeneous networks (Section~\ref{sec:dominance}).  
In Section~\ref{sec:general_results}, we introduce several novel analytic tools for the heterogeneous discrete Bass model. Thus, in Section~\ref{subsec:master_equations} we derive the {\em  master equations for the heterogeneous Bass model}. This linear system of ODEs can be solved analytically to yield an explicit expression for the expected fraction of adopters in any heterogeneous network;  we provide explicit expressions for $M=2$ and $M=3$. 
In Section~\ref{subsec:initial_dynamics} we derive {\em explicit expressions for the first three derivatives} at $t=0$ of the adoption in heterogeneous networks. These expressions  
allow us to analyze  the {\em initial diffusion dynamics} on heterogeneous and homogeneous networks.
In Section~\ref{subsec:cdf_dominance} we introduce the {\em CDF dominance condition}. This condition  allows us to compare the adoption levels of two different networks by comparing the CDFs of the times of the $m$th adoptions in both networks for $m=1, \dots, M$. We use this tool throughout the paper to compare the diffusion in heterogeneous and homogeneous networks.

In Section~\ref{sec:complete}, we compare heterogeneous complete networks with their homogeneous counterparts that have the same number of nodes, the same average~$p$,  and the same average~$q$. When the heterogeneity is only in~$p$, it always slows down the diffusion (Section~\ref{subsec:complete_p}). This is also the case for networks that are heterogeneous only in~$q$  (Section~\ref{subsec:complete_q}), provided that the heterogeneity in~$q$ is {\em mild}, i.e., that a nonadopter is equally influenced by all other adopters, see~\eqref{eq:Bass-model-heterogeneous}.
In Section~\ref{subsec:complete_pq} we consider networks that are heterogeneous in both~$p$ and~$q$. When the heterogeneities in~$p$ and~$q$ are positively correlated, heterogeneity slows down the diffusion. 
When the heterogeneities in~$p$ and~$q$ are not positively correlated, however, 
the diffusion in the heterogeneous case can be slower than, faster than, or equal to that in the homogeneous case.

Consider two networks for which the adoption in the first is lower than in 
the second for all times.  Will this global-in-time dominance be preserved if we increase the $p_j$ values of all the nodes in both networks by the same amount, or if we add an identical node to both networks? In Section~\ref{subsec:abc_lemma} we show that this is indeed the case when there is a node-wise and edge-wise dominance between the two networks, but 
not necessarily in other cases. {\rev We also show that the dominance between heterogeneous and homogeneous networks is not necessarily global in time, but rather can flip with time.} In Section~\ref{sec:1D} we consider heterogeneous periodic 1D networks (circles). In Sections~\ref{subsec:one_side_circle} and~\ref{subsec:two_sided_circle} we  derive the master equations for heterogeneous one-sided and two-sided circle, respectively.
We  explicitly solve these equations for any $M$ in the one-sided case, and for 
$M=2$ and $M=3$ in the two-sided case. 
We then show that the adoption in a heterogeneous two-sided circle can be higher or lower than that in the corresponding heterogeneous one-sided circle (Section~\ref{subsec:circle_compare}). This is different from the homogeneous case, where diffusion on a two-sided circle is identical
to that on the corresponding one-sided circle~\cite{OR-10}.

The analytic tools developed in this study can also be applied to homogeneous networks.
Indeed,  in~\cite{OR-10}, Fibich and Gibori conjectured that diffusion on infinite homogeneous Cartesian networks becomes faster as the dimension of the network increases. In Section~\ref{sec:cart}, we prove this conjecture for small times.
When the network is not complete, 
the effect of heterogeneity depends also on the way in which distributed among the nodes.
To illustrate this, in Section~\ref{sec:effect_of_order} we consider two heterogeneous one-sided circles that have the same nodes, but differ in the way in which the nodes are distributed in the circle.
We explicitly compute the aggregate adoption for both networks as $M \to \infty$. 
We obtain different expressions, which show that the adoption 
indeed depends also on the spatial distribution of the heterogeneity. {\rev Finally, in Section~\ref{sec:level} we show that as we vary the level of heterogeneity, its effect varies continuously and monotonically (at least for weak heterogeneity). In addition, the effect of the variance of the parameters is much smaller than that of their mean.}
\subsection{Emerging picture.}
\label{subsec:results}

This paper contains numerous results. In order to see the wood for the trees, it is useful to summarize some unifying themes:
\begin{enumerate}
	\item
	When the network is heterogeneous only in $p_j$ or only in $q_j$,  heterogeneity always slows the diffusion for all times.
	
	\item When the heterogeneity is both in $p_j$ and $q_j$, the qualitative effect of heterogeneity  is more complex:
	
	\begin{enumerate}
		\item  When the heterogeneities in  $p$ and $q$ are positively correlated,  heterogeneity always slows the diffusion for all times.
		\item
		When, however,  the heterogeneities in $p$ and $q$ are not positively correlated, heterogeneity can accelerate or slow down the diffusion. Moreover, the dominance between the heterogeneous and homogeneous networks is global in time for some cases, but changes with time for others.
	\end{enumerate}
	
	\item Global dominance between two networks is not necessarily preserved under ``additive transformations'', such as 
	increasing all the $\{p_j \}$ of both networks by the same amount, or adding an identical node to both networks. 
	
	\item
	When a network is heterogeneous and not complete, the effect of heterogeneity on the aggregate diffusion depends also on the spatial distribution of the heterogeneity among the nodes.
\end{enumerate}

\section{The heterogeneous discrete Bass model.}
\label{sec:discrete_Bass}

We begin by introducing the diffusion model which is analyzed in this study. A new product is introduced at time $t=0$ to a network with $M$ potential consumers. We denote by $X_j(t)$ the state of consumer $j$ at time $t$, so that 
\begin{equation*}
	X_j(t)=\begin{cases}
		1, \qquad {\rm if\ consumer}\ j\ {\rm adopts\ the\ product\ by\ time}\ t,\\
		0, \qquad {\rm otherwise.}
	\end{cases}
\end{equation*} 
Since all consumers are nonadopters at $t=0$,
\begin{equation}
	\label{eq:general_initial}
	X_j(0)=0, \qquad j=1,\dots,M.
\end{equation}
Once a consumer adopts the product, it remains an adopter for all time. The underlying social network is represented by a weighted directed graph, where the weight of the edge from node $i$ to node $j$ is $q_{i,j}\geq 0$, and $q_{i,j}=0$ if there is no edge from $i$ to $j$. Thus, if $i$ already adopted the product and $q_{i,j}>0$, her rate of internal influence on consumer $j$ to adopt is $q_{i,j}$. 
In addition, consumer $j$ experiences an external influence to adopt, at the rate of $p_j$. Hence, as $dt \to 0$,
\begin{equation}
	\label{eq:general_model}
	{\rm Prob}(X_j(t+dt)=1)=\begin{cases}
		\hfill 1,\hfill &\qquad {\rm if}\ X_j(t)=1,\\
		\left(p_j+\sum\limits_{\substack{k=1 \\k\neq j}}^Mq_{k,j}X_{k}(t)\right)dt, &\qquad {\rm if}\ X_j(t)=0.
	\end{cases}
\end{equation}
If $j$ is a non-adopter, the maximal internal influence that can be exerted on $j$, which occurs when all her peers are adopters, is denoted by
\begin{equation}
	\label{eq:q_on_node}
	q_j:=\sum_{\substack{k=1\\k\neq j}}^{M}q_{k,j}.
\end{equation}
We assume that any individual can be influenced by at least one other individual, i.e., that
\begin{equation}
	\label{eq:q_j>0}
	q_j>0, \qquad j=1, \dots, M.
\end{equation}
We also denote by
\begin{equation}
	\label{eq:q_out}
	q^k:=\sum_{\substack{j=1\\j\neq k}}^{M}q_{k,j}
\end{equation}
the sum of the internal influences that $k$ exerts on her peers.
\par
{\rev We mostly consider a {\em milder form of heterogeneity} in~$q$, where $\{q_j\}_{j=1}^M$ can be heterogeneous but each individual is equally influenced by any of his peers, i.e.,
	\begin{equation}
		\label{eq:q_mild_het}
		q_{i,j}=\begin{cases}
			\frac{q_j}{{\rm degree}(j)},& {\rm if }\ i\ {\rm influences }\ j,\\
			0,& {\rm otherwise.} 
		\end{cases}
	\end{equation}
	Therefore, the network structure is preserved under mild heterogeneity.}
For example, in the case of a mildly-heterogeneous complete network, \eqref{eq:general_model}~reads
\begin{equation*}
	{\rm Prob}(X_j(t+dt)=1)=\begin{cases}
		\hfill 1, \hfill &\qquad {\rm if}\ X_j(t)=1,\\
		\left(p_j+\frac{q_j}{M-1}N(t)\right)dt, &\qquad {\rm if}\ X_j(t)=0,
	\end{cases}
\end{equation*}
where $N(t):=\sum_{j=1}^{M}X_j(t)$ is the number of adopters at time $t$.

Our main goal is to compute the effect of the heterogeneity in $\{p_j\}$ and $\{q_{i,j}\}$ (or $\{q_j\}$) on the expected number of adopters
\begin{equation}
	\label{eq:n(t)_general}
	n(t):=\E \left[\sum_{j=1}^{M}X_j(t)\right]=\E\left[N(t)\right],
\end{equation}
or equivalently on the expected fraction of adopters 
\begin{equation}
	\label{eq:number_to_fraction}
	f(t):=\frac{1}{M}n(t).
\end{equation}
\subsection{Homogeneous complete networks.}
\label{subsec:homogeneous}

When the network is complete and homogeneous, then 
\begin{equation}
	\label{eq:hom_conditions}
	p_j\equiv p,\qquad q_{i,j}=\frac{q}{M-1},\qquad i,j=1,\dots,M,\qquad i\neq j.
\end{equation} 
In that case, \eqref{eq:general_model}~reads
\begin{equation}
	\label{eq:general_homogeneous}
	{\rm Prob}(X_j(t+dt)=1)=\begin{cases}
		\hfill 1, \hfill & \qquad {\rm if}\ X_j(t)=1,\\
		\left(p+\frac{q}{M-1}N(t)\right)dt, & \qquad {\rm if}\ X_j(t)=0.
	\end{cases}
\end{equation}
Niu~\cite{Niu-02} proved that as $M\to\infty$, the expected fraction of adopters in~\eqref{eq:general_homogeneous} approaches the solution of the compartmental Bass model~\cite{Bass-69}
\begin{equation}
	\label{eq:homogeneous_Bass_frac}
	f'(t)= \left(1-f(t)\right)\left(p+qf(t)\right), \qquad f(0)=0.
\end{equation}
This equation can be solved explicitly, yielding the Bass formula~\cite{Bass-69}
\begin{equation}
	\label{eq:Bass_sol}
	f_{\rm Bass}(t;p,q)=\frac{1-e^{-(p+q)t}}{1+\frac{q}{p}e^{-(p+q)t}}.
\end{equation}

\subsection{Homogeneous circles.}
\label{subsec:homogeneous_circles}

Let us denote by $f^{\rm 1-sided}_{\rm circle}(t;p,q,M)$ the expected fraction of adopters in a homogeneous one-sided circle with $M$ nodes where each individual is only influenced by his left neighbor (see Figure~\ref{fig:structure}A), i.e., $$p_j\equiv p,\qquad q_{i,j}=\begin{cases}q, & \qquad {\rm if }\ (j-i)\Mod M=1,\\ 0, &\qquad {\rm if}\ (j-i)\Mod M\neq 1,\end{cases}\qquad j,i=1,\dots,M.$$ Similarly, denote by $f^{\rm 2-sided}_{\rm circle}(t;p,q^R,q^L,M)$ the expected fraction of adopters in a homogeneous two-sided circle with $M$ nodes where each individual can be influenced by his left and right neighbors (see Figure~\ref{fig:structure}B), i.e., $$p_j\equiv p,\qquad q_{i,j}=\begin{cases}
	q^L& \qquad {\rm if}\ (j-i)\Mod M=1,\\ q^R& \qquad  {\rm if}\ (i-j)\Mod M=1,\\  0&\qquad  {\rm if}\  |j-i|\Mod M\neq 1,
\end{cases}\qquad j,i=1,\dots,M.$$ 
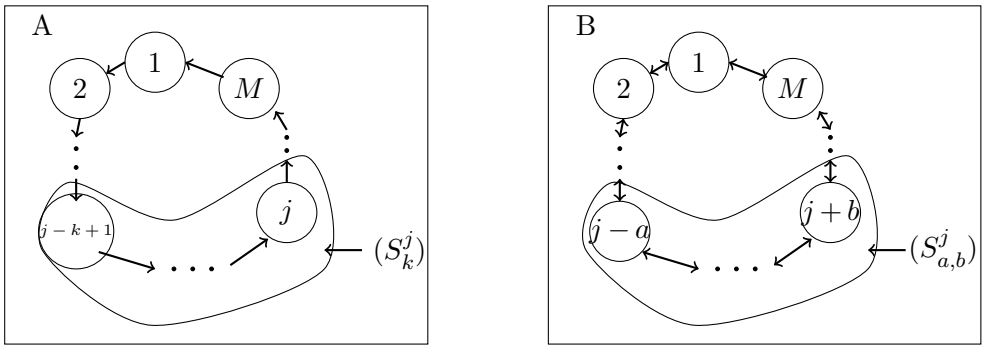
\begin{figure} [h]
	\centering
	\begin{minipage}{.43\textwidth}
		\begin{tikzpicture}[scale=0.5]
			\draw (-3.5,2) circle (0.8cm) node {$1$};
			\draw[->,thick] (-4.3,2)--(-4.8,1.7);
			\draw (-5.5,1.3) circle (0.8cm) node {$2$};
			\draw[->,thick] (-5.5,0.5)--(-5.6,0);
			\draw[black,fill=black] (-5.6,-0.3) circle (.3ex);
			\draw[black,fill=black] (-5.6,-0.8) circle (.3ex);
			\draw[->,thick] (-5.6,-1.1)--(-5.6,-1.7);
			\draw (-5.6,-2.5) circle (1cm) node {\tiny $j-k+1$};
			
			\draw [black] plot [smooth cycle] coordinates {(-5.8,-1.2) (-3,-2.2) (.5,-.5) (1,-3.5)(-3.6,-5) (-6.5,-3)};
			\draw[->,thick] (-5,-3.05)--(-3.5,-3.5);
			\draw[->,thick] (-1.5,-3.4)--(-0.5,-2.7);
			\draw[->,thick] (0,-1.2)--(0,-0.6);
			\draw[black,fill=black] (0,-0.35) circle (.3ex);
			\draw[black,fill=black] (0,0) circle (.3ex);
			\draw[->,thick] (0,0.2)--(-0.3,0.7);
			\draw[->,thick] (-1.68,1.6)--(-2.7,2);
			
			\draw[black,fill=black] (-3,-3.5) circle (.3ex);
			\draw[black,fill=black] (-2,-3.5) circle (.3ex);
			\draw[black,fill=black] (-2.5,-3.5) circle (.3ex);
			
			\draw (-1,1.3) circle (0.8cm) node {$M$};
			\draw (0,-2) circle (0.8cm) node {$j$};
			
			\draw [->, thick] (2,-3)--(1,-3);
			\draw (3,-3) node{$(S_{k}^j)$};
			
			\draw (-6.5,3) node {A};
			
			\draw (-7.5,3.5) rectangle (3.75,-5.5);
			
		\end{tikzpicture}
	\end{minipage}
	\begin{minipage}{.43\textwidth}
		\begin{tikzpicture}[scale=0.5]
			\draw (-3.5,2) circle (0.8cm) node {$1$};
			\draw[<->,thick] (-4.3,2)--(-4.8,1.7);
			\draw (-5.5,1.3) circle (0.8cm) node {$2$};
			\draw[<->,thick] (-5.5,0.5)--(-5.6,0);
			\draw[black,fill=black] (-5.6,-0.3) circle (.3ex);
			\draw[black,fill=black] (-5.6,-0.8) circle (.3ex);
			\draw[<->,thick] (-5.6,-1.1)--(-5.6,-1.7);
			\draw (-5.6,-2.5) circle (0.8cm) node {$j-a$};
			
			\draw [black] plot [smooth cycle] coordinates {(-5.8,-1.2) (-3,-2.2) (.5,-.5) (1,-3.5)(-3.6,-5) (-6.5,-3)};
			\draw[<->,thick] (-5,-3.05)--(-3.5,-3.5);
			\draw[<->,thick] (-1.5,-3.4)--(-0.5,-2.7);
			\draw[<->,thick] (0,-1.2)--(0,-0.6);
			\draw[black,fill=black] (0,-0.35) circle (.3ex);
			\draw[black,fill=black] (0,0) circle (.3ex);
			\draw[<->,thick] (0,0.2)--(-0.3,0.7);
			\draw[<->,thick] (-1.68,1.6)--(-2.7,2);
			
			\draw[black,fill=black] (-3,-3.5) circle (.3ex);
			\draw[black,fill=black] (-2,-3.5) circle (.3ex);
			\draw[black,fill=black] (-2.5,-3.5) circle (.3ex);
			
			\draw (-1,1.3) circle (0.8cm) node {$M$};
			\draw (0,-2) circle (0.8cm) node {$j+b$};
			
			\draw [->, thick] (2,-3)--(1,-3);
			\draw (3,-3) node{$(S_{a,b}^j)$};
			
			\draw (-6.5,3) node {B};
			
			\draw (-7.5,3.5) rectangle (4,-5.5);
			
		\end{tikzpicture}
	\end{minipage}
	\caption{A)~A one-sided circle with the nonadopters chain $(S_{k}^j)$. B)~A two-sided circle with the nonadopters chain $(S_{a,b}^j)$.}
	\label{fig:structure}
\end{figure}

In \cite{OR-10}, Fibich and Gibori proved that
\begin{subequations}
	\label{eqs:compare_Fibich_Gibori}
	the diffusion on one-sided and two-sided homogeneous circles are identical 
	\begin{equation}
		\label{eq:compare_Fibich_Gibori}
		f_{\rm circle}^{{\rm 1-sided}}(t; p,q,M)\equiv f_{\rm circle}^{\rm 2-sided}(t;p,q^R,q^L,M), \qquad 0\leq t< \infty,
	\end{equation}
\end{subequations}
provided that the maximal internal influence experienced by each node is identical in both cases, i.e.,that 
\begin{equation}
	\label{eq:q_Lq_R}
	q=q^L+q^R.
\end{equation}
In addition, they explicitly computed the diffusion on infinite homogeneous circles:
\begin{equation}
	\label{eq:f_1D}
	\lim\limits_{M\to \infty}f_{\rm circle}^{\rm 1-sided}(t;p,q,M)=f_{\rm 1D}(t;p,q):=1-e^{-(p+q)t+q\frac{1-e^{-pt}}{p}}.
\end{equation}

\subsection{Dominance principle.}
\label{sec:dominance}

A useful tool for comparing the diffusion in two networks is the dominance principle.
Let us begin with the following definition:

\begin{definition}[node-wise and edge-wise dominance]
	\label{def:dominance}
	Consider the heterogeneous discrete Bass model~\eqref{eq:general_model} on networks A and B with $M$~nodes, with external parameters $\{p_i^A\}$ and \{$p_i^B\}$, and internal parameters  $\{q_{i,j}^{A}\}$ and $\{q_{i,j}^{B}\}$, respectively. 
	We say that $A \preceq B$ if 
	$$
	p_j^{A} \leq p_j^B \quad \mbox{for all~j} \qquad \mbox{and} \qquad  q_{i,j}^{A} \leq q_{i,j}^{B} \quad \mbox{for all~} i \not=j.
	$$
	We say that $A \prec B$ if at least one of these $M^2$ inequalities is strict.
\end{definition}

\begin{lemma}[Dominance principle~\cite{Bass-boundary-18}]
	\label{lem:dominance-principle}
	If $A \preceq B$ then $f_{A}(t) \leq f_{B}(t)$ for $t>0$. If $A \prec B$, then $f_{A}(t) < f_{B}(t)$ for $t>0$.
\end{lemma}
	\section{Analytic tools.}
\label{sec:general_results}
In this section we introduce several novel analytical tools. These tools will be later used to analyze the effect of heterogeneity in the heterogeneous discrete Bass model~\eqref{eq:general_model}.

\subsection{Master equations.}
\label{subsec:master_equations}

In  order to analytically compute the expected number of adopters, we derive the master equations for a general heterogeneous network with $M$ nodes, as follows.
Let $(S^{m_1},\dots,S^{m_n})(t)$ denote the event that at time $t$, nodes $\{m_1,\dots,m_n\}$ are non-adopters, where $1\leq n\leq M$, $m_i\in \{1,\dots,M\}$, and $m_i\neq m_j$ if $i\neq j$. Let $[S^{m_1},\dots,S^{m_n}](t)$ denote the probability that such an event occurs.

\begin{lemma}
	\label{lem:masterHeterofc}
	The master equations for the heterogeneous discrete Bass model~\eqref{eq:general_model} are
	\begin{subequations}
		\label{eqs:masterHeterofc}
		\begin{equation}
			\label{eq:masterHeterofc}
			\begin{aligned}
				\frac{d}{dt}&[S^{m_1},\dots,S^{m_n}](t)=\\ &-\left(\sum_{i=1}^{n}p_{m_i}+\sum_{j=n+1}^{M}\sum_{i=1}^{n} q_{l_j,m_i}\right)[S^{m_1},\dots,S^{m_n}](t) +\sum_{j=n+1}^{M}\left(\sum_{i=1}^{n} q_{l_j,m_i}\right)[S^{m_1},\dots,S^{m_n},S^{l_j}](t),
			\end{aligned}
		\end{equation}
		for any $\{m_1,\dots,m_n\} \subsetneq \{1,\dots,M\}$, where $\{l_{n+1},\dots,l_M\}=\{1,\dots,M\}\backslash\{m_1,\dots,m_n\}$, and
		\begin{equation}
			\label{eq:masterHeterofcM}
			\frac{d}{dt}[S^{1},S^{2},\dots,S^{M}](t)=-\left(\sum_{i=1}^M p_i\right)[S^{1},S^{2},\dots,S^{M}](t),
		\end{equation}
		subject to the initial conditions
		\begin{equation}
			\label{eq:masterHeterofcInitial}
			[S^{m_1},\dots,S^{m_n}](0)=1, \qquad \forall \{m_1,\dots,m_n\}\subset \{1,\dots,M\}.
		\end{equation}
	\end{subequations}
\end{lemma}
\begin{proof}
	See~\ref{app:lem:masterHeterofc}.
\end{proof}
The master equations~\eqref{eqs:masterHeterofc} constitute a linear system of $2^M-1$ differential equations for all possible subsets $\{m_1,\dots,m_n\}\subset \{1,\dots,M\}$.
These equations can be solved explicitly, as follows. By~(\ref{eqs:masterHeterofc}b-c),
\begin{equation}
	\label{eq:s_M sol}
	[S^1,\dots,S^M](t)=e^{-\left(\sum_{j=1}^{M}p_j\right)t}.
\end{equation}
Proceeding to solve~\eqref{eqs:masterHeterofc} backwards from $n=M-1$ to $n=1$ gives~(\ref{app:S^k})
\begin{equation}
	\label{eq:master_sol}
	[S^k](t)=\sum_{n=1}^{M}\sum_{\{m_1=k,m_2,\dots,m_n\}}c_{\{m_1,\dots,m_n\}}e^{-\left(\sum_{i=1}^{n}p_{m_i}+\sum_{j=n+1}^{M}\sum_{i=1}^{n} q_{l_j,m_i}\right)t},
\end{equation}
where $c_{\{m_1,\dots,m_n\}}$ are constants.
Once we solve for $\{[S^k]\}_{k=1}^{M}$, the expected fraction of adopters in the network is given by, see~\eqref{eq:n(t)_general} and~\eqref{eq:number_to_fraction},
\begin{equation}
	\label{eq:expectedHeterofc}
	f(t; \{p_i\}, \{q_{i,j}\})=1-\frac{1}{M}\sum\limits_{k=1}^{M}[S^k](t).
\end{equation}

For example, the master equations~\eqref{eqs:masterHeterofc} for $M=2$ read
\begin{subequations}
	\label{eqs:master_M=2}
	\begin{align}
		\label{eq:master_M=2}
		\frac{d}{dt}[S^i](t)&=-(p_i+q_{i-1,i})[S^i](t)+q_{i-1,i}[S^1,S^2](t), \qquad i=1,2,\\
		\label{eq:master_M=2b}
		\frac{d}{dt}[S^1,S^2](t)&=-(p_1+p_2)[S^1,S^2](t),\\
		[S^1](0)&=[S^2](0)=[S^1,S^2](0)=1.
	\end{align}
\end{subequations}
Solving this system for $M=2$ yields
\begin{equation*}
	[S^1](t)=a_1e^{-\left(p_1+q_{2,1}\right)t}+b_1e^{-\left(p_1+p_2\right)t}, \qquad [S^2](t)=a_2e^{-\left(p_2+q_{1,2}\right)t}+b_2e^{-\left(p_1+p_2\right)t}.
\end{equation*}
Therefore, by~\eqref{eq:expectedHeterofc},
\begin{equation}
	\label{eqs:fc_M=2}
	f(t)=1-\frac{1}{2}\sum_{j=1}^{2}\left[a_je^{-(p_j+q_{j-1,j})t}-b_je^{-(p_1+p_2)t}\right], \quad
	a_j=\frac{p_{j-1}}{p_{j-1}-q_{j-1,j}}, \qquad b_j=\frac{q_{j-1,j}}{p_{j-1}-q_{j-1,j}}.
\end{equation}
Similarly, when $M=3$,
\begin{subequations}
	\label{eqs:fc_M=3}
	\begin{equation}
		\label{eq:fc_M=3}
		\begin{aligned}
			f(t)=
			1-\frac{1}{3}\sum\limits_{j=1}^{3} \bigg[a_je^{-(p_j+q_{j-1,j}+q_{j+1,j})t}-b_je^{-(p_j+p_{j+1}+q_{j-1,j}+q_{j+2,j+1})t}+c_je^{-(p_1+p_2+p_3)t}\bigg],
		\end{aligned}
	\end{equation}
	where
	\begin{align}
		\label{eq:fc_M=3_coefs}
		a_j=1&+\bigg(1+\frac{q_{j-1,j}+q_{j+2,j+1}}{p_{j-1}-q_{j-1,j}-q_{j+2,j+1}}\bigg)
		\frac{q_{j+1,j}}{p_{j+1}+q_{j+2,j+1}-q_{j+1,j}} 
		\\&
		+\bigg(1+\frac{q_{j-2,j-1}+q_{j+1,j}}{p_{j+1}-q_{j-2,j-1}-q_{j+1,j}}\bigg)
		\frac{q_{j-1,j}}{p_{j-1}+q_{j-2,j-1}-q_{j-1,j}} 
		\nonumber\\&
		-\frac{q_{j+1,j}\frac{q_{j-1,j}+q_{j+2,j+1}}{p_{j-1}-q_{j-1,j}-q_{j+2,j+1}}+q_{j-1,j}\frac{q_{j-2,j-1}+q_{j+1,j}}{p_{j+1}-q_{j-2,j-1}-q_{j+1,j}}}{p_{j+1}+p_{j-1}-q_{j-1,j}-q_{j+1,j}},\nonumber\\
		b_j=&\bigg(1+\frac{q_{j-1,j}+q_{j+2,j+1}}{p_{j-1}-q_{j-1,j}-q_{j+2,j+1}}\bigg)\bigg(\frac{q_{j+1,j}}{p_{j+1}+q_{j+2,j+1}-q_{j+1,j}}+\frac{q_{j,j+1}}{p_{j}+q_{j-1,j}-q_{j,j+1}}\bigg),\\
		c_j=&\frac{q_{j+1,j}\frac{q_{j-1,j}+q_{j+2,j+1}}{p_{j-1}-q_{j-1,j}-q_{j+2,j+1}}+q_{j-1,j}\frac{q_{j-2,j-1}+q_{j+1,j}}{p_{j+1}-q_{j-2,j-1}-q_{j+1,j}}}{p_{j+1}+p_{j-1}-q_{j-1,j}-q_{j+1,j}}.
	\end{align}
\end{subequations}
\Remark \label{rem:q_{i,j}}{\rev The subscripts of $q_{i,j}$ in~\eqref{eqs:master_M=2}-\eqref{eqs:fc_M=3} are modulo $M$ and in $\{1,\dots,M\}$. For example, if $i=1$, then ``$i-1=M$", and if $i=M$, then ``$i+2=2$", etc.}

\Remark The master equations~\eqref{eqs:masterHeterofc} hold for heterogeneous (and homogeneous) networks {\em with any structure}. For example, we can have a one-sided circle by setting $q_{i,j}=0$ for $(j-i)\Mod M\neq 1$ (Section~\ref{subsec:one_side_circle}), a two-sided circle by setting $q_{i,j}=0$ for $|j-i| \Mod M\neq 1$ (Section~\ref{subsec:two_sided_circle}), a D-dimensional Cartesian structure by setting $q_{i,j}$ as in eq.~\eqref{eq:cart_hom} below (Section~\ref{sec:cart}), etc.

\subsection{Initial dynamics.}
\label{subsec:initial_dynamics}

We can use the master equations~\eqref{eqs:masterHeterofc} to analyze the initial dynamics, by deriving explicit expressions for $f'(0)$, $f''(0)$, and $f'''(0)$. We begin with the most general heterogeneity:
\begin{lemma}
	\label{lem:n_initial_dynamics}
	Consider the heterogeneous discrete Bass model~\eqref{eq:general_model}. 
	Then
	\begin{equation}
		\label{eq:n'(0)}
		f'(0)=\frac{1}{M}\sum_{j=1}^{M}p_j,\qquad
		f''(0)=\frac{1}{M}\left(\sum_{i=1}^{M}p_iq^i-\sum_{i=1}^{M}p_i^2\right),
	\end{equation}
	where $q^i$ is defined in~\eqref{eq:q_out}. In addition, if the heterogeneity in $q$ is mild, see~\eqref{eq:q_mild_het}, then 
	\begin{equation}
		\label{eq:n''(0)_mild}
		f''(0)=\frac{1}{M(M-1)}\left[\sum_{j=1}^{M}q_j\sum_{i=1}^{M}p_i-\sum_{j=1}^{M}q_jp_j\right]-\frac{1}{M}\sum_{i=1}^{M}p_i^2.
	\end{equation}
\end{lemma}
\begin{proof}
	Substituting $t=0$ in~\eqref{eq:masterHeterofc} and using~\eqref{eq:masterHeterofcInitial} gives
	\begin{equation}
		\label{eq:s'_0}
		\frac{d}{dt}[S^{m_1},\dots,S^{m_n}](0)=-\left(\sum_{i=1}^{n}p_{m_i}+\sum_{j=n+1}^{M}\sum_{i=1}^{n} q_{l_j,m_i}\right)+\sum_{j=n+1}^{M}\left(\sum_{i=1}^{n} q_{l_j,m_i}\right)=-\sum_{i=1}^{n}p_{m_i}.
	\end{equation}
	Hence, by~\eqref{eq:number_to_fraction} and~\eqref{eq:expectedHeterofc}, we get eq.~\eqref{eq:n'(0)} for~$f'(0)$. 
	Differentiating~\eqref{eq:masterHeterofc} and using~\eqref{eq:s'_0} gives the equation for $f''(0)$. Substituting~\eqref{eq:q_mild_het} in~\eqref{eq:n'(0)} gives~\eqref{eq:n''(0)_mild}.
\end{proof}
These explicit expressions allow us to determine on which network the diffusion is initially faster.
The expressions for the derivatives become simpler when the heterogeneity is just in~$q$:\footnote{This is the case, e.g., for homogeneous Cartesian networks, see eq.~\eqref{eq:cart_hom} in Section~\ref{sec:cart}.}
\begin{corollary}
	\label{cor:n'''(0)_mild}
	Consider a network of size $M$ which is homogeneous in $p$ and heterogeneous in $q$. Then~\eqref{eq:n'(0)} reads
	\begin{equation}
		\label{eq:n''(0)_p_hom}
		f'(0)=p,\qquad f''(0)=p\left(\frac{1}{M}\sum_{j=1}^{M}q_{j}-p\right).
	\end{equation}
	If, in addition, the heterogeneity in $q$ is mild, see~\eqref{eq:q_mild_het}, then
	\begin{equation}
		\label{eq:n'''(0)_mild}
		f'''(0)=p^3+\frac{p}{M}\left(\frac{M-2}{(M-1)^2}\left(\sum_{i=1}^{M}q_i\right)^2-\frac{2M-3}{(M-1)^2}\sum_{i=1}^{M}q_i^2-4p\sum_{i=1}^{M}q_i\right).
	\end{equation}
\end{corollary}
\begin{proof}
	See~\ref{app:cor:n'''(0)_mild}.
\end{proof}
Finally, we consider the homogeneous case:
\begin{corollary}
	\label{cor:initial_hom}
	Consider a complete homogeneous network with $p$, $q$, and $M$, see~\eqref{eq:hom_conditions}. Then
	\begin{equation}
		\label{eq:initial_hom}
		f'(0)=p,\qquad f''(0)=p(q-p),\qquad f'''(0)=p\left(p^2-4pq+\frac{M-3}{M-1}q^2\right).
	\end{equation}
\end{corollary}
\begin{proof}
	This follows from Corollary~\ref{cor:n'''(0)_mild}.	
\end{proof}

\subsection{CDF dominance condition.}
\label{subsec:cdf_dominance}

In this section, we derive a sufficient condition for the adoption in network~$A$ to be slower than in network~$B$. Let $A$ be a network with $M$ nodes. For a specific realization of the discrete model~\eqref{eq:general_model}, let $t_i^A$ denote the time between the $(i-1)$th and $i$th adoptions, where $i=1,\dots,M$ and $t_0^A:=0$. Therefore, the time of the $m$th adoption is
$$
T_m^A :=  t_0 + t_1 + \cdots + t_m , \qquad m=0,1,\dots,M,
$$
and the number of adopters at time~$t$ is given by
\begin{equation}
	\label{eq:N(t)}
	N_A(t) = \max\{ m \in \{0,\ldots,M\} \colon T_m^A \leq t\}. 
\end{equation}
The expected number of adopters in network $A$ is
$	n_A(t)=\E [N_A(t)]$.

Let $B$  be a different network with $M$ nodes, and define $t_i^B$, $T_m^B$, $N_B(t)$, and $n_B(t)$ in a similar manner. We now show that the adoption in $A$ is slower than in $B$, if $\{t_i^A\}$ and $\{t_i^B\}$ satisfy a certain CDF dominance condition:

\begin{theorem}
	\label{thm:CDF-dominance}
	Let $\{t_i^A(\omega_i)\}_{i=1}^M$ and $\{t_i^B(\omega_i)\}_{i=1}^M$ be two sequences of  independent nonnegative random variables that satisfy the  CDF dominance condition
	\begin{subequations}
		\label{eqs:equ:1_2}
		\begin{equation}
			\label{equ:1}
			F_{t_i^A}(\tau) \leq F_{t_i^B}(\tau), \qquad 1\leq i \leq M, \qquad \tau\geq 0,
		\end{equation}
		such that there exists at least one index $1\leq j\leq M$ for which
		\begin{equation}
			\label{equ:2}
			F_{t_j^A}(\tau)<F_{t_j^B}(\tau), \qquad \tau> 0.
		\end{equation} 
	\end{subequations}
	Then the expected number of adopters in $A$ is less than that in $B$, i.e.
	\[ n_A(t)<n_B(t), \qquad  0<t<\infty . \]
\end{theorem}

\begin{proof}
	See~\ref{app:CDF-dominance}. 
\end{proof}	
		\section{Heterogeneity in complete networks.}
\label{sec:complete}

In this section we consider the qualitative effect of heterogeneity in complete networks. To do that, we compare the adoption in a heterogeneous network with that in a homogeneous network that has the same number of nodes~$M$, the same average $\{p_j\}$, and the same average $\{q_{i,j}\}$ or $\{q_j\}$. Thus, we compare $f^{\rm het}(t; \{p_i\}_{i=1}^M, \{q_{i,j}\}_{i,j=1}^M)$ 
or $f^{\rm het}(t; \{p_i\}_{i=1}^M, \{q_{j}\}_{j=1}^M)$ with $f^{\rm hom}(t; p,q,M)$, where
\begin{equation}
	\label{eq:fair}
	p=\frac{1}{M}\sum_{i=1}^Mp_i, \qquad
	q=\frac{1}{M}\sum_{i=1}^M\sum_{\substack{j=1\\j\neq i}}^Mq_{i,j}=\frac{1}{M}\sum_{j=1}^{M}q_j.
\end{equation}

\subsection{Heterogeneity in $p$.}
\label{subsec:complete_p}

We begin with complete networks that are heterogeneous in $p$ and homogeneous in $q$, i.e.,  
\begin{equation}
	\label{eq:p_het_q_hom}
	q_{i,j}\equiv \frac{q}{M-1},\qquad \forall j\neq i.
\end{equation}
When the network has just two nodes, we can use the master eqs.~\eqref{eqs:master_M=2} for $M=2$ to show
that heterogeneity in $p$ always slows down the diffusion:
\begin{lemma}
	\label{lem:compare_M=2p}
	Consider a heterogeneous network with $M=2$, $p_1\neq p_2$ and $q_{1,2}=q_{2,1}\equiv q$, and 
	let  $p=\frac{p_1+p_2}{2}$. Then
	\begin{equation*}
		f^{\rm het}(t; \{p_1,p_2\},q)<f^{\rm hom}(t; p,q,M=2), \qquad 0<t<\infty.
	\end{equation*}
\end{lemma}
\begin{proof}
	See~\ref{app:lem:compare_M=2p}.	
\end{proof}

One could try to generalize this result to any network size~$M$, by induction on the network size $M$. To do that, we need a result that global dominance between two networks is preserved when we add a new identical node to both networks.  As Lemma~\ref{lem:ABC_counter_2} in Section~\ref{subsec:abc_lemma} will show, however, this is not always the case. Therefore, we take a different approach, and generalize Lemma~\ref{lem:compare_M=2p} to any network size~$M$, by making use of the {\em CDF dominance condition}:

\begin{theorem}
	\label{thm:compare_p}
	Consider a complete graph with $M$ nodes which are heterogeneous in $\{p_i\}_{i = i}^M$ and homogeneous in~$q$, see~\eqref{eq:p_het_q_hom}. Let $p =  \frac{1}{M} \sum_{i = 1}^M p_i$. 
	Then 
	$$
	f^{\rm het}(t; \{p_i\}_{i=1}^M,q) < f^{\rm hom}(t; p,q,M), \qquad 0<t<\infty.
	$$
\end{theorem}
\begin{proof}
	Let $t_k^{\rm hom}$ denote the time between the $(k-1)$th and $k$th adoptions in the homogeneous network, where $k=1, \dots, M$. Let $F^{\rm hom}_k(\tau):= {\rm Prob}(t_k^{\rm hom} \le \tau)$ denote the cumulative distribution function (CDF) of~$t_k^{\rm hom}$.
	Let $t_k^{\rm het}$ and $F^{\rm het}_k(\tau)$ be defined similarly for the heterogeneous network. {\rev We now introduce two auxiliary lemmas:
		\begin{lemma}
			\label{lem:F^heter_k<F^hom_k_pq}
			Consider a complete graph with $M$ nodes which are heterogeneous in $\{p_i\}_{i = 1}^M$ and in~$\{q_i\}_{i = 1}^M$ where $q_i$ is the influence exerted on (and not by) node $i$. Furthermore, assume that the $p_i$ values and $q_i$ values are positively correlated, i.e. $p_1\leq p_2 \leq \dots \leq p_M$ and $q_1\leq q_2 \leq\dots \leq q_M$. Let $p =  \frac{1}{M} \sum_{i = 1}^M p_i$ and $q =  \frac{1}{M} \sum_{i = 1}^M q_i$. 
			Let $t_1^{\rm hom}$ denote the time until the first adoption, and $t_k^{\rm hom}$ denote the time between the $(k-1)$th and $k$th adoptions in the homogeneous network
			for $k=2, \dots, M$. Let $F^{\rm hom}_k(t):= {\rm Prob}(t_k^{\rm hom} \le t)$ denote the cumulative distribution function (CDF) of~$t_k^{\rm hom}$ for $k=1, \dots, M$.
			Let $t_k^{\rm het}$ and $F^{\rm het}_k(t)$ be defined similarly for the heterogeneous network.
			Then $	F^{\rm het}_1(t)=F^{\rm hom}_1(t)$ for $t\geq 0$, and
			\begin{equation}
				\label{eq:F^heter_k<F^hom_k_pq}
				F^{\rm het}_k(t)  <  F^{\rm hom}_k(t), \qquad k = 2, \dots, M, \qquad 0<t<\infty.
			\end{equation}
		\end{lemma}
		\begin{proof}
			See~\ref{app:cdf_dominance}.			
		\end{proof}
		\begin{lemma}
			\label{lem:t_i_independent}
			The random variables $(t_i)$ are independent.
		\end{lemma}
		\begin{proof}
			See~\ref{app:cdf_dominance}			
		\end{proof}
	}
	By Lemma~\ref{lem:F^heter_k<F^hom_k_pq}, (in the special case where $q_1=\dots=q_M=q$), $F^{\rm het}_1(\tau)  \equiv  F^{\rm hom}_1(\tau)$ for $\tau\geq 0$, and
	\begin{equation*}
		F^{\rm het}_k(\tau)<F^{\rm hom}_k(\tau), \qquad k = 2, \dots, M, \qquad 0
		<\tau<\infty.
	\end{equation*}
	In addition, by Lemma~\ref{lem:t_i_independent}, the random variables $\{t_k^{\rm hom}\}$ and the random variables $\{t_k^{\rm het}\}$ are independent. Hence, the conditions of Theorem~\ref{thm:CDF-dominance} are satisfied, and so Theorem~\ref{thm:compare_p} follows by~\eqref{eq:number_to_fraction}.	
\end{proof}
$\\$
\Remark
We could also prove Theorem~\ref{thm:compare_p} for small times using the explicit expressions for $f'(0)$ and $f''(0)$, see~\ref{app:small_time}.

The  results of this section show that
{\bf in complete networks, heterogeneity in~$\{p_j\}$ always slows down the diffusion.} {\rev This follows from the convexity of the model. Indeed, as the proof of Lemma~\ref{lem:F^heter_k<F^hom_k_pq} shows, the times between consecutive adoptions follow an exponential distribution, for which the CDF is convex in $\{p_j\}$. 
	
	To further motivate this result, consider a heterogeneous network with $M=2$ nodes, where $p_1=2p$, $p_2=0$, and $q_1=q_2=q$, and its homogeneous counterpart with $p_1=p_2=p$ and $q_1=q_2=q$. In both the homogeneous and the heterogeneous cases, the first adoption occurs at a rate of $2p$. In the homogeneous case, however, the second adoption occurs at a rate of $p+q$, whereas in the heterogeneous case it occurs at a rate of $q$. Hence, the adoption is faster in the homogeneous network. More generally, consider an $M=2$ heterogeneous network with $p_1=p+\varepsilon$, $p_2=p-\varepsilon$, and $q_1=q_2=q$, where $0<\varepsilon\leq p$. The first adoption occurs at a rate of $2p$ in both the heterogeneous and homogeneous networks. For the homogeneous network, the second adoption is at the rate of $p+q$. In the heterogeneous network, since $p_1>p_2$, then in the majority of cases, node $1$ is the first to adopt, in which case the second adoption is by node $2$ at a rate of $p-\varepsilon+q$. In the minority of cases, node $2$ is the first to adopt, in which case the second adoption is by node $1$ at a rate of $p+\varepsilon+q$. Therefore, the overall adoption in the heterogeneous case is slower. This intuition can be generalized to networks with $M$ nodes. The first adoption always occurs at the same rate for the homogeneous and heterogeneous cases, but subsequent adoptions are slower in the heterogeneous case.
}

\subsection{Mild heterogeneity in $q$.}
\label{subsec:complete_q}

We now consider complete networks which are  mildly heterogeneous in~$q$ but homogeneous in~$p$, i.e., 
$$
p_i\equiv p, \qquad i=1,\dots,M.
$$
\begin{theorem}
	\label{thm:compare_q}
	Consider a complete graph with $M$ nodes which are mildly heterogeneous in $\{q_i\}_{i = i}^M$, see~\eqref{eq:q_mild_het}, and homogeneous in $p$. Let $q =  \frac{1}{M} \sum_{i = 1}^M q_i$. 
	Then 
	$$
	f^{\rm het}(t; p, \{q_i\}_{i=1}^M) <  f^{\rm hom}(t; p,q,M), \qquad 0<t<\infty.
	$$
\end{theorem}
\begin{proof}
	Let $t_k^{\rm hom}$, $F^{\rm hom}_k(\tau)$, $t_k^{\rm het}$, and $F^{\rm het}_k(\tau)$ be defined as in the proof of Theorem~\ref{thm:compare_p}.
	By Lemma~\ref{lem:F^heter_k<F^hom_k_pq}, in the special case where $p_1=p_2=\dots p_M$, $F^{\rm het}_1(\tau)  \equiv  F^{\rm hom}_1(\tau)$ for $\tau\geq 0$, and
	\begin{equation*}
		F^{\rm het}_k(\tau)<F^{\rm hom}_k(\tau), \qquad k = 2, \dots, M, \qquad 0<\tau<\infty.
	\end{equation*}
	Therefore, the proof is the same as for Theorem~\ref{thm:compare_p}.
\end{proof}

\Remark We can also prove Theorem~\ref{thm:compare_q} for small times using the explicit expressions for $f'(0)$, $f''(0)$, and $f'''(0)$, see~\ref{app:small_time}.

\Remark
Theorem~\ref{thm:compare_q} does {\em not} extend to the case of a general heterogeneity $\{q_{i,j}\}$. For example, if $M=3$, then a complete network is also a two-sided circle, and so, by~\eqref{eqs:compare_Fibich_Gibori}, if $q^L\neq q^R$ and $q^L+q^R\equiv q$, then
\begin{equation*}
	f^{\rm het}(t; p, \{q_{i,j}\})\equiv f^{\rm 2-sided}_{\rm circle}(t;p,q^R,q^L,M=3)\equiv f^{\rm 2-sided}_{\rm circle}(t;p,\frac{q}{2},\frac{q}{2},M=3)\equiv f^{\rm hom}(t; p,q).
\end{equation*} 
{\newrev The question of whether $f^{\rm het} \le f^{\rm hom}$ when the heterogeneity in~$q$ is not mild, is currently open.}

In summary, the  results in this section show that
{\bf in complete networks, mild heterogeneity in $\{q_j\}$ always slows down the diffusion.}

\subsection{Heterogeneity in $p$ and $q$.}
\label{subsec:complete_pq}

So far, we saw that one-dimensional heterogeneity (i.e. just in $p$ or just in $q$) always slows down the adoption. We now show that this is not always the case when the network is heterogeneous in both $p$ and $q$.

\subsubsection{Positive correlation between $\{p_j\}$ and $\{q_j\}$.}

When  $\{p_j\}$ and $\{q_j\}$ are {\em positively correlated},
heterogeneity always slows down the adoption:
\begin{theorem}
	\label{thm:compare_pq}
	Consider a complete graph with $M$ nodes which are heterogeneous in $\{p_i\}_{i = 1}^M$ and in~$\{q_i\}_{i = 1}^M$, such that~\eqref{eq:q_mild_het} holds. Furthermore, assume that $\{p_i\}$ and $\{q_i\}$ are positively correlated, so that $p_1\leq p_2 \leq \dots \leq p_M$ and $q_1\leq q_2 \leq\dots \leq q_M$. Let $p =  \frac{1}{M} \sum_{i = 1}^M p_i$ and $q =  \frac{1}{M} \sum_{i = 1}^M q_i$. 
	Then 
	$$
	f^{\rm het}(t; \{p_i\}_{i=1}^M, \{q_i\}_{i=1}^M) < f^{\rm hom}(t; p,q,M), \qquad 0<t<\infty.
	$$
\end{theorem}
\begin{proof}
	Let $t_k^{\rm hom}$, $F^{\rm hom}_k(\tau)$, $t_k^{\rm het.}$, and $F^{\rm het.}_k(\tau)$ be defined as in the proof of Theorem~\ref{thm:compare_p}.
	In Lemma~\ref{lem:F^heter_k<F^hom_k_pq}, we prove the CDF dominance condition, $F^{\rm het}_1(\tau)\equiv F^{\rm hom}_1(\tau)$ for $\tau\geq 0$, and
	\begin{equation*}
		F^{\rm het}_k(\tau)<F^{\rm hom}_k(\tau), \qquad k = 2, \dots, M, \qquad 0<\tau<\infty.
	\end{equation*}
	The rest of the proof is the same as for Theorem~\ref{thm:compare_p}.	
\end{proof}

\Remark
We can prove Theorem~\ref{thm:compare_pq} for small times using the explicit expressions for $f'(0)$ and  $f''(0)$, see~\ref{app:small_time}.

\subsubsection{Non-positive correlation between $\{p_j\}$ and $\{q_j\}$.}

When $\{p_j\}$ and $\{q_j\}$ are {not positively correlated}, the effect of the heterogeneity can become more diverse, even for networks with $M=2$. In particular, {\em  $f^{\rm het}$ can be higher than $f^{\rm hom}$ for all times}:
\begin{lemma}
	\label{lem:compare_M=2pq}
	Let $B$ be a heterogeneous network with $M=2$, $\{p_1,p_2\}=\{2p,0\}$, and $\{q_1,q_2\}=\{0,2q\}$, and let
	$A$ be the corresponding homogeneous network with $M=2$, $p_1=p_2=p$, and $q_1=q_2=q$,  see Figure~\ref{fig:M=2pq_network}. Then for $0<t<\infty$,
	\begin{equation*}
		\begin{cases}
			f^{B}(t)>f^{A}(t), & \qquad {\rm if}\ q>p,\\
			f^{B}(t)=f^{A}(t), & \qquad {\rm if}\ q=p,\\
			f^{B}(t)<f^{A}(t), & \qquad {\rm if}\ q<p.
		\end{cases}
	\end{equation*}
\end{lemma}
\begin{proof}
	This follows from the master eqs.~\eqref{eqs:master_M=2} for $M=2$, see~\ref{app:lem:compare_M=2pq}.	
\end{proof}

\Remark
We can also prove Lemma~\ref{lem:compare_M=2pq} for small times using the explicit expressions for $f'(0)$ and $f''(0)$, see~\ref{app:small_time}.

\begin{figure}[ht!]
	\centering
	\begin{tikzpicture}
		\draw (-5,-1) rectangle (-.4, 2);
		\draw (5,-1) rectangle (.4,2);
		
		\draw (-4,0) circle(0.6cm) node {p};
		\draw (-1.4, 0) circle(0.6cm) node {p};
		
		\draw (1.4,0) circle(0.6cm) node {$p_1=2p$};
		\draw (4, 0) circle(0.6cm) node {$p_2=0$};
		
		\draw [->, thick] (-3.3, .25)--(-2.1,.25);
		\draw [->, thick] (-2.1,-.25)--(-3.3, -.25);
		
		\draw [->, thick] (2.05,0)--(3.35,0);
		
		\draw (-2.7,.5) node{q};
		\draw (-2.7,-.5) node{q};
		
		\draw (2.7,.5) node{2q};
		
		\draw (-4.3, 1.5) node{A};
		\draw (1.1,1.5) node{B};

	\end{tikzpicture}  
	\caption{Networks used in Lemma~\ref{lem:compare_M=2pq}. A)~Homogeneous. B)~Heterogeneous. }    
	\label{fig:M=2pq_network}                   
\end{figure}
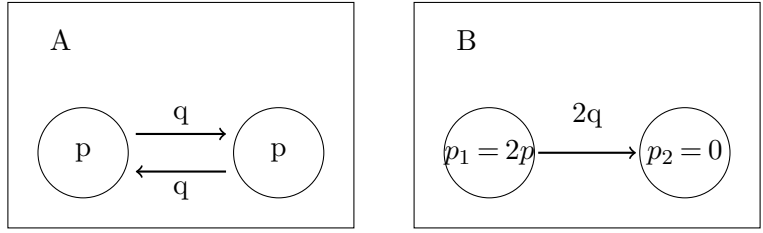

\begin{figure}[ht!]
	\centering
	\begin{minipage}{.3\textwidth}
		\centering
		\begin{tikzpicture}
			\node[anchor=south west, inner sep =0] at (0,0) {\includegraphics[width=\textwidth]{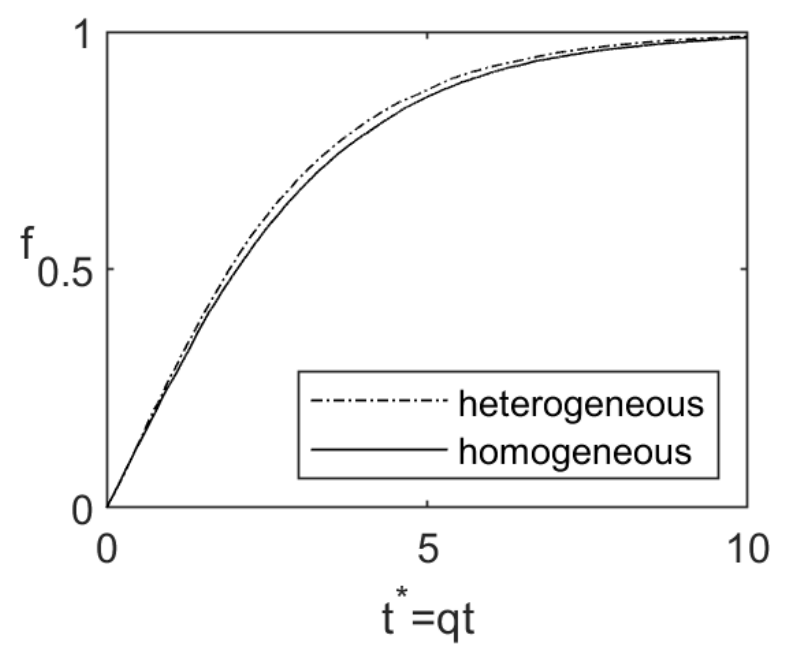}};
			\draw (4,2.5) node{A};
		\end{tikzpicture}
	\end{minipage}
	\begin{minipage}{.3\textwidth}
		\centering
		\begin{tikzpicture}
			\node[anchor=south west, inner sep =0] at (0,0) {\includegraphics[width=\textwidth]{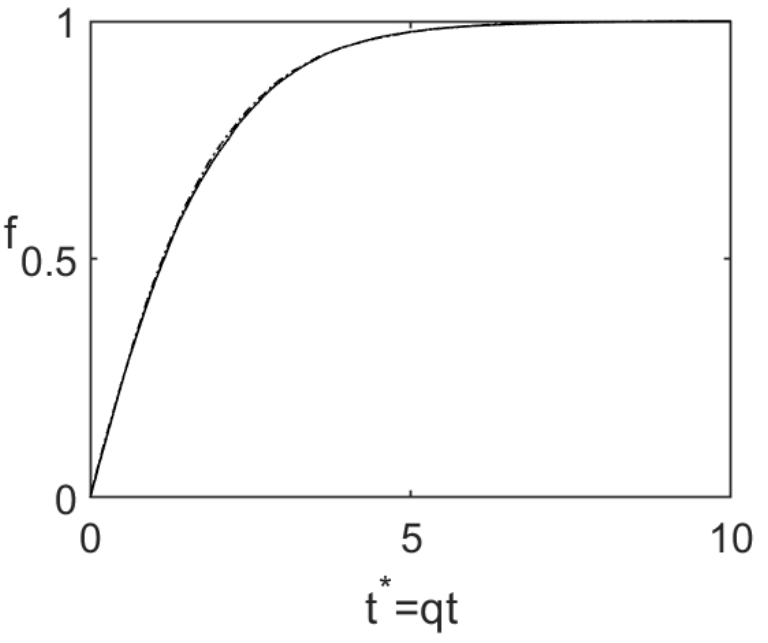}};
			\draw (4,2.5) node{B};
		\end{tikzpicture}
	\end{minipage}
	\begin{minipage}{.3\textwidth}
		\centering
		\begin{tikzpicture}
			\node[anchor=south west, inner sep =0] at (0,0) {\includegraphics[width=\textwidth]{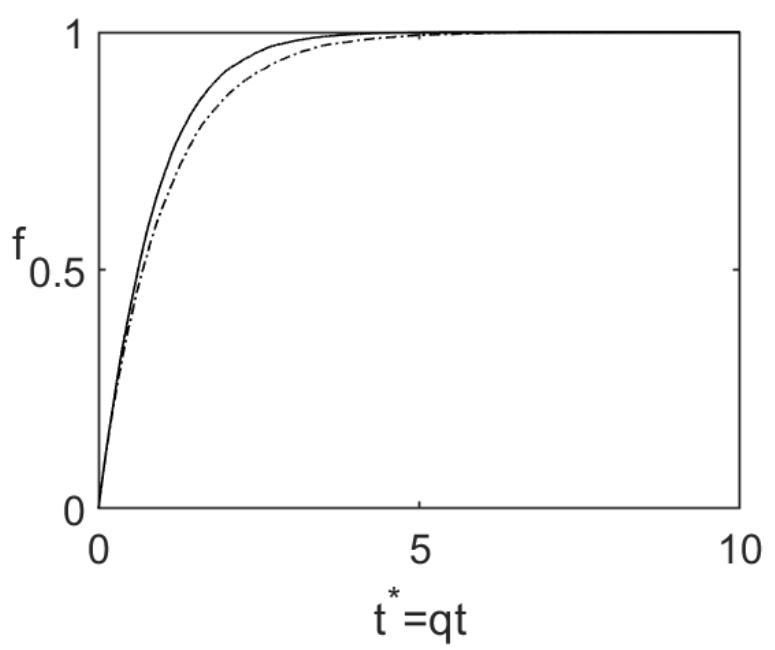}};
			\draw (4,2.5) node{C};
		\end{tikzpicture}
	\end{minipage}
	\caption{Fractional adoption in the homogeneous (solid) and heterogeneous (dash-dot) networks shown in Figure~\ref{fig:M=2pq_network}. Plot shows the average of $10^4$   simulations of~\eqref{eq:general_model}. A)~$p=\frac{q}{2}$. B)~$p=q$. C)~$p=2q$.} 
	\label{fig:M=2pq}
\end{figure}
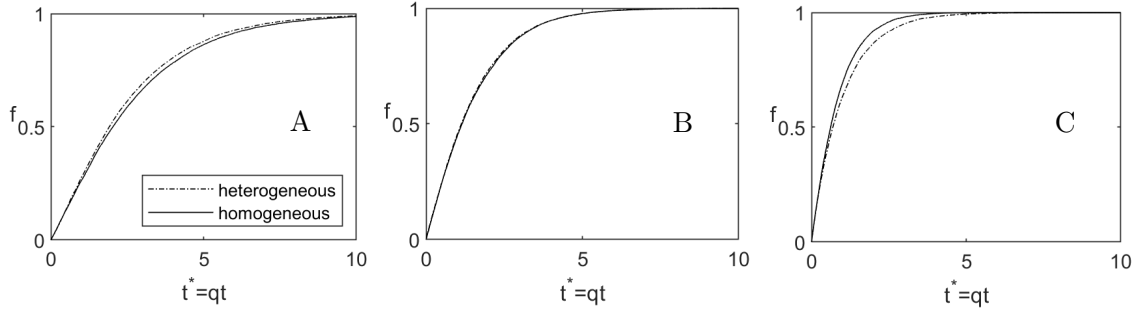

Lemma~\ref{lem:compare_M=2pq} is confirmed in Figure~\ref{fig:M=2pq} using   simulations of~\eqref{eq:general_model}.
The intuition behind this result is as follows. In both networks, the first adoption occurs at the rate of $2p$. The second adoption occurs at a rate of $2q$ in the heterogeneous case, and at the rate of $p+q$ in the homogeneous case. Therefore, by Theorem~\ref{thm:CDF-dominance},  $f^{\rm het}(t)$ is (globally in-time) greater than, equal to, or less than $f^{\rm hom}(t)$, if $q$ is greater than, equal to, or less than $p$, respectively. We can use this intuition 
to show that for any $M$, $f^{\rm het}(t)$ can be (globally in-time) greater than, equal to, or less than $f^{\rm hom}(t)$:
\begin{lemma}
	\label{lem:het>hom}
	Consider two complete networks with M nodes: a homogeneous network with $p_i=p$ and $q_{i,j}=\frac{q}{M-1}$ for all nodes, and a heterogeneous network with $p_1=Mp$, $p_2=p_3=\dots=p_M=0$, $q_{1,2}=q_{1,3}=\dots=q_{1,M}=\frac{2q}{M-1}$, $q_{2,1}=q_{3,1}=\dots=q_{M,1}=0$, and $q_{i,j}=\frac{q}{M-1}$ when $i\neq1$ and $j\neq1$,  see Figure~\ref{fig:M_general_pq_network}. Then for $0<t<\infty$,
	\begin{equation*}
		\begin{cases}
			f^{\rm het}(t)>f^{\rm hom}(t), & \qquad {\rm if}\ \frac{q}{M-1}>p,\\
			f^{\rm het}(t)=f^{\rm hom}(t), & \qquad {\rm if}\ \frac{q}{M-1}=p,\\
			f^{\rm het}(t)<f^{\rm hom}(t), & \qquad {\rm if}\ \frac{q}{M-1}<p.
		\end{cases}
	\end{equation*}
\end{lemma}
\begin{proof}
	The first adoption in both networks occurs at the rate of $Mp$. The second adoption occurs at a rate of $2q$ in the heterogeneous case and $(M-1)p+q$ in the homogeneous case. The $n$th adoption occurs at a rate of $n(M-n+1)\frac{q}{M-1}$ in the heterogeneous case and $(M-n+1)(p+(n-1)\frac{q}{M-1})$ in the homogeneous case. Hence, the result follows from Theorem~\ref{thm:CDF-dominance}.
\end{proof}






%
%
%
%
%
\begin{figure}
	\centering
	\scalebox{.7}[.7]{
		\centering
		\begin{minipage}{.6\textwidth}
			\centering
			\begin{tikzpicture}
				\draw (-6.2,-3.4) rectangle (-.8,2.7);
				\draw (-3.5,2) circle (0.5cm) node {p};
				\draw[<->,thick] (-3.5,1.5)--(-3.5,-2.3);
				\draw[<->,thick] (-4,1.8)--(-4.45,1.55);
				\draw[<->,thick] (-3.75,1.6)--(-4.8,-1.7);
				\draw[<->,thick] (-3.25,1.6)--(-2.2,-1.7);
				\draw (-5,1.3) circle (0.5cm) node {p};
				\draw[<->,thick] (-4.6,0.95)--(-2.4,-1.75);
				\draw[<->,thick] (-4.5,1.3)--(-2.5,1.3);
				\draw[<->,thick] (-5.3,0.8)--(-5.45,0.3);
				\draw[<->,thick] (-5,0.8)--(-5,-1.6);
				\draw[black,fill=black] (-5.5,0.1) circle (.3ex);
				\draw[black,fill=black] (-5.6,-0.4) circle (.3ex);
				\draw[black,fill=black] (-5.5,-0.9) circle (.3ex);
				\draw[<->,thick] (-5.45,-1.1)--(-5.3,-1.6);
				\draw (-5,-2.1) circle (0.5cm) node {p};
				\draw[<->,thick] (-4.45,-2.35)--(-4,-2.6);
				\draw (-3.5,-2.8) circle (0.5cm) node {p};
				\draw[<->,thick] (-3,-2.6)--(-2.55,-2.35);
				\draw[<->,thick] (-3.1,-2.45)--(-2.2,0.86);
				\draw[<->,thick] (-3.9,-2.45)--(-4.8,0.86);
				\draw (-2,-2.1) circle (0.5cm) node {p};
				\draw[<->,thick] (-4.5,-2.1)--(-2.5,-2.1);
				\draw (-2,1.3) circle (0.5cm) node {p};
				\draw[<->,thick] (-2.4,0.95)--(-4.6,-1.8);
				\draw[<->,thick] (-2,0.8)--(-2,-1.6);
				\draw[black,fill=black] (-1.5,0.1) circle (.3ex);
				\draw[black,fill=black] (-1.4,-0.4) circle (.3ex);
				\draw[black,fill=black] (-1.5,-0.9) circle (.3ex);
				\draw[<->,thick] (-1.7,-1.6)--(-1.55,-1.1);
				\draw[<->,thick] (-1.55,0.3)--(-1.7,0.8);
				\draw[<->,thick] (-2.55,1.55)--(-3,1.8);
				
				\draw (-5.5,2.2) node {A};
				
			\end{tikzpicture}
		\end{minipage}
		\begin{minipage}{.6\textwidth}
			\centering
			\begin{tikzpicture}
				\draw (-6.2,-3.4) rectangle (-.8,2.7);
				\draw (-3.5,2) circle (0.5cm) node {$Mp$};
				\draw[->,dashed] (-3.5,1.5)--(-3.5,-2.3);
				\draw[->,dashed] (-4,1.8)--(-4.45,1.55);
				\draw[->,dashed] (-3.75,1.6)--(-4.75,-1.65);
				\draw[->,dashed] (-3.25,1.6)--(-2.17,-1.62);
				\draw (-5,1.3) circle (0.5cm) node {$0$};
				\draw[<->,thick] (-4.6,0.95)--(-2.4,-1.75);
				\draw[<->,thick] (-4.5,1.3)--(-2.5,1.3);
				\draw[<->,thick] (-5.3,0.8)--(-5.45,0.3);
				\draw[<->,thick] (-5,0.8)--(-5,-1.6);
				\draw[black,fill=black] (-5.5,0.1) circle (.3ex);
				\draw[black,fill=black] (-5.6,-0.4) circle (.3ex);
				\draw[black,fill=black] (-5.5,-0.9) circle (.3ex);
				\draw[<->,thick] (-5.45,-1.1)--(-5.3,-1.6);
				\draw (-5,-2.1) circle (0.5cm) node {$0$};
				\draw[<->,thick] (-4.45,-2.35)--(-4,-2.6);
				\draw (-3.5,-2.8) circle (0.5cm) node {$0$};
				\draw[<->,thick] (-3,-2.6)--(-2.55,-2.35);
				\draw[<->,thick] (-3.1,-2.45)--(-2.2,0.86);
				\draw[<->,thick] (-3.9,-2.45)--(-4.8,0.86);
				\draw (-2,-2.1) circle (0.5cm) node {$0$};
				\draw[<->,thick] (-4.5,-2.1)--(-2.5,-2.1);
				\draw (-2,1.3) circle (0.5cm) node {$0$};
				\draw[<->,thick] (-2.4,0.95)--(-4.6,-1.8);
				\draw[<->,thick] (-2,0.8)--(-2,-1.6);
				\draw[black,fill=black] (-1.5,0.1) circle (.3ex);
				\draw[black,fill=black] (-1.4,-0.4) circle (.3ex);
				\draw[black,fill=black] (-1.5,-0.9) circle (.3ex);
				\draw[<->,thick] (-1.7,-1.6)--(-1.55,-1.1);
				\draw[<->,thick] (-1.55,0.3)--(-1.7,0.8);
				\draw[<-,dashed] (-2.55,1.55)--(-3,1.8);
				
				\draw (-5.5,2.2) node {B};
				
			\end{tikzpicture}
		\end{minipage}
	}
	\caption{Networks used in Lemma~\ref{lem:het>hom}. A solid edge indicates an internal influence of $\frac{q}{M-1}$, and a dashed edge indicates an internal influence of $\frac{2q}{M-1}$. A)~Homogeneous network. B)~Heterogeneous network.}
	\label{fig:M_general_pq_network}
\end{figure}
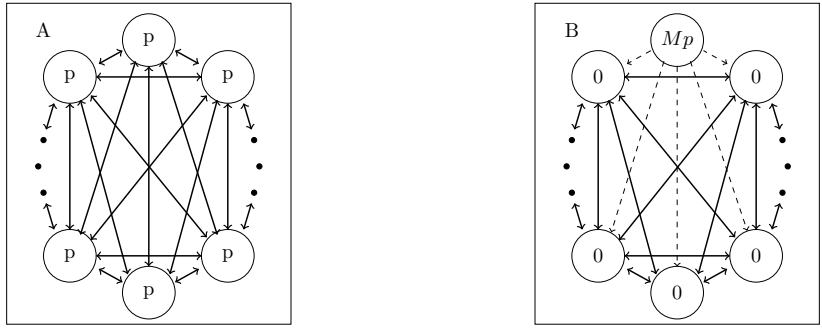

\section{Loss of global dominance under ``additive transformations''.}
\label{subsec:abc_lemma}

Let $A$ and $B$ be two networks with $M$ nodes such that the expected adoption in~$A$ is slower than in~$B$ for all times, i.e., 
$$
f_A(t)<f_B(t), \qquad 0< t<\infty.
$$
Consider the following two  ``additive transformations'' of these networks:
\begin{enumerate}
	\item[{\bf (T1):}]
	$(A,B) \longrightarrow (A',B')$, where $A'$ and $B'$ are obtained by adding $\varDelta p$ to all nodes in both networks, i.e., 
	\begin{equation}
		\label{eq:add_delta_p}
		p_j^{A'}=p_j^{A}+\varDelta p, \qquad p_j^{B'}=p_j^{B}+\varDelta p,\qquad j=1,\dots,M.
	\end{equation}
	\item[{\bf (T2):}]
	$(A,B) \longrightarrow (A'',B'')$, where $A''$ and $B''$  are obtained by adding an identical $(M+1)$ node to $A$ and $B$, so that 
	\begin{subequations}
		\label{eq:abc_condition}
		\begin{align}
			p_{M+1}^{A''}&=p_{M+1}^{B''}=p_{M+1},\\ q_{i,M+1}^{A''}&=q_{i,M+1}^{B''}=q_{M+1}^{\rm in}, \qquad q_{M+1,i}^{A''}=q_{M+1,i}^{B''}=q_{M+1}^{\rm out},\qquad i=1,\dots,M.
		\end{align}
	\end{subequations}
\end{enumerate}

It is natural to ask:\footnote{One motivation for asking this question is given in Section~\ref{subsec:complete_p}.}
\par
\Question {\it Is the global dominance between  two networks  preserved under the  ``additive'' transformations~\eqref{eq:add_delta_p} and~\eqref{eq:abc_condition}, i.e., is it true that  $f_{A'}(t)<f_{B'}(t)$ and $f_{A''}(t)<f_{B''}(t)$ for  $0< t<\infty$? }

In some cases, global dominance is indeed preserved:
\begin{lemma}
	Let $A \prec B$ (see Definition~\ref{def:dominance}). Then
	$$
	f_{A}(t) < f_{B}(t), \qquad 0< t<\infty,
	$$
	and for any $\Delta p$, $p_{M+1}$, $q_{M+1}^{\rm in}$, and $q_{M+1}^{\rm out}$,  
	$$
	f_{A'}(t)<f_{B'}(t), \quad 
	f_{A''}(t)<f_{B''}(t), \qquad 0< t<\infty.
	$$	
\end{lemma}
\begin{proof}
	if  $A \prec B$. Then  $A' \prec B'$ and  $A'' \prec B''$. Hence, the result follows 
	from the {\em dominance principle}  (Lemma~\ref{lem:dominance-principle}).	
\end{proof}

This lemma may seem to suggest that global dominance is indeed preserved under the  ``additive'' transformations~\eqref{eq:add_delta_p} and~\eqref{eq:abc_condition}.
This, however, is not always the case. Indeed, global dominance can be lost under a uniform addition of $\Delta p$ to all nodes {\bf (T1)}, {or under the addition of an identical node {\bf (T2)}}: 
\begin{lemma}
	\label{lem:ABC_counter_2}
	There exist two networks $A$ and $B$ of size $M$ such that $f_A(t)<f_B(t)$ for $0< t<\infty$, but
	\begin{subequations}
		\label{eqs:flips}
		\begin{equation}
			\label{eq:flips2}
			f^{A'}(t)> f^{B'}(t), \qquad  t \gg 1,
		\end{equation}  
		and
		\begin{equation}
			\label{eq:flips_t2}
			f^{A''}(t)>f^{B''}(t), \qquad  t \gg 1.
		\end{equation}
	\end{subequations}
\end{lemma}
\begin{proof}
	Let $B$ be a heterogeneous network with $M=2$, $\{p_1,p_2\}=\{2p,0\}$, and $\{q_1,q_2\}=\{0,2q\}$, let
	$A$ be the corresponding homogeneous network with $p_1=p_2=p$, and $q_1=q_2=q$.
	and let $p<q$. Then by Lemma~\ref{lem:compare_M=2pq}, $f^A(t)<f^B(t)$ for $0<t<\infty$,
	see Figure~\ref{fig:ABC_counter}A.

	
	Let $A'$ and $B'$ 
	be the networks obtained from~$A$ and~$B$ 
	when we increase all the $\{p_j\}$s by $\varDelta p$, see~\eqref{eq:add_delta_p}.
	By~\eqref{eqs:fc_M=2},
	$		f^{A'}(t)=1-\frac{p+\varDelta p}{p-q+\varDelta p}e^{-(p+\varDelta p+q)t}+\frac{q}{p-q+\varDelta p}e^{-(2p+2\varDelta p)t},
	$
	and
	$
	f^{B'}=1-\frac{1}{2}\bigg[e^{-(2p+\varDelta p)t}+\frac{2p+\varDelta p}{2p-2q+\varDelta p}e^{-(2q+\varDelta p)t}-\frac{2q}{2p-2q+\varDelta p}e^{-(2p+2\varDelta p)t}\bigg].
	$
	Let $\Delta p>q-p$. Then for $t\gg 1$,
	$f^{A'}(t)\approx
	1-\frac{p+\varDelta p}{p-q+\varDelta p}e^{-(p+\varDelta p+q)t}$ and $f^{B'}\approx1-\frac{1}{2}e^{-(2p+\varDelta p)t}$.
	Therefore, since $q>p$,  we have~\eqref{eq:flips2}
	
	{\newrev If we add an identical $(M+1)$st node to $A$ and $B$ with $p_{M+1}=\infty$ and $q_{M+1}^{{\rm out}}\equiv \varDelta p$, then $A''$ and $B''$ are equivalent to $A'$ and $B'$, and so~\eqref{eq:flips_t2} holds. By continuity, the result also holds for finite but sufficiently large values of $p_{M+1}$.}	
\end{proof}

The above calculations show that the dominance between a homogeneous and heterogeneous network is not always global in time, but rather can change with time: 
\begin{corollary}
	\label{lem:local_dominance}
	
	There exists a heterogeneous network and a corresponding homogeneous network for which ${f^{\rm het}(t)-f^{\rm hom}(t)}$ changes its sign in $0<t<\infty$.
\end{corollary}
\begin{proof}
	By~\eqref{eq:n'(0)} and~\eqref{eq:initial_hom},
	$\left(f^{A'}\right)'(0)=\left(f^{B'}\right)'(0)=p+\Delta p$, 
	and
	$\left(f^{A'} \right)''(0)-\left(f^{B'} \right)''(0)  = p(p-q)<0$. Therefore, 
	\begin{equation}
		\label{eq:flips1}
		f^{A'}(t)<f^{B'}(t), \qquad  0<t\ll 1.
	\end{equation}
	The result follows from~\eqref{eq:flips2}~and~\eqref{eq:flips1}.	
\end{proof}
Thus, the flip of the dominance as 
$(A,B) \to  (A',B')$ 
occurs for $t\gg 1$, see~\eqref{eq:flips2}, but not for $ t \ll 1$, see~\eqref{eq:flips1}, as is illustrated in Figure~\ref{fig:ABC_counter}. {\rev Indeed, the uniform increase of $\{p_i\}$ by $\Delta p$ does not change the dominance between the heterogeneous and homogeneous networks during the initial dynamics, because by Lemma~\ref{lem:n_initial_dynamics}, the initial dynamics are determined by the mean and variance of $\{p_i\}$, each of which is equally affected by a uniform shift by $\Delta p$ in the heterogeneous and homogeneous cases~(\ref{app:small_time}). This uniform increase, however, can affect the dominance later on. Indeed, as $\Delta p\to\infty$, $A'$ and $B'$ become homogeneous in $p$. However, $B'$ is heterogeneous in $q$, and so its diffusion is slower than in the homogeneous case $A'$ (Theorem~\ref{thm:compare_q}).}

\begin{figure}[ht!]
	\centering
	\begin{minipage}{.3\textwidth}
		\centering
		\begin{tikzpicture}
			\node[anchor=south west, inner sep =0] at (0,0) {\includegraphics[width=\textwidth]{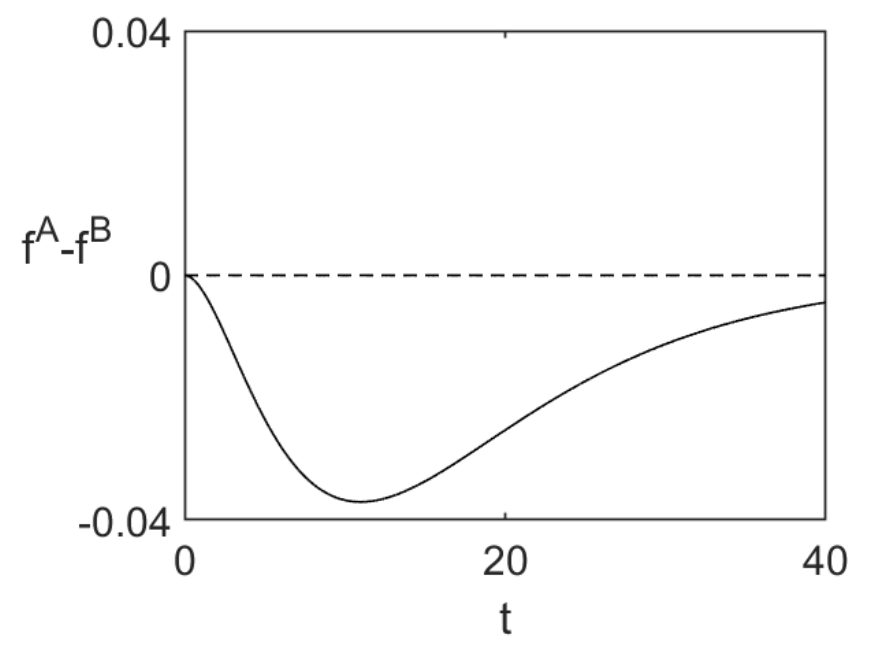}};
			\draw (4,3) node{A};
		\end{tikzpicture}
	\end{minipage}
	\begin{minipage}{.32\textwidth}
		\centering
		\begin{tikzpicture}
			\node[anchor=south west, inner sep =0] at (0,0) {\includegraphics[width=\textwidth]{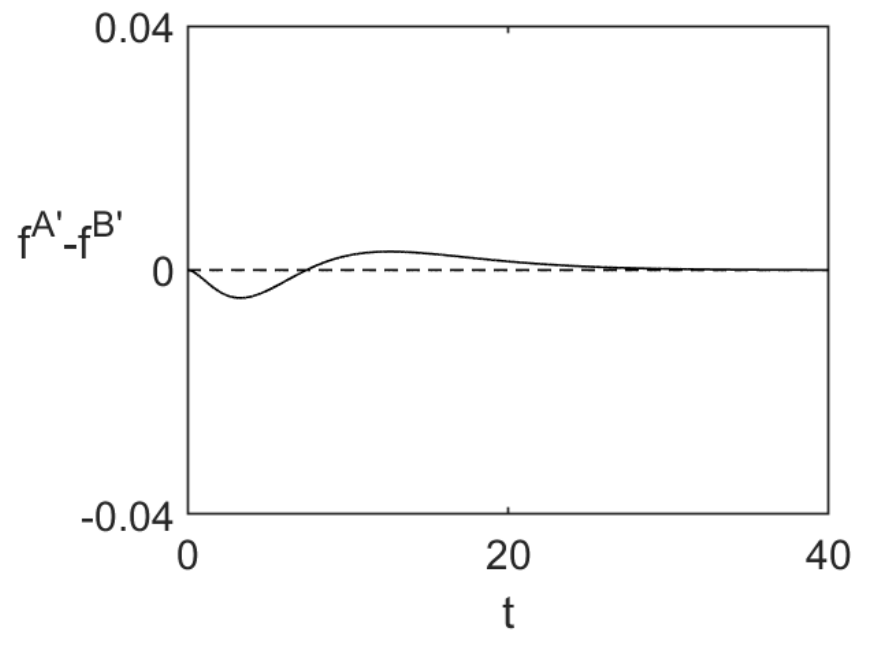}};
			\draw (4,3) node{B};
		\end{tikzpicture}
		
	\end{minipage}
	\caption{A)~Difference between $f^A$ and $f^B$ from Lemma~\ref{lem:compare_M=2pq}. B)~Difference between $f^{A'}$ and $f^{B'}$ from Lemma~\ref{lem:ABC_counter_2}. Here $p=0.05$, $q=0.15$, and $\varDelta p=0.15$. 
	}
	\label{fig:ABC_counter}
\end{figure}

\section{Heterogeneity in 1D networks.}
\label{sec:1D}
In Section~\ref{sec:complete}, we analyzed the diffusion in complete networks. We now consider the opposite type of networks, namely, structured sparse networks  where each node is only connected to one or two nodes.
\subsection{One-sided circle.}
\label{subsec:one_side_circle}
Assume that M consumers are located on a one-sided circle, such that each node can only be influenced by its left neighbor (Fig.~\ref{fig:structure}A). Thus, if $(j-i) \Mod M\neq 1$, then $q_{i,j}=0$. In this case,~\eqref{eq:general_model} reads
\begin{equation}
	\label{eq:one_sided_model}
	{\rm Prob}(X_j(t+dt)=1)=\begin{cases}
		1,&\qquad {\rm if}\ X_j(t)=1,\\
		\left(p_j+q_jX_{j-1}(t)\right)dt, &\qquad {\rm if}\ X_j(t)=0,
	\end{cases}
\end{equation}
where $q_j=q_{j-1,j}$, see~\eqref{eq:q_on_node}.
Let $(S_k^j)(t):=(S^{j-k+1},S^{j-k+2},\dots,S^j)(t)$ denote the event that the chain of $k$ nodes that ends at node j are all non-adopters at time $t$, see Fig.~\ref{fig:structure}A, and let $[S_k^j](t)$ denote the probability of that event. We proceed to derive the master equations for $[S_k^j](t)$:

\begin{lemma}
	\label{lem:masterHetero}
	Consider the heterogeneous discrete Bass model~\eqref{eq:one_sided_model} on a one-sided circle. For any $j$, $1\leq j \leq M$, the $M$ master equations for $\{[S_k^j]\}_{k=1}^M$ are\footnote{\label{foot:q_{i,j}} Here we use the conventions of Remark~\ref{rem:q_{i,j}}.}
	\begin{subequations}
		\label{eqs:masterHeteros}
		\begin{equation}
			\label{eq:masterHetero}
			\frac{d}{dt}[S_k^j](t)=-\left(\left( \sum_{i=j-k+1}^{j}p_i\right)+ q_{j-k+1}\right) [S_k^j](t)+q_{j-k+1}[S_{k+1}^j](t), \qquad  k=1,\dots,M-1,
		\end{equation} 
		and\footnote{Note that $(S_M^j)(t)=(S_M)(t)$ is independent of $j$.}
		\begin{equation}
			\label{eq:masterHeteroM}
			\frac{d}{dt}[S_M^j](t)=\left(-\sum_{i=1}^{M}p_i\right)[S_M^j](t),
		\end{equation}
		subject to the initial conditions
		\begin{equation}
			\label{eq:masterHeteroMInitial}
			[S_k^j](0)=1, \qquad k=1,\dots,M.
		\end{equation}
	\end{subequations}
\end{lemma}
\begin{proof}
	See~\ref{app:lem:masterHetero}.	
\end{proof}
Equations~\eqref{eqs:masterHeteros} for $\{[S_k^j](t)\}_{k=1}^M$ are decoupled from those for $\{[S_k^i](t)\}_{k=1}^M$ for $i\neq j$. This allows us to solve them explicitly for any $M$, and thus to obtain the fractional adoption on a one-sided heterogeneous circle:
\begin{theorem}
	\label{thm:one_sided}
	Consider the heterogeneous discrete Bass model~\eqref{eq:one_sided_model} on the one-sided circle. Then 
	\begin{equation}
		\label{eq:expectedHeteroOneSidedCircle}
		f_{\rm circle}^{\rm 1-sided}(t)=1-\frac{1}{M}\sum\limits_{j=1}^{M}[S_1^j](t),
		\qquad 
		[S_1^j](t)=\sum_{k=1}^{M}c_k^j\boldsymbol{v}_k^j(1)e^{\lambda_k^jt},
	\end{equation}
	where$^{\ref{foot:q_{i,j}}}$
	\begin{subequations}
		\label{eqs:one_sided_eigenvectors}
		\begin{equation}
			\label{eq:one_sided_eigenvalues}
			\lambda_k^j = \begin{cases}
				\left(-\sum_{i=j-k+1}^{j}p_i\right) -q_{j-k+1}, &\qquad k=1,\dots,M-1,\\
				-\sum_{i=1}^{M}p_i, &\qquad k=M,
			\end{cases}
		\end{equation}
		\begin{equation}
			\boldsymbol{v}_k^j(1)=
			\begin{cases}
				1, &\qquad k=1,\\
				\prod\limits_{m=j-k+2}^{j}\frac{-q_m}{\left(\sum\limits_{i=j-k+1}^{m-1}p_i\right) +q_{j-k+1}-q_m}, &\qquad k=2,\dots,M-1,\\
				\prod\limits_{m=j-M+2}^{j}\frac{-q_m}{\left(\sum\limits_{i=j-M+1}^{m-1}p_i\right) -q_m}, &\qquad k=M,
			\end{cases}
		\end{equation}
		$c_M^j=1$, and for  $k=M-1,\dots,1$,
		\begin{equation}
			\label{eq:one_sided_constants}
			\begin{aligned}
				c_k^j=1-&\prod\limits_{m=j+1-(M-1)}^{j+1-k}\frac{-q_m}
				{\left(\sum\limits_{i=j-M+1}^{m-1}p_i\right)-q_m}
				-\sum_{l=k+1}^{M-1}\left(\prod\limits_{m=j-l+2}^{j-k+1}\frac{-q_m}
				{\left(\sum\limits_{i=j-l+1}^{m-1}p_i\right)+q_{j-l+1}-q_m}\right)c_l^j. 
			\end{aligned}
		\end{equation}
	\end{subequations}
\end{theorem}
\begin{proof}
	See~\ref{app:thm:one_sided}.	
\end{proof}
As expected, when M=2, \eqref{eq:expectedHeteroOneSidedCircle}-\eqref{eqs:one_sided_eigenvectors}
%
reduced to~\eqref{eqs:fc_M=2}, and 
when $M=3$,
it reduces to~\eqref{eqs:fc_M=3} with $q_{j+1,j}=q_{j+2,j+1}=0$.

\subsection{Two-sided circle.}
\label{subsec:two_sided_circle}

Consider now a circle with M nodes where each node can be influenced by its left and right neighbors. Thus, if $|j-i| \Mod M\neq 1$, then $q_{i,j}=0$ (figure~\ref{fig:structure}B). In this case,~\eqref{eq:general_model} reads
\begin{equation}
	\label{eq:two_sided_model}
	{\rm Prob}(X_j(t+dt)=1)=\begin{cases}
		1,&\qquad {\rm if}\ X_j(t)=1,\\
		\left[p_j+q_{j-1,j}X_{j-1}(t)+q_{j+1,j}X_{j+1}(t)\right]dt, &\qquad {\rm if}\ X_j(t)=0.
	\end{cases}
\end{equation}
Let $(S_{m,n}^j)(t):=(S^{j-m},\dots,S^j,\dots,S^{j+n})(t)$ denote the event that the  $m+n+1$ nodes $\{j-m,\dots,j,\dots,j+n\}$ are all non-adopters at time t, see Figure~\ref{fig:structure}B, and let $[S_{m,n}^j](t;M)$ denote the probability of that event. We proceed to derive the master equations for $[S^j]=[S_{0,0}^j]$:

\begin{lemma}
	\label{lem:masterHetero2}
	Consider the heterogeneous discrete Bass model~\eqref{eq:two_sided_model} on the two-sided circle. For any $1\leq j\leq M$ and any  $0\leq m+n\leq M-2$,  the master equations for $\{[S_{m,n}^j]\}$ are\footnote{Here we also use the conventions of Remark~\ref{rem:q_{i,j}}.}
	\begin{subequations}
		\label{eqs:masterHetero2}
		
		\begin{equation}
			\label{eq:masterHetero2}
			\begin{aligned}
				\frac{d}{dt}[S_{m,n}^j](t)=&-\Bigg( \Bigg(\sum_{i=j-m}^{j+n}p_i\Bigg)+ q_{j-m-1,j-m} + q_{j+n+1,j+n}\Bigg) [S_{m,n}^j](t)+q_{j+n+1,j+n}[S_{m,n+1}^j](t)
				\\& \quad +q_{j-m-1,j-m}[S_{m+1,n}](t), 
			\end{aligned}
		\end{equation}
		and
		\begin{equation}
			\label{eq:masterHetero2M}
			\frac{d}{dt}[S_M](t)=\left(-\sum_{i=1}^{M}p_i\right)[S_M](t),
		\end{equation} 
		subject to the initial condition
		\begin{equation}
			\label{eq:masterHetero2Initial}
			[S_{m,n}^j](0)=1, \qquad 0\leq m+n\leq M-1,
		\end{equation}
	\end{subequations}
	where $[S_M](t):=[S_{M-1,0}^j](t)=[S_{M-2,1}^j](t)=\dots=[S_{0,M-1}^j](t)$.
\end{lemma}
\begin{proof}
	See~\ref{app:lem:masterHetero2}.	
\end{proof}

\Remark As in the one-sided case, the $\frac{M(M-1)}{2}+1$ equations for $\{[S_{m,n}^j](t)\}_{m,n}$ are decoupled from those for $\{[S_{m,n}^i](t)\}_{m,n}$ for $i\neq j$.

\begin{theorem}
	\label{thm:two_sided}
	Consider the heterogeneous discrete Bass model~\eqref{eq:two_sided_model} on the two-sided circle. Then
	\begin{equation}
		\label{eq:expectedHeteroTwoSidedCircle}
		f_{\rm circle}^{\rm 2-sided}(t)=1-\frac{1}{M}\sum\limits_{j=1}^{M}[S_{0,0}^j](t),
	\end{equation}
	where $\{[S_{0,0}^j](t)\}_{j=1}^M$ can be determined from~\eqref{eqs:masterHetero2}.
\end{theorem}

Unlike the one-sided case, we have not found a way to explicitly solve for $[S_{0,0}^j](t)$ for a general $M$ (see~\ref{app:S_{0,0}^j}).
We did, however, obtain explicit solutions of~\eqref{eqs:masterHetero2} for $M=2$ and $M=3$~(\ref{app:S_{0,0}^j}). As expected, the resulting expressions for $f_{\rm circle}^{\rm 2-sided}$ identify with~\eqref{eqs:fc_M=2} and~\eqref{eqs:fc_M=3}, respectively.

\subsection{Comparison of one-sided and two-sided circles.}
\label{subsec:circle_compare}

Recall that on homogeneous circles, one-sided and two-sided diffusion are identical if $q=q^R+q^L$, see~\eqref{eqs:compare_Fibich_Gibori}. To extend this condition to the heterogeneous case,
we interpret it as saying that  the sum of the incoming~$q_{k,j}$ into each node is identical in both networks, i.e., 
%
%
%
\begin{equation}
	\label{eq:convert_two_to_one}
	q_j^{\rm 1-sided}=q_{j-1,j}^{\rm 2-sided}+q_{j+1,j}^{\rm 2-sided}, \qquad j=1, \dots, M.
\end{equation}
In light of~\eqref{eqs:compare_Fibich_Gibori}, is it true that in the heterogeneous case $f_{\rm circle}^{{\rm 1-sided}}(t; p,\{q_j\})\equiv f_{\rm circle}^{\rm 2-sided}(t;,p,\{q_{j\pm 1,j}\})$ when~\eqref{eq:convert_two_to_one} holds?

When $M=2$, heterogeneous one-sided and two-sided circles are by definition identical, and so $f_{\rm circle}^{{\rm 1-sided}}\equiv f_{\rm circle}^{\rm 2-sided}$. 
When $M>2$ however, this is no longer the case:
\begin{lemma}
	\label{lem:one_vs_two_sided}
	Consider a one-sided and a two-sided heterogeneous circles with $M\geq 3$ nodes, for which~\eqref{eq:convert_two_to_one} holds. Then $f_{{\rm circle}}^{\rm 1-sided}(t)$ can be (globally in-time) larger or smaller than $f_{{\rm circle}}^{\rm 2-sided}(t)$.
\end{lemma}
\begin{proof}
	Consider the following $M=3$ circles with $p_1=p$ and $p_2=p_3=0$.
	\begin{enumerate}[label=(\arabic*)]
		\item 
		Let $q_{1,3} = q_{3,1} = q$ and $q_{2,3}=q_{3,2} = q_{2,1} = q_{1,2} =0$ (Figure~\ref{subfig:network_1}). By~\eqref{eq:convert_two_to_one}, in the one-sided case $q_{3}=q_{1}=q$ and $q_{2}=0$. Hence, node 1 adopts at the same rate in both networks, node 2 adopts in neither case, and node 3 only adopts in the two-sided circle. Therefore, $f_{{\rm circle}}^{\rm 2-sided}(t)$ is larger than $f_{{\rm circle}}^{\rm 1-sided}(t)$ for all $t$.
		\item 
		Let $q_{2,3} = q_{3,2} = q$ and $q_{1,3}=q_{3,1} = q_{2,1} = q_{1,2} =0$ (Figure~\ref{subfig:network_2}). By~\eqref{eq:convert_two_to_one}, in the one-sided case $q_{3}=q_{2}=q$ and $q_{1}=0$. Hence, node 1 adopts at the same rate in both cases, but nodes 2 and 3 only adopt in the one-sided case. Therefore $f_{{\rm circle}}^{\rm 1-sided}(t)$ is larger than $f_{{\rm circle}}^{\rm 2-sided}(t)$ for all $t$.
	\end{enumerate}	
\end{proof}

{\rev Intuitively, the overall impact of an edge depends not only on its own weight, but also on the node that it originates from. Thus, generally speaking, an edge that originates from a node with $p_i=0$ has a weaker effect than an equal-weight edge that originates from a node with $p_i>0$. Hence, one can utilize this insight to construct networks for which $f_{{\rm circle}}^{\rm 1-sided}(t)$ is higher or lower than $f_{{\rm circle}}^{\rm 2-sided}(t)$.}
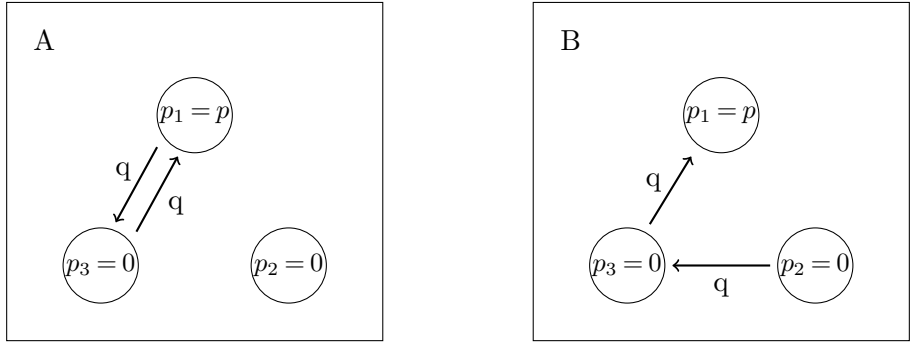
\begin{figure}[h]
	\centering
	\begin{tikzpicture}
		\draw (-6,-1) rectangle (-1,3.5);
		\draw (1,3.5) rectangle (6,-1);
		\draw (-3.5,2) circle (0.5cm) node {\small $p_1=p$};
		\draw[->, thick] (-4.0,1.575)--(-4.55,0.575);
		\draw[->, thick] (-4.275,0.45)--(-3.725,1.45);
		\draw (-4.75,0) circle (0.5cm) node {\small $p_3=0$};
		\draw (-5.5,3) node {A};
		\draw (1.5,3) node {B};
		
		\draw (-2.25,0) circle (0.5cm) node {\small $p_2=0$};
		
		\draw (-4.45,1.25) node {q};
		\draw (-3.75,0.8) node {q};
		
		\draw (3.5,2) circle (0.5cm) node {\small $p_1=p$};
		\draw[->, thick] (4.15,0)--(2.85,0);
		\draw[->, thick] (2.55,0.55)--(3.1,1.45);
		
		\draw (4.75,0) circle (0.5cm) node {\small $p_2=0$};
		\draw (2.25,0) circle (0.5cm) node {\small $p_3=0$};
		
		\draw (3.5,-0.3) node {q};
		\draw (2.6,1.1) node {q};
		
	\end{tikzpicture}
	\caption{Heterogeneous networks that satisfy~\eqref{eq:convert_two_to_one} for which $f_{\rm circle}^{{\rm 2-sided}}(t)>f_{\rm circle}^{\rm 1-sided}(t)$. A)~Two-sided network. B)~One-sided network.}
	\label{subfig:network_1}
\end{figure}

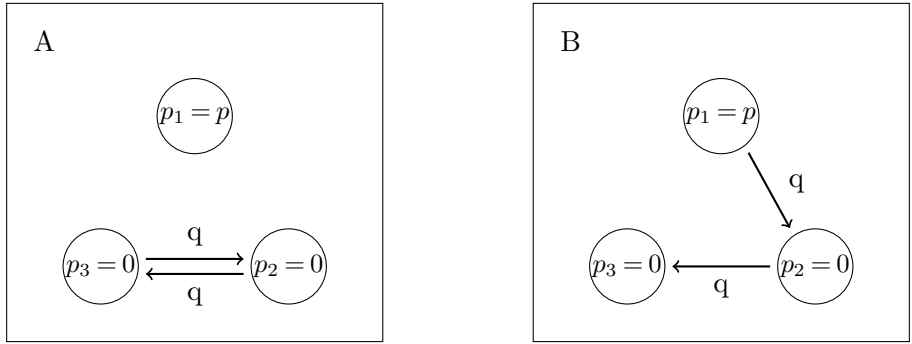
\begin{figure}[h]
	\centering
	\begin{tikzpicture}
		\draw (-6,-1) rectangle (-1,3.5);
		\draw (1,3.5) rectangle (6,-1);
		\draw (-3.5,2) circle (0.5cm) node {\small $p_1=p$};
		\draw[->, thick] (-2.85,-0.1)--(-4.15,-0.1);
		\draw[->, thick] (-4.15,0.1)--(-2.85,0.1);
		\draw (-4.75,0) circle (0.5cm) node {\small $p_3=0$};
		\draw (-5.5,3) node {A};
		\draw (1.5,3) node {B};
		
		\draw (-2.25,0) circle (0.5cm) node {\small $p_2=0$};
		\draw (-3.5,-0.4) node {q};
		\draw (-3.5,0.4) node {q};
		
		\draw (3.5,2) circle (0.5cm) node {\small $p_1=p$};
		\draw[->, thick] (4.15,0)--(2.85,0);
		\draw[->, thick] (3.8625,1.5125)--(4.4125,0.5125);
		
		\draw (4.75,0) circle (0.5cm) node {\small $p_2=0$};
		\draw (2.25,0) circle (0.5cm) node {\small $p_3=0$};
		
		\draw (3.5,-0.3) node {q};
		\draw (4.5,1.1) node {q};
		
	\end{tikzpicture}
	\caption{Same as Figure~\ref{subfig:network_1} with heterogeneous networks for which $f_{\rm circle}^{{\rm 2-sided}}(t)<f_{\rm circle}^{\rm 1-sided}(t)$.} 
	\label{subfig:network_2}
\end{figure}


\section{D-dimensional homogeneous Cartesian networks.}
\label{sec:cart}

Let $f_D(t;p,q)$ denote the fraction of adopters on the infinite $D$-dimensional Cartesian {\em  homogeneous} network, where nodes are labeled by their $D$-dimensional coordinate vector ${\bf i}=(i_1,\dots,i_D)$,
\begin{equation}
	\label{eq:cart_hom}
	p_{\bf i}\equiv p,\qquad q_{{\bf i},{\bf j}}=\begin{cases}
		\frac{q}{2D},&\qquad {\rm if}\ \|{\bf i}-{\bf j}\|_1=1,\\
		0, &\qquad {\rm otherwise,}
	\end{cases}\qquad {\bf i},{\bf j}\in \mathbb{Z}^{D},
\end{equation}
and $\|{\bf i}-{\bf j}\|_1:=\sum_{k=1}^{D}|i_k-j_k|$. Thus, each node can be influenced by its $2D$ nearest neighbors at the rate of~$\frac{q}{2D}$.  See,  e.g., Fig.~\ref{fig:2D_grid} for $M=2$.

\begin{figure} [ht!]
	\centering
	\scalebox{.6}{
		\begin{tikzpicture}
			\draw (-4,-4) rectangle (4,4);
			\draw[black,fill=black] (-3.7,0) circle (.3ex);
			\draw[black,fill=black] (-3.4,0) circle (.3ex);
			\draw[black,fill=black] (-3.1,0) circle (.3ex);
			\draw[black,fill=black] (-3.7,-1.7) circle (.3ex);
			\draw[black,fill=black] (-3.4,-1.7) circle (.3ex);
			\draw[black,fill=black] (-3.1,-1.7) circle (.3ex);
			\draw[black,fill=black] (-3.7,1.7) circle (.3ex);
			\draw[black,fill=black] (-3.4,1.7) circle (.3ex);
			\draw[black,fill=black] (-3.1,1.7) circle (.3ex);
			\draw[black,fill=black] (3.7,-1.7) circle (.3ex);
			\draw[black,fill=black] (3.4,-1.7) circle (.3ex);
			\draw[black,fill=black] (3.1,-1.7) circle (.3ex);
			\draw[black,fill=black] (3.7,1.7) circle (.3ex);
			\draw[black,fill=black] (3.4,1.7) circle (.3ex);
			\draw[black,fill=black] (3.1,1.7) circle (.3ex);
			\draw[black,fill=black] (3.7,1.7) circle (.3ex);
			\draw[black,fill=black] (3.4,1.7) circle (.3ex);
			\draw[black,fill=black] (3.1,1.7) circle (.3ex);
			\draw[black,fill=black] (0,3.1) circle (.3ex);
			\draw[black,fill=black] (0,3.4) circle (.3ex);
			\draw[black,fill=black] (0,3.7) circle (.3ex);
			\draw[black,fill=black] (0,-3.1) circle (.3ex);
			\draw[black,fill=black] (0,-3.4) circle (.3ex);
			\draw[black,fill=black] (0,-3.7) circle (.3ex);
			\draw[black,fill=black] (-1.7,3.1) circle (.3ex);
			\draw[black,fill=black] (-1.7,3.4) circle (.3ex);
			\draw[black,fill=black] (-1.7,3.7) circle (.3ex);
			\draw[black,fill=black] (1.7,3.1) circle (.3ex);
			\draw[black,fill=black] (1.7,3.4) circle (.3ex);
			\draw[black,fill=black] (1.7,3.7) circle (.3ex);
			\draw[black,fill=black] (1.7,-3.1) circle (.3ex);
			\draw[black,fill=black] (1.7,-3.4) circle (.3ex);
			\draw[black,fill=black] (1.7,-3.7) circle (.3ex);
			\draw[black,fill=black] (-1.7,-3.1) circle (.3ex);
			\draw[black,fill=black] (-1.7,-3.4) circle (.3ex);
			\draw[black,fill=black] (-1.7,-3.7) circle (.3ex);
			\draw[<->, thick] (-2.2,0)--(-2.9,0);
			\draw[<->, thick] (-2.9,0)--(-2.2,0);
			\draw[<->, thick](-2.2,1.7)--(-2.9,1.7);
			\draw[<->, thick] (-2.9,1.7)--(-2.2,1.7);
			\draw[<->, thick](-2.2,-1.7)--(-2.9,-1.7);
			\draw[<->, thick] (-2.9,-1.7)--(-2.2,-1.7);
			\draw[black] (-1.7,0) circle (0.5cm) node {p};
			\draw[<->, thick](-1.2,0)--(-0.5,0);
			\draw[<->, thick] (-0.5,0)--(-1.2,0);
			\draw[<->, thick] (-1.2,1.7)--(-0.5,1.7);
			\draw[<->, thick] (-0.5,1.7)--(-1.2,1.7);
			\draw[<->, thick] (-1.2,-1.7)--(-0.5,-1.7);
			\draw[<->, thick] (-0.5,-1.7)--(-1.2,-1.7);
			\draw[black] (0,0) circle (0.5cm) node {p};
			\draw[<->, thick] (0.5,0)--(1.2,0);
			\draw[<->, thick] (1.2,0)--(0.5,0);
			\draw[<->, thick] (0.5,1.7)--(1.2,1.7);
			\draw[<->, thick] (1.2,1.7)--(0.5,1.7);
			\draw[<->, thick] (0.5,-1.7)--(1.2,-1.7);
			\draw[<->, thick] (1.2,-1.7)--(0.5,-1.7);
			\draw[black] (1.7,0) circle (0.5cm) node {p};
			\draw[<->, thick] (2.2,0)--(2.9,0);
			\draw[<->, thick] (2.9,0)--(2.2,0);
			\draw[<->, thick] (2.2,1.7)--(2.9,1.7);
			\draw[<->, thick] (2.9,1.7)--(2.2,1.7);
			\draw[<->, thick] (2.2,-1.7)--(2.9,-1.7);
			\draw[<->, thick] (2.9,-1.7)--(2.2,-1.7);
			\draw[black,fill=black] (3.7,0) circle (.3ex);
			\draw[black,fill=black] (3.4,0) circle (.3ex);
			\draw[black,fill=black] (3.1,0) circle (.3ex);
			\draw[black] (0,1.7) circle (0.5cm) node {p};
			\draw[black] (0,-1.7) circle (0.5cm) node {p};
			\draw[black] (1.7,1.7) circle (0.5cm) node {p};
			\draw[black] (-1.7,-1.7) circle (0.5cm) node {p};
			\draw[black] (1.7,-1.7) circle (0.5cm) node {p};
			\draw[black] (-1.7,1.7) circle (0.5cm) node {p};
			\draw[<->, thick] (-1.7,2.2)--(-1.7,2.9);
			\draw[<->, thick] (-1.7,2.9)--(-1.7,2.2);
			\draw[<->, thick] (1.7,2.2)--(1.7,2.9);
			\draw[<->, thick] (1.7,2.9)--(1.7,2.2);
			\draw[<->, thick] (0,2.2)--(0,2.9);
			\draw[<->, thick] (0,2.9)--(0,2.2);
			
			\draw[<->, thick] (-1.7,-2.2)--(-1.7,-2.9);
			\draw[<->, thick] (-1.7,-2.9)--(-1.7,-2.2);
			\draw[<->, thick] (1.7,-2.2)--(1.7,-2.9);
			\draw[<->, thick] (1.7,-2.9)--(1.7,-2.2);
			\draw[<->, thick] (0,-2.2)--(0,-2.9);
			\draw[<->, thick] (0,-2.9)--(0,-2.2);
			
			\draw[<->, thick] (-1.7,0.5)--(-1.7,1.2);
			\draw[<->, thick] (-1.7,1.2)--(-1.7,0.5);
			\draw[<->, thick] (1.7,0.5)--(1.7,1.2);
			\draw[<->, thick] (1.7,1.2)--(1.7,0.5);
			\draw[<->, thick] (0,0.5)--(0,1.2);
			\draw[<->, thick] (0,1.2)--(0,0.5);
			
			\draw[<->, thick] (-1.7,-0.5)--(-1.7,-1.2);
			\draw[<->, thick] (-1.7,-1.2)--(-1.7,-0.5);
			\draw[<->, thick] (1.7,-0.5)--(1.7,-1.2);
			\draw[<->, thick] (1.7,-1.2)--(1.7,-0.5);
			\draw[<->, thick] (0,-0.5)--(0,-1.2);
			\draw[<->, thick] (0,-1.2)--(0,-0.5);
			
		\end{tikzpicture}
	}
	\caption{An infinite 2-dimensional homogeneous Cartesian  network. Each node has external influence of~$p$, and is influenced by its four nearest-neighbors at internal influence rates of~$\frac{q}{4}$.}
	\label{fig:2D_grid}
\end{figure}
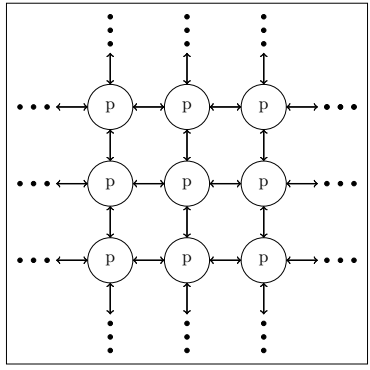

In~\cite{OR-10}, Fibich and Gibori conjectured that for any $p,q>0$, $$f_1(t;p,q)<f_2(t;p,q)<\dots<f_{\rm Bass}(t;p,q), \qquad 0<t<\infty.$$ So far, this conjecture has remained open. We now prove this conjecture for the initial dynamics:

\begin{lemma}
	\label{cor:dimension_dominance}
	Consider the $D$-dimensional Cartesian networks~\eqref{eq:cart_hom}. Then 
	\begin{equation*}
		f_{1D}(t;p,q)<f_{2D}(t;p,q)<f_{3D}(t;p,q)<\dots<f_{\rm Bass}(t;p,q), \qquad 0<t\ll 1.
	\end{equation*}
\end{lemma}
\begin{proof}
	Since $f_{\rm Bass}(t;p,q)=\lim\limits_{M\to\infty}f^{\rm complete}(t;p,q,M)$, see Section~\ref{subsec:homogeneous}, then by~\eqref{eq:initial_hom},
	\begin{equation*}
		f'_{\rm Bass}(0)=p,\qquad f''_{\rm Bass}(0)=p(q-p),\qquad
		f'''_{\rm Bass}(0)=\lim\limits_{M\to\infty}p\left(p^2-4pq+\frac{M-3}{M-1}q^2\right)=p\left(p^2-4pq+q^2\right).
	\end{equation*}
	In~\ref{app:lem:dimension_initial} we show that for any dimension $D$,
	\begin{equation}
		\label{eq:der_dimension_initial}
		f_D'(0)=p,\qquad f_D''(0)=p(q-p), \qquad f_D'''(0)=p\left(p^2-4pq+\frac{D-1}{D}q^2\right).
	\end{equation}
	Therefore, the result follows.	
\end{proof}

\section{Distribution of heterogeneity.}
\label{sec:effect_of_order}

When the network is not complete, the effect of heterogeneity on $f(t)$ depends also on the relative locations of the heterogeneous nodes in the network. To illustrate this, let $A$ be a one-sided circle with $M$ nodes which is homogeneous in $q$ and heterogeneous in $p$, so that
\begin{equation}
	\label{eq:p_distribution_A}
	q_i^A=q, \qquad
	p_i^A=
	\begin{cases}
		p_1,& 1\leq i\leq \frac{M}{2},\\
		p_2,& \frac{M}{2}<i\leq M,
	\end{cases} \qquad i=1,\dots,M.
\end{equation}
Let $B$ be a one-sided circle with $M$ nodes which are homogeneous in $q$ and heterogeneous in $p$, so that
\begin{equation}
	\label{eq:p_distribution_B}
	q_i^B=q, \qquad
	p_i^B=
	\begin{cases}
		p_1,& i \ {\rm odd,}\\
		p_2,& i \ {\rm even,}
	\end{cases} \qquad i=1,\dots,M.
\end{equation}
Thus, $A$ and $B$ have exactly the same heterogeneous nodes $\{(p_i,q_i)\}_{i=1}^M$, but their relative locations along the circle are different. 

We can explicitly compute the aggregate adoptions in $A$ and $B$ as $M\to\infty$:
\begin{lemma}
	\label{lem:order}
	Consider the one-sided circles~\eqref{eq:p_distribution_A} and~\eqref{eq:p_distribution_B}. Then

	\begin{equation}
		\label{eq:f_A}
		\lim\limits_{M\to\infty}f_A^{\rm het}(t)=\frac{f_{\rm 1D}(t;p_1,q)+f_{\rm 1D}(t;p_2,q)}{2},
	\end{equation}
	where $f_{\rm 1D}$ is given by~\eqref{eq:f_1D}, and
	\begin{subequations}
		\label{eqs:f_B}
		\begin{equation}
			\label{eq:f_B}
			\lim\limits_{M\to\infty}f_B^{\rm het}(t)=1-\frac{1}{2q}e^{-qt}(\dot{V_1}(t)+\dot{U_1}(t)),
		\end{equation}
		where $U_1$ and $V_1$ are the solution of
		\begin{equation}
			\label{eq:v1_system}
			\begin{aligned}
				\dot{U_1}(t)&=qe^{-p_1t}V_1(t), \qquad U_1(0)=1,\\
				\dot{V_1}(t)&=qe^{-p_2t}U_1(t), \qquad V_1(0)=1.
			\end{aligned}
		\end{equation}
		
	\end{subequations}
\end{lemma}
\begin{proof}
	See~\ref{app:lem:order}.
\end{proof}

\begin{figure}[ht!]
	\centering
	\begin{minipage}{.3\textwidth}
		\centering
		\begin{tikzpicture}
			\node[anchor=south west, inner sep =0] at (0,0) {\includegraphics[width=\textwidth]{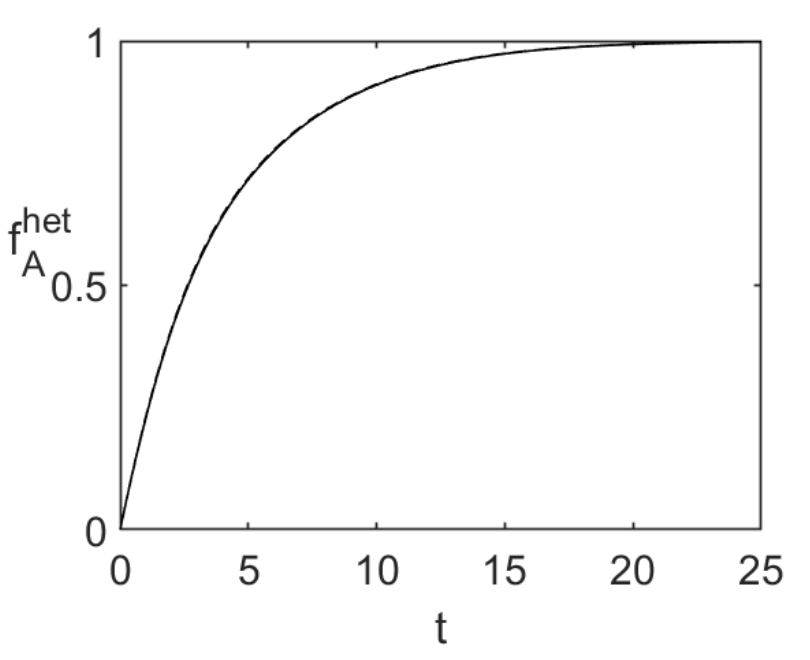}};
			\draw (4,3) node{A};
		\end{tikzpicture}
	\end{minipage}
	\begin{minipage}{.3\textwidth}
		\centering
		\begin{tikzpicture}
			\node[anchor=south west, inner sep =0] at (2,-1.2) {\includegraphics[width=\textwidth]{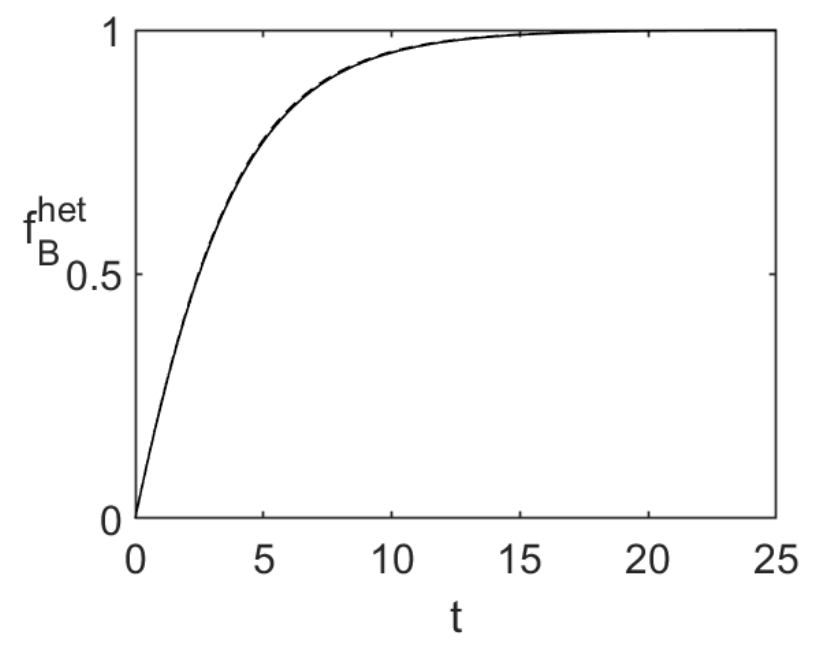}};
			\draw (6,1.8) node{B};
		\end{tikzpicture}
	\end{minipage}
	\begin{minipage}{.3\textwidth}
		\centering
		\begin{tikzpicture}
			\node[anchor=south west, inner sep =0] at (2,-1.2)
			{\includegraphics[width=\textwidth]{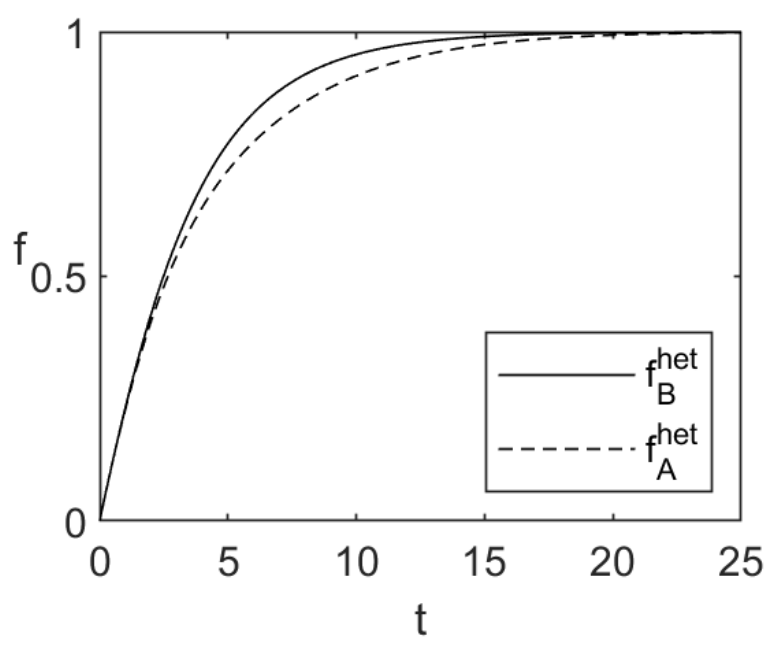}};
			\draw (6,1.8) node{C};
		\end{tikzpicture}
	\end{minipage}
	\caption{A)~Fraction of adopters $f_A$ as a function of time in circle~$A$ with $M=1000$, $p_1=0.4$, $p_2=0.1$, and $q=0.2$. Solid line is the explicit expression~\eqref{eq:f_A} and dashed line is the average of $10^4$   simulations of~\eqref{eq:one_sided_model}. The 2 curves are nearly indistinguishable. B)~Same for $f_B$. Solid line is the explicit expression~\eqref{eqs:f_B}. C)~Comparison of $f_A$~(dashes) and $f_B$~(solid).
	}
	\label{fig:order}
\end{figure}
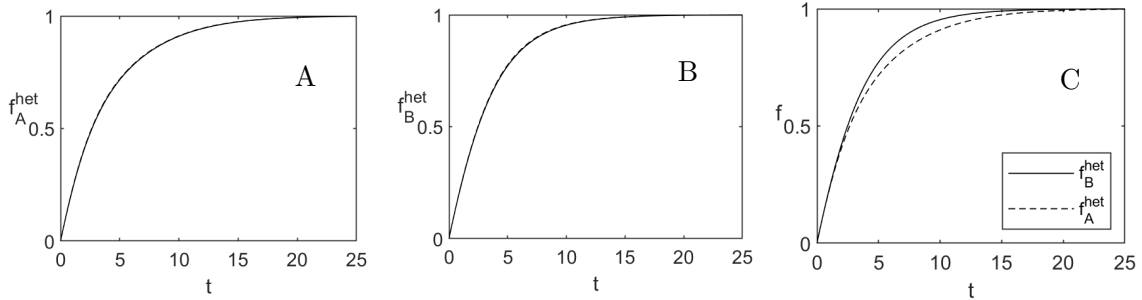

Figure~\ref{fig:order}A-B confirms the results of Lemma~\ref{lem:order} numerically. In addition, Figure~\ref{fig:order}C shows that
$$
\lim\limits_{M\to \infty}f_A^{\rm het}(t; p_1, p_2, q)<\lim\limits_{M\to \infty}f_B^{\rm het}(t; p_1,p_2,q) 
, \qquad 0< t<\infty.
$$
Therefore, in particular, the effect of heterogeneity depends also on the locations of the heterogeneous nodes along the circle. {\rev The intuition behind this inequality is as follows. Assume without loss of generality that $p_1>p_2$. In both networks, the diffusion is limited by the rate at which the ``weak" $p_2$ nodes adopt the product. In circle~$A$ there is negligible interaction between the separate regions of the weak and strong nodes, and so the weak $p_2$ nodes adopt the product without any ``assistance" from the $p_1$ nodes. In circle~$B$, however, whenever a strong $p_1$ node adopts, it immediately exerts an internal influence on its adjacent weak node to adopt. Hence, the weak nodes adopt the product more quickly, and so the aggregate diffusion is faster.
	
	We can further consider the one-sided circles $A$, $B$, and $C$ with $q_i^{A,B,C}\equiv q$ and 
	\begin{equation*}
		p_i^{A}=\begin{cases}
			p_1,& 1\leq i<\frac{M}{3},\\
			p_2,& \frac{M}{3}\leq i <\frac{2M}{3},\\
			p_3,& \frac{2M}{3}\leq i\leq M,
		\end{cases}\qquad
		p_i^{B}=\begin{cases}
			p_1,& i\ \Mod3=1,\\
			p_2,& i\ \Mod3=2,\\
			p_3,& i\ \Mod3=0,
		\end{cases}\qquad
		p_i^{C}=\begin{cases}
			p_1,& i\ \Mod3=1,\\
			p_3,& i\ \Mod3=2,\\
			p_2,& i\ \Mod3=0,
		\end{cases}
	\end{equation*}
	where $p_1>p_2>p_3$. Following the previous arguments, we expect that the diffusion in circle~$A$ (three separate uniform regions) is slower than in circles~$B$~and~$C$ (alternating patterns). Moreover, the diffusion in $C$ is faster than in $B$, since the weakest $p_3$ nodes are directly influenced by the strongest $p_1$ nodes. Figure~\ref{fig:order_M3} confirms these predictions numerically.
	
	This naturally leads to the following question: For a given network structure, what is the optimal distribution of $\{p_i\}$ that maximizes the diffusion? Based on the above arguments, the weak nodes should be close to strong ones. A systematic study of this intriguing optimization problem however, is left for a future study.
}

\begin{figure}[ht!]
	\centering
	\begin{minipage}{.3\textwidth}
		\centering
		\begin{tikzpicture}
			\node[anchor=south west, inner sep =0] at (0,0) {\includegraphics[width=\textwidth]{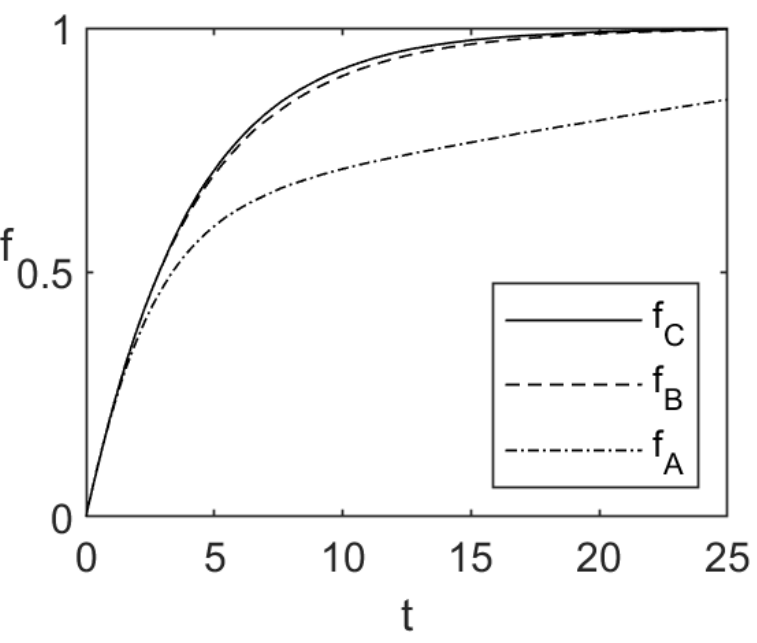}};
		\end{tikzpicture}
	\end{minipage}
	\caption{Comparison of $f_{A}$~(dash-dot), $f_{B}$~(dashes), and $f_{C}$~(solid), where $M=900$, $p_1=0.5$, $p_2=0.2$, $p_3=0.01$, and $q=0.2$.}
	\label{fig:order_M3}
\end{figure}

\section{Level of heterogeneity}
\label{sec:level}
{\rev
	Until now, we considered a dichotomous distinction of heterogeneity versus homogeneity. We now briefly consider the {\newrev quantitative effect of varying the level of heterogeneity on $f^{\rm hom}(t)-f^{\rm het}(t)$}. We consider vertex-transitive networks, i.e., networks that are structured exactly the same about any node (e.g., circles, infinite D-dimensional Cartesian networks, complete networks).
	\begin{lemma} [{\cite{Bass-SIR-analysis-17,PNAS-12}}]
		\label{lem:averaging_principal}
		Let $\varepsilon>0$ and $h_{j, p}$, $h_{j, q}\in \mathbb{R}$ for $j=1,\dots,M$, such that $\sum_{j=1}^{M}h_{j, p}=0$ and $\sum_{j=1}^{M}h_{j, q}=0$. Consider a vertex-transitive network with $M$ nodes which is heterogeneous in $p$ and mildly heterogeneous in $q$, see~\eqref{eq:q_mild_het}, i.e.,
		\begin{equation*}
			p_j(\varepsilon)=p(1+\varepsilon h_{j, p}), \qquad q_j(\varepsilon)=q(1+\varepsilon h_{j, q}), \qquad j=1,\dots, M.
		\end{equation*}
		Then, for $\varepsilon\ll 1$,
		\begin{equation}
			\label{eq:averaging_principal}
			\begin{aligned}
				f^{\rm het}(\varepsilon):&=f^{\rm het}(t;\{p_j(\varepsilon)\},\{q_j(\varepsilon)\})
				\\&=f^{\rm hom}(t;p,q)
				+\frac{\varepsilon^2}{2} \sum_{i=1}^{M}\sum_{j=1}^{M}
				\left(p^2 h_{i, p} h_{j, p} a_{i,j} +2pqh_{i, p} h_{j, q} b_{i,j} +q^2 h_{i, q} h_{j, q} c_{i,j}\right)+O(\varepsilon^3), 
			\end{aligned}
		\end{equation}
		where 
		$$
		a_{i,j}:=\frac{\partial^2f^{\rm het}}{\partial p_i\partial p_j}\bigg|_{\varepsilon=0}, \qquad
		b_{i,j}:=\frac{\partial^2f^{\rm het}}{\partial p_i\partial q_j}\bigg|_{\varepsilon=0}, \qquad
		c_{i,j}:=\frac{\partial^2f^{\rm het}}{\partial q_i\partial q_j}\bigg|_{\varepsilon=0}.
		$$
	\end{lemma}
	\begin{proof}
		The proof is similar to~\cite{Bass-SIR-analysis-17,PNAS-12}, except that here we prove the smoothness of $f^{\rm het}$ using the novel master equations~\eqref{eqs:masterHeterofc}. By~\eqref{eq:expectedHeterofc}, $f^{\rm het}(t)$ is a linear combination of $\{[S^k](t)\}$. Since $\{[S^k](t)\}$ are solutions of the master equations~\eqref{eqs:masterHeterofc}, which are linear constant-coefficient ODEs, they depend smoothly on $\{p_j\}$ and $\{q_{i,j}\}$, and so $f^{\rm het}$ depends smoothly on $\varepsilon$. Furthermore, since the network is vertex transitive, $f(t;\{p_j\},\{q_j\})$ is weakly symmetric, i.e., $f(t;\{(p,\dots,p)+\eta_1{\bf e}_i\},\{(q,\dots,q)+\eta_2 {\bf e}_i\})$ does not depend on $1\leq i\leq M$ for any $p,q,\eta_1,\eta_2$, where ${\bf e}_i$ is the unit vector in the $i$th coordinate. Therefore, by the averaging principle~\cite{PNAS-12}, $\frac{\partial f}{\partial \varepsilon}\equiv 0$, and so relation~\eqref{eq:averaging_principal} holds.	
	\end{proof}
	
	\begin{figure}[ht!]
		\centering
		\begin{minipage}{.3\textwidth}
			\centering
			\begin{tikzpicture}
				\node[anchor=south west, inner sep =0] at (0,0) {\includegraphics[width=\textwidth]{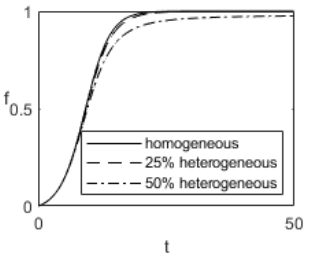}};
				\draw (4,3) node{A};
			\end{tikzpicture}
		\end{minipage}
		\begin{minipage}{.34\textwidth}
			\centering
			\begin{tikzpicture}
				\node[anchor=south west, inner sep =0] at (0,0) {\includegraphics[width=\textwidth]{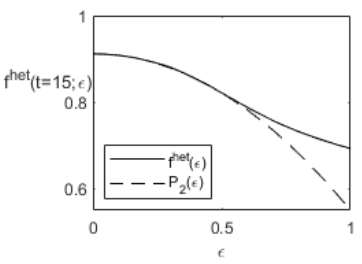}};
				\draw (4.5,3) node{B};
			\end{tikzpicture}
		\end{minipage}
		\caption{\rev A)~Fractional adoption in a complete homogeneous (solid) and mildly-heterogeneous in q networks with $\varepsilon=25\%$ (dashes) and $\varepsilon=50\%$ (dash-dot). Plots show the average of $10000$ simulations with $M=1000$, $p=0.01$, $q=0.4$, and where $\{h_{j, q}\}$ are generated by a standard normal distribution. B)~$f^{\rm het}(t=15;\varepsilon)$ as a function of $\varepsilon$ (solid). The fitted parabola is $P_2(\varepsilon)=0.9122-0.3616\varepsilon^2$ (dashes).}
		\label{fig:variance}
	\end{figure}
	
	Lemma~\ref{lem:averaging_principal} shows that the effect of heterogeneity increases smoothly and monotonically with the variance in $\{p_j\}$ and $\{q_j\}$, at least for a weak heterogeneity. {\em The effects of the variances of $\{p_j\}$ and $\{q_j\}$ however, are $O(\varepsilon^2)$ small, i.e., are much smaller than the effects of their means.} These conclusions also hold for any network structure for $t\ll 1$~(\ref{app:small_time}), and were observed numerically in~\cite{Bass-SIR-analysis-17, PNAS-12, OR-10, GLM-01}. 
	
	For example, consider a complete network which is mildly heterogeneous in q. Then $f^{\rm het}-f^{\rm hom}<0$ by Theorem~\ref{thm:compare_q}, and so a higher variance leads to  slower diffusion (Figure~\ref{fig:variance}A). This slowdown effect, however, is quite small, and barely noticeable for $\varepsilon=25\%$. The $O(\varepsilon^2)$ effect of heterogeneity holds not only for very small values of $\varepsilon$, but rather for $0\leq \varepsilon\leq 50\%$~(Figure~\ref{fig:variance}B). Therefore, although Lemma~\ref{lem:averaging_principal} is formally stated for small values of $\varepsilon$, in practice it holds for ``larger" values of $\varepsilon$.
	
}



\appendix 
\label{Appendix}
\renewcommand\thesection{Appendix~\Alph{section}}
\section{Proof of Lemma~\ref{lem:masterHeterofc}.}
\label{app:lem:masterHeterofc}
Let $(S^{m_1},\dots,S^{m_n},I^l)(t)$ denote the event that at time t, nodes $\{m_1,\dots,m_n\}$ are non-adopters and node $l_j$ is an adopter, where $1\leq n<M$, $m_i\in \{1,\dots,M\}$, $m_i\neq m_j$ if $i\neq j$, and $l_j \notin \left \lbrace m_1,\dots,m_n \right \rbrace$. Let $[S^{m_1},\dots,S^{m_n},I^{l_j}](t)$ denote the probability that such an event occurs.

The configuration $(S^{m_1},\dots,S^{m_n})$ cannot be created. It is destroyed when:
\begin{enumerate}[label=(\arabic*)]
	\item 
	Any of $n$ non-adopters adopts the product through external influence, which happens at rate $p_{m_i}$ for the $m_i^{{\rm th}}$ non-adopter.
	\item 
	Any of the $n$ non adopters in $(S^{m_1},\dots,S^{m_n},I^{l_j})$ adopts due to external influence by node ${l_j}$ for all $n+1\leq j\leq M$, which happens at rate $q_{l_j,m_i}$.
\end{enumerate}
Therefore,
\begin{equation*}
	\label{eq:premasterHeterofc}
	\frac{d}{dt}[S^{m_1},\dots,S^{m_n}](t)= -\left(\sum_{i=1}^{n}p_{m_i}\right)[S^{m_1},\dots,S^{m_n}](t)-\sum_{j=n+1}^{M}\left(\sum_{i=1}^{n}q_{l_j,m_i}\right)[S^{m_1},\dots,S^{m_n},I^{l_j}](t).
\end{equation*}
By the total probability theorem,
\begin{equation*}
	\label{eq:IS->S}
	[S^{m_1},\dots,S^{m_n},I^{l_j}](t)=[S^{m_1},\dots,S^{m_n}](t)-[S^{m_1},\dots,S^{m_n},S^{l_j}](t).
\end{equation*}
Combining these two relations gives~\eqref{eq:masterHeterofc}. 

The configuration $(S^1,S^2,\dots,S^M)$ cannot be created. It is destroyed when any of the ``$S$''s turns into an ``$I$'', which happens at the rate $p_i$ for the $i^{{\rm th}}$ node, hence giving~\eqref{eq:masterHeterofcM}.

The initial conditions~\eqref{eq:masterHeterofcInitial} follow from~\eqref{eq:general_initial}.

\section{Derivation of eq.~\eqref{eq:master_sol} for~$[S^k](t)$.}
\label{app:S^k}

Substituting \eqref{eq:s_M sol} in~\eqref{eq:masterHeterofc} with $n=M-1$ and $l_1=\{1,\dots,M\}\backslash\{m_1,\dots,m_{M-1}\}$ gives
\begin{equation*}
	\frac{d}{dt}[S^{m_1},\dots,S^{m_{M-1}}](t)= -\left(\sum_{i=1}^{M-1}\left(p_{m_i}+ q_{{l_1},m_i}\right)\right)[S^{m_1},\dots,S^{m_{M-1}}](t) +\left(\sum_{i=1}^{M-1}q_{{l_1},m_i}\right)e^{-\left(\sum_{j=1}^{M}p_j\right)t}.
\end{equation*}
The solution of this ODE with the initial condition~\eqref{eq:masterHeterofcInitial} is
\begin{equation*}
	[S^{m_1},\dots,S^{m_{M-1}}](t)=c_0e^{-\left(\sum_{j=1}^{M}p_j\right)t}+c_1e^{-\left(\sum_{i=1}^{M-1}\left(p_{m_i}+ q_{{l_1},m_i}\right)\right)t},
\end{equation*}
where
$c_0=\frac{\sum_{i=1}^{M-1}q_{{l_1},m_i}}{\sum_{i=1}^{M-1}q_{{l_1},m_i}-p_{l_1}}$ and $c_1=-\frac{p_{l_1}}{\sum_{i=1}^{M-1}q_{{l_1},m_i}-p_{l_1}}$.
As this process is repeated for $n=M-2,M-3,\dots,1$, at each stage we solve an ODE of the form
\begin{equation*}
	y'+ay=\sum_{j}c_je^{-b_j}t, \qquad a:=\sum_{i=1}^{n}p_{m_i}+\sum_{j=n+1}^{M}\sum_{i=1}^{n} q_{l_j,m_i}.
\end{equation*}
The solution of this ODE is of the form $y=c_0e^{-at}+\sum_{j}\tilde{c}_je^{-b_j}t.$
Thus, $y$ is a linear combination of all the ``old" right-hand-side exponents plus the new exponent $e^{-at}$.
Therefore, for $n=1$ and $m_1=k$ we get~\eqref{eq:master_sol}.

\section{Proof of Corollary~\ref{cor:n'''(0)_mild}.}
\label{app:cor:n'''(0)_mild}
When $p$ is homogeneous and $q$ is mildly heterogeneous, the master equation~\eqref{eq:masterHeterofc} reads
\begin{equation}
	\label{eq:masterHeterofc_mild}
	\begin{aligned}
		\frac{d}{dt}&[S^{m_1},\dots,S^{m_n}](t)=\\ &-\left(np+\frac{M-n}{M-1}\sum_{i=1}^{n} q_{m_i}\right)[S^{m_1},\dots,S^{m_n}](t) +\sum_{j=n+1}^{M}\left(\sum_{i=1}^{n} \frac{q_{m_i}}{M-1}\right)[S^{m_1},\dots,S^{m_n},S^{l_j}](t),
	\end{aligned}
\end{equation}
Substituting $n=1$ in~\eqref{eq:masterHeterofc_mild} and differentiating gives
\begin{equation}
	\label{eq:masterHeterofc_mild_prime}
	\frac{d^2}{dt^2}[S^{m_1}](t)=-\left(p+q_{m_1}\right)[S^{m_1}](t) +\sum_{j=2}^{M} \frac{q_{m_1}}{M-1}[S^{m_1},S^{m_j}](t).
\end{equation} 
Substituting $[S^{m_1}]'(0)=-p$ and $[S^{m_1,l}]'(0)=-2p$, see~\eqref{eq:s'_0}, gives
\begin{subequations}
	\label{eqs:s''(0)_mild}
	\begin{equation}
		\label{eq:s_1''(0)_mild}
		\frac{d^2}{dt^2}[S^i](0)=(p+q_i)(p)-(2p)q_i=p(p-q_i).
	\end{equation}
	Similarly, 
	\begin{equation}
		\label{eq:s_2''(0)_mild}
		\frac{d^2}{dt^2}[S^i,S^j](0)=\left(2p+\frac{M-2}{M-1}(q_i+q_j)\right)(2p)-(3p)(M-2)\left(\frac{q_i}{M-1}+\frac{q_j}{M-1}\right)=p\left(4p-\frac{M-2}{M-1}\left(q_i+q_j\right)\right).
	\end{equation}
\end{subequations}
Differentiating~\eqref{eq:masterHeterofc_mild_prime} and using equations~\eqref{eqs:s''(0)_mild} gives
\begin{equation*}
	\frac{d^3}{dt^3}[S^i](0)=-(p+q_i)p(p-q_i)+\frac{q_i}{M-1}\sum_{j=1,j\neq i}^{M}p\left(4p-\frac{M-2}{M-1}\left(q_i+q_j\right)\right),
\end{equation*}
and so we get the desired result by~\eqref{eq:expectedHeterofc}.

\section{Proof of Theorem~\ref{thm:CDF-dominance}.}
\label{app:CDF-dominance}
Inequalities~\eqref{eqs:equ:1_2} imply that for every $i$, there is a probability space $(\Omega'_i,P'_i)$
and two random variables ${t'}_i^A(\omega'_i)$ and ${t'}_i^B(\omega'_i)$ that satisfy 
\begin{subequations}
	\label{eqs:ti'si'}
	\begin{equation}
		\begin{aligned}
			\label{eq:ti'si'cdf}
			F_{t_i^A}(\tau) &= F_{{t'}_i^A}(\tau), \qquad 1\leq i \leq M, \qquad \tau\geq 0,\\
			F_{t_i^B}(\tau) &= F_{{t'}_i^B}(\tau), \qquad 1\leq i \leq M, \qquad \tau\geq 0,
		\end{aligned}
	\end{equation}
	and also the pointwise dominance condition
	\begin{equation}
		\label{eq:ti'si'pointwise}
		{t'}_i^A (\omega'_i)\geq {t'}_i^B(\omega'_i), \qquad \forall i, \forall \omega'_i \in  \Omega'_i, \qquad {t'}_j^A (\omega'_i)>{t'}_j^B(\omega'_i) \qquad \forall \omega'_i \in  \Omega'_i.
	\end{equation}
\end{subequations}
For example, let $0 \le \omega_i' \le 1$, let $P'_i$ be the Lebesgue measure of~$\Omega'_i = [0,~1]$, and let ${t'}_i^A =  F_{t_i^A}^{-1}$. 
Then 
$$
F_{{t'}_i^A}(t) = {\rm Prob}({t'}_i^A(\omega_i') \le t)  = \mu(0 \le \omega_i' \le ({t'}_i^A)^{-1}(t)) = \mu(0 \le \omega_i' \le F_{t_i^A}(t)) = F_{t_i^A}(t),
$$
which gives~\eqref{eq:ti'si'cdf}. In addition, since $F_{t_i^A}(t) \leq F_{t_i^B}(t)$, then ${t'}_i^A =  F_{t_i^A}^{-1} \ge F_{t_i^B}^{-1} = {t'}_i^B$, and since $F_{t_j^A}(t) < F_{t_j^B}(t)$, then ${t'}_j^A > {t'}_j^B$. Therefore, we have~\eqref{eq:ti'si'pointwise}.
By~\eqref{eq:ti'si'pointwise}, ${T'}_m^A(\omega') \ge {T'}_m^B(\omega')$ for all~$m$ and~$\omega'$, and furthermore ${T'}_k^A(\omega')> {T'}_k^B(\omega')$ for $j\leq k\leq M$.  Hence, for $0<t<\infty$,
\begin{equation}
	\label{eq:T_CDF_Inequality}
	\begin{cases}
		{\rm Prob}({T'}_m^A \leq t)\leq{\rm Prob}({T'}_m^B \leq t),&\qquad 1\leq k\leq j-1,\\
		{\rm Prob}({T'}_k^A \leq t)<{\rm Prob}({T'}_k^B \leq t),&\qquad j\leq k\leq M.
	\end{cases}
\end{equation}
Since $
{\rm Prob}(N_A(t) \geq m)  = \sum_{k=m}^M  {\rm Prob}(N_A(t)=k),
$
then 
\begin{equation}
	\label{eq:P(n_T(t)>=n-B}
	\sum_{m=1}^M    {\rm Prob}(N_A(t) \geq m) = \sum_{k=1}^M k \cdot {\rm Prob}(N_A(t)=k) = E[N_A(t)].
\end{equation}
Therefore,
\begin{equation}
	\label{eq:expected_inequality}
	\E_{\omega'}[N'_A(t)]=\sum_{m=1}^{M}{\rm Prob}(N'_A(t)\geq m)=\sum_{m=1}^{M}{\rm Prob}({T'}_m^A \leq t)<\sum_{m=1}^{M}{\rm Prob}({T'}_m^B \leq t)=\E_{\omega'}[N'_B(t)],
\end{equation}
where the inequality follows from~\eqref{eq:T_CDF_Inequality}.
Since the $(t_i^A)$'s are independent, then so are the $({t'}_i^A)$'s. Therefore, 
since $t_i^A$ and ${t'}_i^A$ are identically distributed, see~\eqref{eq:ti'si'cdf},  
then
$$
{\rm Prob} (T_m^A \le t)=  {\rm Prob} ({T'}_m^A \le t)  \quad  \forall m,\forall t. 
$$

By~\eqref{eq:N(t)}, this equality can be rewritten as 
\begin{equation}
	\label{eq:P(n_T(t)>=n}
	{\rm Prob}(N_A(t) \geq m) =  {\rm Prob}(N'_A(t) \geq m), \ \ \ \forall m,\forall t. 
\end{equation}
Therefore, by~\eqref{eq:P(n_T(t)>=n} and~\eqref{eq:P(n_T(t)>=n-B},
\[ \E_{\omega}[N_A(t)]  = \E_{\omega'}[N'_A(t)], \qquad t \ge 0 . \]
Similarly, 
\[ \E_{\omega}[N_B(t)]  = \E_{\omega'}[N'_B(t)], \qquad t \ge 0 . \]
The result follows from~\eqref{eq:expected_inequality} and the last two relations.

\section{Proof of Lemma~\ref{lem:compare_M=2p}.}
\label{app:lem:compare_M=2p}
We first recall an auxiliary lemma:
\begin{lemma} [{\cite[Lemma E.1.]{Bass-boundary-18}}]
	\label{lem:diff_eq_positive}
	Let $\sigma(t)$ be the solution of 
	\begin{equation*}
		\frac{d}{dt}\sigma(t)+K\sigma(t)=b(t), \qquad t>0, \qquad \sigma(0)=0,
	\end{equation*}
	where K is a constant, and $b(t)>0$ for $t>0$. Then $\sigma(t)>0$ for $t>0$.
\end{lemma}
To prove Lemma~\ref{lem:compare_M=2p}, let $\delta:=p_1-p=-(p_2-p)>0$. By~\eqref{eq:expectedHeterofc},
\begin{equation}
	\label{eq:p_difference}
	f^{\rm hom}(t)-f^{\rm het}(t)=\frac{([S^1]^{{\rm het}}(t)+[S^2]^{{\rm het}}(t))-([S^1]^{{\rm hom}}(t)+[S^2]^{{\rm hom}}(t))}{2}.
\end{equation}
Since  for both networks $p_1+p_2=2p$, then by~\eqref{eq:s_M sol},
\begin{equation*}
	[S^1,S^2]^{\rm het}(t)=[S^1,S^2]^{\rm hom}(t)=[S_2](t), \qquad [S_2](t):=1-e^{-2pt}.
\end{equation*}
By~\eqref{eq:master_M=2},
\begin{equation}
	\label{eq:master_p_var_lemma}
	\frac{d}{dt}[S^i]^{\rm het}(t)=-(p_i+q)[S^i]^{\rm het}(t)+q[S_2](t),
\end{equation}
and so
\begin{equation}
	\label{eq:S_hetp}
	\frac{d}{dt}([S^1]^{\rm het}+[S^2]^{\rm het})(t)=-(p+\delta+q)[S^1]^{\rm het}(t)+q[S_2](t)-(p-\delta+q)[S^2]^{\rm het}(t)+q[S_2](t),
\end{equation}
and
\begin{equation}
	\label{eq:S_hom}
	\frac{d}{dt}([S^1]^{\rm hom}+[S^2]^{\rm hom})(t)=2\frac{d}{dt}[S^1]^{\rm hom}(t)=2\left(-(p+q)[S^1]^{\rm hom}(t)+q[S_2](t)\right).
\end{equation}
Let $y(t):=([S^1]^{\rm het}(t)+[S^2]^{\rm het}(t))-([S^1]^{\rm hom}(t)+[S^2]^{\rm hom}(t))=2(f^{\rm hom}(t)-f^{\rm het}(t))$, see~\eqref{eq:p_difference}. We need to prove that $y(t)>0$ for $0<t<\infty$.
By~\eqref{eq:S_hetp} and~\eqref{eq:S_hom},
\begin{equation*}
	\frac{d}{dt}y(t)= -(p+q)y(t)-\delta[S^1]^{\rm het}(t)+\delta[S^2]^{\rm het}(t), \qquad y(0)=0.
\end{equation*}
Therefore, by Lemma~\ref{lem:diff_eq_positive}, it suffices to show that
\begin{equation}
	\label{eq:y>0}
	-\delta([S^1]^{\rm het}(t)-[S^2]^{\rm het}(t))>0.
\end{equation}
By~\eqref{eq:master_p_var_lemma}
\begin{equation*}
	(\frac{d}{dt}[S^1]^{\rm het}-\frac{d}{dt}[S^2]^{\rm het})(t)+(p+q)([S^1]^{\rm het}-[S^2]^{\rm het})(t)=-\delta([S^1]^{\rm het}+[S^2]^{\rm het})(t).
\end{equation*}
Applying Lemma~\ref{lem:diff_eq_positive} to this ODE gives that  $[S^1]^{\rm het}(t)<[S^2]^{\rm het}(t)$ for $t>0$,  hence~\eqref{eq:y>0} holds.

\section{CDF dominance condition.}
\label{app:cdf_dominance}
We begin with two auxiliary lemmas.

\begin{lemma}
	\label{lem:i_is_first}
	Assume that the time $t_j$ at which $j$ adopts is exponentially distributed with parameter $\lambda_j$, where $j=1:J$. Then
	\begin{equation}
		\label{eq:i_is_first}
		{\rm Prob}{i~{\rm adopts\ before} \choose {\rm all\ others}} = \frac{\lambda_i}{\sum_{j=1}^J \lambda_j}.
	\end{equation}
\end{lemma}

\begin{proof}
	\begin{eqnarray*}
		{\rm Prob}{i~{\rm adopts\ before} \choose {\rm all\ others}} & = & \int_0^\infty f(i~{\rm adopts\ at}\ t) \prod_{j=1, j \not=i}^J \left( \int_t^\infty f(j~{\rm adopts\ at}\ \tau_j) \, d\tau_j \right) \,  dt 
		\\ & = &   \int_0^\infty \lambda_i e^{-\lambda_i t}  \prod_{j=1, j \not=i}^J \left( \int_t^\infty \lambda_j e^{-\lambda_j \tau_j} \, d\tau_j \right) \,  dt 
		= \int_0^\infty \lambda_i e^{-\lambda_i t}  \prod_{j=1, j \not=i}^J e^{-\lambda_j t}  \,  dt 
		\\ &=&  \int_0^\infty \lambda_i e^{-(\sum_{j=1}^J \lambda_j) t}   \,  dt
		= \frac{\lambda_i}{\sum_{j=1}^J \lambda_j}.
	\end{eqnarray*}		
\end{proof}

\begin{lemma}
	\label{lem:iChebyshev}
	Let $a = \frac1n \sum_{i=1}^n  a_i$ and $w_i = g(a_i)$, where $g$ is monotonically increasing and  $\sum_{i=1}^n w_i = 1$.
	Then
	$
	\sum_{i=1}^n w_i a_i \ge  a.   
	$
\end{lemma}
\begin{proof}
	Without loss of generality,  $a_1 \le a_2 \le \dots \le  a_n$. Therefore, from the monotonicity of $g$, $w_1 \le w_2  \le \dots \le w_n$. 
	Hence, one can apply the {\em Chebyshev sum inequality}
	\begin{equation}
		\label{eq:chebyshev}
		\frac1n \sum_{i=1}^n  a_i w_i \ge \left(\frac1n \sum_{i=1}^n  a_i \right) \left( \frac1n \sum_{i=1}^n  w_i \right) =  \frac{a}n.
	\end{equation}	
\end{proof}
\noindent We now prove the CDF dominance condition:


\begin{proof}{Proof of Lemma~\ref{lem:F^heter_k<F^hom_k_pq}: }
	The time $t_1^{\rm hom}$ until the first adoption in the homogeneous network is exponentially distributed with parameter~$Mp$. 
	Therefore, the corresponding CDF is 
	$$
	F^{\rm hom}_1(t) :={\rm Prob}(t_1^{\rm hom} \le t) =   1-e^{-Mp t}.
	$$
	Similarly, the time $t_1^{\rm het}$ until the first adoption in the heterogeneous network is exponentially distributed with parameter~$\sum_{i = 1}^M p_i$. 
	Therefore, the corresponding CDF is 
	$$
	F^{\rm het}_1(t) := {\rm Prob}(t_1^{\rm het} \le t) =   1-e^{-(\sum_{i = 1}^M p_i) t}.
	$$
	Hence, by definition of~$p$, 
	$$
	F^{\rm hom}_1(t)   =  F^{\rm het}_1(t).
	$$
	
	In the homogeneous case,  the time $t_2^{\rm hom}$ between the first and second adoptions  is exponentially distributed with parameter 
	$(M-1)(p+\frac{q}{M-1})=(M-1)p+q$.
	Therefore, the corresponding CDF is 
	$$
	F^{\rm hom}_2(t) :={\rm Prob}(t_2^{\rm hom} \le t) =   1-e^{-((M-1)p+q)t}.
	$$
	In the heterogeneous case,  let $w_k$ denote the probability that the first adopter was~$k$. In that case,
	the time between the first and second adoptions is exponentially distributed with parameter
	$\sum_{i = 1, i \not=k}^M (p_i+\frac{q_i}{M-1}) = Mp-p_k+\frac{Mq-q_k}{M-1}$, and so the corresponding conditional CDF 
	is 
	$$
	F_{2,k}(t) :=   {\rm Prob}(t_2^{\rm hom} \le t ~|~ \rm{1st\ adopter\ was}\ k) =    1-e^{-(Mp-p_k+\frac{Mq-q_k}{M-1})t}.
	$$
	Therefore, the overall CDF for~$t_2$ is 
	\begin{eqnarray*}
		F^{\rm het}_2(t) &=& \sum_{k = 1}^M  {\rm Prob}(t_2^{\rm hom} \le t ~|~ \rm{1st\ adopter\ was}\ k) \cdot   {\rm Prob}(\rm{1st\ adopter\ was}\ k)
		\\  &=& \sum_{k = 1}^M  w_k F_{2,k}(t) 
		= \sum_{k = 1}^M  w_k \left( 1-e^{-((M-1)p+\frac{Mq-q_k}{M-1})t}\right)  < 1- e^{-\sum_{k = 1}^M  w_k ((Mp-p_k+\frac{Mq-q_k}{M-1})t},
	\end{eqnarray*}
	where the last inequality follows from the strict concavity ($G_{\lambda \lambda}<0$) of the function $G = 1-e^{-\lambda t}$ when $0<t<\infty$.
	%
	%
	%
	
	By~\eqref{eq:i_is_first}, 
	\begin{equation}
		\label{eq:w_k}
		w_k = \frac{p_k}{\sum_{j=1}^M p_i} = \frac{p_k}{Mp}.
	\end{equation}
	Therefore, $w_k$ is monotonically increasing in~$p_k$, and so by Lemma~\ref{lem:iChebyshev},
	\begin{equation}
		\label{eq:sum_w_k*p_k>p}
		\sum_{k = 1}^M  w_k p_k  \geq \frac1M  \sum_{k = 1}^M  p_k  =  p.
	\end{equation}
	Hence, since $G = 1-e^{-\lambda t}$ is monotonically increasing in $\lambda$, 
	$$
	F^{\rm het}_2(t)  < 1- e^{-((M-1)p+q)t} = F^{\rm hom}_2(t), \qquad 0<t<\infty.
	$$

	In the homogeneous case,  the time $t_3$ between the second and third adoptions is exponentially distributed with parameter $(M-2)(p+ \frac{2q}{M-1})$.
	Therefore, the corresponding CDF is given by
	\begin{equation}
		\label{eq:F_+3^hom}
		F^{\rm hom}_3(t) :=  {\rm Prob}(t_3^{\rm hom} \le t) =  1-e^{-((M-2)p+2q \frac{M-2}{M-1})t}.
	\end{equation}
	In the heterogeneous case,  let $w_{k,m}$ denote the probability that the first and second adopters were~$k$ and $m$, respectively. In that case, 
	the time between the second and third  adoptions is exponentially distributed with parameter
	$\sum_{i = 1, i \not=k,m}^M (p_i+ 2\frac{q_i}{M-1}) = Mp-p_k-p_m+2\frac{Mq-q_k-q_m}{M-1}$, and so the corresponding conditional CDF 
	is $F_{3,k,m}(t) = 1-e^{-(Mp-p_k-p_m+2\frac{Mq-q_k-q_m}{M-1})t}$. Therefore, the overall CDF for $t_3^{\rm het}$ is 
	\begin{eqnarray}
		\nonumber
		F^{\rm het}_3(t) &=& \sum_{\substack{k,m = 1\\ m\neq k}}^M  w_{k,m} F_{3,k,m}(t) = \sum_{\substack{k,m = 1\\ m\neq k}}^M  w_{k,m} \left( 1-e^{-(Mp-p_k-p_m+2\frac{Mq-q_k-q_m}{M-1})t}\right) 
		\\ &<& 
		1- e^{-\sum_{\substack{k,m = 1\\ m\neq k}}^M  w_{k,m} ((Mp-p_k-p_m+2\frac{Mq-q_k-q_m}{M-1})t},
		\label{eq:CDF3pq}
	\end{eqnarray}
	where the last inequality follows from the strict concavity of~$G = 1-e^{-\lambda t}$ when $0<t<\infty$.
	%
	%
	%
	By~\eqref{eq:i_is_first}, 
	$$
	w_{k,m} = w_k \cdot {\rm Prob}({\rm 2nd\ adopter\ is\ m}\ |\ \rm{1st\ adopter\ was}\ k)  =\frac{p_k}{Mp}\left( \frac{p_m+\frac{q_m}{M-1}}{Mp-p_k+\frac{Mq-q_k}{M-1}}  \right)  .
	$$
	In addition, 
	\begin{align*}
		\sum_{\substack{k,m = 1\\ m\neq k}}^M  w_{k,m} \left(p_k+\frac{2q_k}{M-1}\right)  &= \sum_{k= 1}^M w_k \left(p_k+\frac{2q_k}{M-1}\right) \underbrace{\sum_{m= 1, m\not=k}^M  {\rm Prob}({\rm 2nd\ adopter\ is\ m}\ |\ \rm{1st\ adopter\ was}\ k)}_{=1}
		\\&=\left(p+2\frac{q}{M-1}\right),
	\end{align*}

	and
	\begin{eqnarray*}
		\sum_{\substack{k,m = 1\\ m\neq k}}^M  w_{k,m} \left(p_m+\frac{2q_m}{M-1}\right)  &= &
		\sum_{k = 1}^M  \frac{p_k}{Mp} \sum_{m = 1, m \not=k}^M \frac{p_m+\frac{q_m}{M-1}}{Mp-p_k+\frac{Mq-q_k}{M-1}}\left(p_m+\frac{2q_m}{M-1}\right)   
		\\ &\geq& \sum_{k = 1}^M  \frac{1}{M} \sum_{m = 1, m \not=k}^M \frac{p_m+\frac{q_m}{M-1}}{(M-1)p+\frac{Mq-q}{M-1}}\left(p_m+\frac{2q_m}{M-1}\right)  
		\\ &\geq& \sum_{k = 1}^M  \frac{1}{M} \sum_{m = 1, m \not=k}^M \frac{p+\frac{q}{M-1}}{(M-1)p+\frac{Mq-q}{M-1}}\left(p+\frac{2q}{M-1}\right) 
		= \left(p+\frac{2q}{M-1}\right)  ,
	\end{eqnarray*}
	where the two inequalities follow from Lemma~\ref{lem:iChebyshev}. 
	Therefore, 
	$$
	\sum_{\substack{k,m = 1\\ m\neq k}}^M  w_{k,m} (p_k+p_m+2\frac{q_k+q_m}{M-1}) 
	\geq  2p+4\frac{q}{M-1}
	$$
	Hence, since $G = 1-e^{-\lambda t}$ is monotonically increasing in $\lambda$, then using~\eqref{eq:F_+3^hom} and~\eqref{eq:CDF3pq},
	$$
	F^{\rm het}_3(t) < 1-e^{-((M-2)p+2q \frac{M-2}{M-1})t} = F^{\rm hom}_3(t), \qquad 0<t<\infty.
	$$
	
	Proceeding similarly, one can show that~\eqref{eq:F^heter_k<F^hom_k_pq} also holds for $k = 4, \dots, M$.  	
\end{proof}

\begin{proof}{Proof of Lemma~\ref{lem:t_i_independent}: }
	We need to prove that for any $m \in \{ 2, \dots,  M\}$, $t_m \in \{ t_m^{\rm het}, t_m^{\rm hom}\}$, and $\{\tau_1, \dots, \tau_m \}  \in ({\Bbb R}^+)^m$, 
	$$
	{\rm Prob}(t_m \le \tau_m) = {\rm Prob}(t_m \le \tau_m ~|~t_1 \le \tau_1, \dots,  t_{m-1} \le \tau_{m-1} ). 
	$$ 
	In the homogeneous case, the random variable $t_m^{\rm hom}$ and the conditional random variable $t_m^{\rm hom} ~|~(t_1^{\rm hom} \le \tau_1, \dots,  t_{m-1}^{\rm hom} \le \tau_{m-1})$
	are both exponentially distributed with parameter  $(M-(m-1))(p+ (m-1)\frac{q)}{M-1}) $. 
	Hence,
	$$
	{\rm Prob}(t_m^{\rm hom} \le \tau_m) = 1-e^{-((M-(m-1))p+(m-1)q \frac{M-(m-1)}{M-1})t} =  {\rm Prob}(t_m^{\rm hom} \le \tau_m ~|~t_1^{\rm hom} \le \tau_1, \dots,  t_{m-1}^{\rm hom} \le \tau_{m-1} ). 
	$$
	
	
	In the heterogeneous case, to simplify the calculations, we present the proof for $m=2$, i.e., 
	\begin{equation}
		\label{eq:t_2_ind_of_t1}
		{\rm Prob}(t_2^{\rm het} \le t) =   {\rm Prob}(t_2^{\rm het} \le t ~|~t_1^{\rm het} \le \tau_1).
	\end{equation}
	The proof for the other cases is identical. 
	Recall that
	$$
	{\rm Prob}(t_2^{\rm het} \le t) = \sum_{k = 1}^M  {\rm Prob}(t_2^{\rm het} \le t ~|~ \rm{1st\ adopter\ was}\  {\it k}) \cdot   {\rm Prob}(\rm{1st\ adopter\ was}\  {\it k}).
	$$
	Therefore,
	\begin{eqnarray*}
		&&  {\rm Prob}(t_2^{\rm het} \le t ~|~t_1^{\rm het} \le \tau_1)=
		\\
		&& \qquad \sum_{k = 1}^M  {\rm Prob}(t_2^{\rm het} \le t ~|~ \rm{1st\ adopter\ was}\ {\it k} ,~t_1^{\rm het} \le \tau_1) \cdot   {\rm Prob}(\rm{1st\ adopter\ was}\  {\it k}  ~|~t_1^{\rm het} \le \tau_1).
	\end{eqnarray*}
	The conditional random variables $t_2^{\rm het} ~|~ (\rm{1st\ adopter\ was}\  {\it k})$ and $t_2^{\rm het} ~|~ (\rm{1st\ adopter\ was}\  {\it k},~t_1^{\rm het} \le \tau_1)$
	are both exponentially distributed with parameter  $(Mp -p_k+q)$. Therefore, 
	$$
	{\rm Prob}(t_2^{\rm het} \le t ~|~ \rm{1st\ adopter\ was}\  {\it k}) =  {\rm Prob}(t_2^{\rm het} \le t ~|~ \rm{1st\ adopter\ was}\  {\it k},~t_1^{\rm het} \le \tau_1)
	$$
	In addition, by~\eqref{eq:w_k}, ${\rm Prob}(\rm{1st\ adopter\ was}\  {\it k})  = w_k =  \frac{p_k}{\sum_{j=1}^M p_i}$
	is independent of $t_1^{\rm het}$, and so 
	$$
	{\rm Prob}(\rm{1st\ adopter\ was}\  {\it k})  =  {\rm Prob}(\rm{1st\ adopter\ was}\  {\it k}  ~|~t_1^{\rm het} \le \tau_1).
	$$
	Therefore, we proved~\eqref{eq:t_2_ind_of_t1}.	
\end{proof}

\section{Small-time analysis}
\label{app:small_time}
{\rev
	Many results in this paper can be easily proven for $t\ll 1$ using the explicit expressions for $f'(0)$, $f''(0)$, and $f'''(0)$ in Section~\ref{subsec:initial_dynamics}:
	\begin{itemize}
		\item
		Theorem~\ref{thm:compare_p}: 
		
		By~\eqref{eq:n'(0)}, when the network is heterogeneous in $p$ and homogeneous in $q$, then $$\left(f^{\rm het}\right)'(0)=\frac{1}{M}\sum_{i=1}^{M}p_i, \qquad	\left(f^{\rm het}\right)''(0)=\frac{1}{M}\sum_{i=1}^{M}p_i\left(q-p_i\right)=\frac{q}{M}\sum_{i=1}^{M}p_i-\frac{1}{M}\sum_{i=1}^{M}p_i^2.$$ 
		By~\eqref{eq:initial_hom},
		$\left(f^{\rm hom}\right)'(0)=p$ and $\left(f^{\rm hom}\right)''(0)=qp-p^2$.
		Therefore, under the conditions of Theorem~\ref{thm:compare_p}, $$\left(f^{\rm het}\right)'(0)=\left(f^{\rm hom}\right)'(0), \qquad  \left(f^{\rm hom}\right)''(0)-\left(f^{\rm het}\right)''(0)={\rm Var}\{p_1,\dots,p_m\}>0.\footnote{Recall that $\frac{1}{M}\sum_{i=1}^{M}x_i^2-\E[X]^2=\frac{1}{M}\sum_{i=1}^{M}\left(x_i-\E[X]\right)^2=Var(X)$.}$$  Hence, $f^{\rm het}(t)<f^{\rm hom}(t)$ for $0<t\ll 1$.
		
		\item
		Theorem~\ref{thm:compare_q}:
		
		Under the conditions of Theorem~\ref{thm:compare_q}, by~\eqref{eq:n''(0)_p_hom} and~\eqref{eq:initial_hom}, $f'(0)=p$ and $f''(0)=p(q-p)$, both for the heterogeneous and homogeneous networks. Furthermore, by~\eqref{eq:n'''(0)_mild} and~\eqref{eq:initial_hom}, 
		$
		\left(f^{\rm het}\right)'''(0)=p^3+pq\left(\frac{M(M-2)}{(M-1)^2}q-4p\right)-\frac{(2M-3)}{M(M-1)^2}p\sum_{i=1}^{M}q_i^2
		$
		and
		$
		\left(f^{\rm hom}\right)'''(0)=p^3+pq\left(\frac{M(M-2)}{(M-1)^2}q-4p\right)-\frac{(2M-3)}{(M-1)^2}pq^2.
		$
		Therefore, $\left(f^{\rm hom}\right)'''(0)-\left(f^{\rm het}\right)'''(0)=p\frac{(2M-3)}{(M-1)^2}{\rm Var}\{q_1,\dots,q_M\}>0$. Hence, $f^{\rm het}(t)<f^{\rm hom}(t)$ for $0<t\ll 1$.
		
		\item
		Theorem~\ref{thm:compare_pq}:
		
		Under the conditions of Theorem~\ref{thm:compare_pq}, by~\eqref{eq:n'(0)} and~\eqref{eq:initial_hom},
		$	\left(f^{\rm het}\right)'(0)=\frac{1}{M}\sum_{i=1}^{M}p_i=p=\left(f^{\rm hom}\right)'(0).
		$
		By the Chebyshev sum inequality, see~\eqref{eq:chebyshev}, since $\{p_i\}$ and $\{q_i\}$ are positively correlated, then $\sum_{i=1}^{M}p_iq_i\geq \frac{1}{M}\sum_{j=1}^{M}q_j\sum_{i=1}^{M}p_i$, and so by~\eqref{eq:n''(0)_mild},
		\begin{equation*}
			\left(f^{\rm het}\right)''(0)\leq \frac{1}{M(M-1)}\left[\sum_{j=1}^{M}q_j\sum_{i=1}^{M}p_i-\frac{1}{M}\sum_{j=1}^{M}q_j\sum_{i=1}^{M}p_i\right]-\frac{1}{M}\sum_{i=1}^{M}p_i^2=pq-\frac{1}{M}\sum_{i=1}^{M}p_i^2.
		\end{equation*}
		Therefore, since $\left(f^{\rm hom}\right)''(0)=pq-p^2$, see~\eqref{eq:initial_hom}, then $\left(f^{\rm hom}\right)''(0)-\left(f^{\rm het}\right)''(0)\geq{\rm Var}\{p_1,\dots,p_m\}>0$. Hence, $f^{\rm het}(t)<f^{\rm hom}(t)$ for $0<t\ll 1$.
		
		\item
		Lemma~\ref{lem:compare_M=2pq}:
		
		By~\eqref{eq:n'(0)} and~\eqref{eq:initial_hom},
		$\left(f^{\rm het}\right)'(0)=\left(f^{\rm hom}\right)'(0)=p$.
		By~\eqref{eq:n'(0)} and~\eqref{eq:initial_hom}, $\left(f^{\rm het}\right)''(0)=2p(q-p),$ and $\left(f^{\rm hom}\right)''(0)=p(q-p).$ Therefore, the difference
		$\left(f^{\rm het}\right)''(0)-\left(f^{\rm hom}\right)''(0)=p(q-p)$
		is positive if $q>p$ and negative if $q<p$, which implies Lemma~\ref{lem:compare_M=2pq} for $0<t\ll 1$.
		\item
		Level of heterogeneity~(Section~\ref{sec:level}):
		
		By Lemma~\ref{lem:n_initial_dynamics} and Corollary~\ref{cor:n'''(0)_mild}, $f'(0)$ only depends on the mean of $p$, while $f''(0)$ depends also on the variance of $p$. Similarly, $f''(0)$ only depends on the mean of $q$, while $f'''(0)$ depends also on the variance of $q$. Hence, the effects of the variances are much smaller than those of the parameters themselves.
	\end{itemize}
}

\section{Proof of Lemma~\ref{lem:compare_M=2pq}.}
\label{app:lem:compare_M=2pq}
First, assume that $p\neq q$. By equation~\eqref{eqs:fc_M=2},
\begin{equation}
	\label{eq:two_sided_hom_M=2}
	f^{\rm hom}(t)=1-\left(\frac{p}{p-q}\right)e^{-(p+q)t}+\left(\frac{q}{p-q}\right)e^{-2pt}.
\end{equation}
We cannot use~\eqref{eqs:fc_M=2} to obtain $f^{\rm het}(t)$, because $b_1=\frac{q_{2,1}}{p_2-q_{2,1}}=\frac{0}{0}$. Therefore, we solve~\eqref{eqs:master_M=2} directly. By~\eqref{eq:s_M sol},
$	[S^1,S^2](t)=e^{-2pt}.
$
By~\eqref{eq:master_M=2} and~\eqref{eq:masterHeterofcInitial},
\begin{equation*}
	\begin{cases}
		\frac{d}{dt}[S^1](t)=-2p[S^1](t), &\qquad [S^1](0)=1,\\
		\frac{d}{dt}[S^2](t)=-2q[S^2](t)+2qe^{-2pt}, &\qquad [S^2](0)=1.
	\end{cases}
\end{equation*}
Therefore,
\begin{equation*}
	[S^1](t)=e^{-2pt}, \qquad [S^2](t)=\frac{q}{q-p}e^{-2pt}-\frac{p}{q-p}e^{-2qt},
\end{equation*}
and so, by~\eqref{eq:expectedHeterofc},
\begin{equation}
	\label{eq:M=2pq_het_equation}
	f^{\rm het}(t)=1-\frac{2q-p}{2(q-p)}e^{-2pt}+\frac{p}{2(q-p)}e^{-2qt}.
\end{equation}
Hence,
\begin{equation}
	\label{eq:f_pq_diff}
	g(t):=f^{\rm het}(t)-f^{\rm hom}(t)=\frac{p}{q-p}\left(\frac{1}{2}(e^{-2pt}+e^{-2qt})-e^{-(p+q)t}\right).
\end{equation}
For any strictly convex function ($f''>0$) we have that
\begin{equation}
	\label{eq:convex}
	f\left(\frac{x_1+x_2}{2}\right)<\frac{1}{2}\left(f(x_1)+f(x_2)\right).
\end{equation}
Since $f=e^{at}$ is convex, equations~\eqref{eq:f_pq_diff} and~\eqref{eq:convex} give that $g(t)>0$  if $q>p$ and $g(t)<0$ if $q<p$.
By continuity, $g(t)\equiv 0$ if $p=q$.

\section{Proof of Lemma~\ref{lem:masterHetero}.}
\label{app:lem:masterHetero}
By~\eqref{eq:masterHeterofc}, for $1\leq n \leq M-1$,
\begin{equation*}
	\begin{aligned}
		\frac{d}{dt}[S_k^j](t)=&\frac{d}{dt}[S^{j-k+1},S^{j-k+2},\dots,S^{j}](t)\\= -\Bigg(\Bigg(&\sum_{i=j-k+1}^{j}p_{i}\Bigg) +\sum_{\substack{l=1 \\ l \notin \{j-k+1,j-k+2,\dots,j\} }}^{M}\Bigg(\sum_{i=j-k+1}^{j} q_{l,i}\Bigg)\Bigg)[S^{j-k+1},S^{j-k+2},\dots,S^{j}](t)\\ &+\sum_{\substack{l=1 \\ l \notin \{j-k+1,j-k+2,\dots,j\} }}^{M}\Bigg(\sum_{i=j-k+1}^{j}q_{l,i}\Bigg)[S^{j-k+1},S^{j-k+2},\dots,S^{j},S^l](t).
	\end{aligned}
\end{equation*}
Since the network is a one-sided circle, $q_{l,i}\neq0$ only if $(i-l) \Mod M=1$. The only non-zero $q_{l,i}$ in the sum is $q_{j-k,j-k+1}$, since $i\in \{j-k+1,j-k+2,\dots,j\}$ and $l\notin \{j-k+1,j-k+2,\dots,j\}$.
Equations~\eqref{eq:masterHeteroM} and~\eqref{eq:masterHeteroMInitial} follow directly from~\eqref{eq:masterHeterofcM} and~\eqref{eq:masterHeterofcInitial}, respectively.

\section{Proof of Theorem~\ref{thm:one_sided}.}
\label{app:thm:one_sided}
We can rewrite~\eqref{eqs:masterHeteros} as a system of $M$ linear constant-coefficient ODEs
\begin{subequations}
	\label{eqs:one_sided_circle_ivp}
	\begin{equation}
		\label{eq:one_side_circle_odes}
		\dot{[\boldsymbol{S}^j]}=\boldsymbol{A}^j[\boldsymbol{S}^j],
	\end{equation}
	where
	\begin{equation*}
		[\boldsymbol{S}^j]= 
		\begin{pmatrix}
			[S_1^j](t)\\ [S_2^j](t)\\ \vdots\\ [S_M^j](t)
		\end{pmatrix}, \qquad 
		\dot{[\boldsymbol{S}^j]}= 
		\begin{pmatrix}
			\dot{[S_1^j]}(t)\\ \dot{[S_2^j]}(t)\\ \vdots\\ \dot{[S_M^j]}(t)
		\end{pmatrix}.
	\end{equation*}
	and $\boldsymbol{A}^j$ is the bi-diagonal matrix
	whose two non-zero diagonals are
	\begin{equation*}
		\begin{pmatrix}
			a_{1,1}\\a_{2,2}\\\vdots\\a_{k,k}\\\vdots\\a_{M,M}
		\end{pmatrix}
		=
		\begin{pmatrix}
			-p_j-q_j\\ -p_{j-1}-p_j-q_{j-1}\\\vdots\\(-\sum_{i=j-k+1}^{j}p_i) -q_{j-k+1}\\\vdots\\-\sum_{i=1}^{M}p_i
		\end{pmatrix},
		\qquad
		\begin{pmatrix}
			a_{1,2}\\a_{2,3}\\\vdots\\a_{k-1,k}\\\vdots\\a_{M-1,M}
		\end{pmatrix}
		=
		\begin{pmatrix}
			q_j\\q_{j-1}\\\vdots\\q_{j-k+1}\\\vdots\\q_{j-M+2}
		\end{pmatrix},
	\end{equation*}
	
	together with the initial condition 
	\begin{equation}
		\label{eq:initial_one_sided}
		[\boldsymbol{S}^j]_{|t=0}=
		\begin{pmatrix}
			1 \\ \vdots \\ 1
		\end{pmatrix}.
	\end{equation}
\end{subequations}

\begin{lemma}
	\label{lem:systemSolution}
	The solution of~\eqref{eqs:one_sided_circle_ivp} is
	\begin{equation}
		\label{eq:systemSolution}
		[\boldsymbol{S}^j]=\sum_{k=1}^{M}c_k^j\boldsymbol{v}_k^je^{\lambda_k^jt},
	\end{equation} 
	where $\lambda_k^j$ is given by~\eqref{eq:one_sided_eigenvalues}, $c_k^j$ is given by~\eqref{eq:one_sided_constants}, and
	\begin{equation*}
		v_k^j(n)=
		\begin{cases}
			\prod\limits_{m=j-k+2}^{j-n+1}\left(\frac{-q_m}{\left(\sum\limits_{i=j-k+1}^{m-1}p_i\right) +q_{j-k+1}-q_m}\right), &\qquad {\rm if}\ n\leq k-1, k=1,\dots,M-1,\\
			\prod\limits_{m=j-M+2}^{j-n+1}\left(\frac{-q_m}{\left(\sum\limits_{i=j-M+1}^{m-1}p_i\right) -q_m}\right), &\qquad {\rm if}\ n\leq M-1, k=M,\\
			1, &\qquad {\rm if}\ n=k,\\
			0, & \qquad {\rm otherwise.}
		\end{cases}
	\end{equation*}
\end{lemma}
\begin{proof}
	The eigenvalues of $\boldsymbol{A}^j$ are its diagonal elements $\{\lambda_k^j\}_{k=1}^M$. Hence we can solve for their corresponding eigenvectors $\{v_k^j\}_{k=1}^M$. The coefficients $\{c_k^j\}_{k=1}^M$ are determined from~\eqref{eq:systemSolution} and~\eqref{eq:initial_one_sided}, which gives
	$
	\sum_{k=1}^{M}c_k^j\boldsymbol{v}_k^j=
	\begin{pmatrix}
		1 \\ \vdots \\ 1
	\end{pmatrix}.
	$
\end{proof}
Since ${\rm Prob}(X_j(t)=0)=[S_1^j](t)$, the expected fraction of adopters is given by~\eqref{eq:expectedHeteroOneSidedCircle}.




\section{Proof of Lemma~\ref{lem:masterHetero2}.}
\label{app:lem:masterHetero2}
By~\eqref{eq:masterHeterofc}, for $q\leq n\leq M-1$,
\begin{equation*}
	\begin{aligned}
		\frac{d}{dt}[S_{m,n}^j](t)=\frac{d}{dt}[&S^{j-m},\dots,S^{j},\dots,S^{j+n}](t)=\\ &-\left(\left(\sum_{i=j-m}^{j+n}p_{i}\right) +\sum_{\substack{l=1 \\ l \notin \{j-m,\dots,j,\dots,j+n\} }}^{M}\left(\sum_{i=1}^{n} q_{l,i}\right)\right)[S^{j-m},\dots,S^{j},\dots,S^{j+n}](t)\\ &+\sum_{\substack{l=1 \\ l \notin \{j-m,\dots,j,\dots,j+n\} }}^{M}\left(\sum_{i=j-m}^{j+n}q_{l,i}\right)[S^{j-m},\dots,S^{j},\dots,S^{j+n},S^l](t).
	\end{aligned}
\end{equation*}
Since this network is a two-sided circle, $q_{l,i}\neq 0$ only if $|i-l|\Mod M=1$. Since $i\in \{j-m,\dots,j,\dots,j+n\}$ and $l\notin \{j-m,\dots,j,\dots,j+n\}$, the only non-zero $q_{l,i}$ are $q_{j-m-1,j-m}$ and $q_{j+n+1,j+n}$. Equations~\eqref{eq:masterHetero2M} and~\eqref{eq:masterHetero2Initial} follow directly from~\eqref{eq:masterHeterofcM} and~\eqref{eq:masterHeterofcInitial}, respectively.

\section{$[S_{0,0}^j]$ in the two-sided case.}
\label{app:S_{0,0}^j}
Equations~\eqref{eqs:masterHetero2} can be written as the system of $\frac{M(M-1)}{2}+1\; $linear constant-coefficient ODEs
\begin{subequations}
	\label{eqs:two_sided_ivp}
	\begin{equation}
		\label{eq:two_side_circle_odes}
		\dot{[\boldsymbol{S}^j]}=\boldsymbol{A}^j[\boldsymbol{S}^j],
	\end{equation}
	where 
	\begin{small}
		\begin{align*}
			[\boldsymbol{S}^j]&= 
			\begin{pmatrix}
				[S_{0,0}^j](t)\\ [S_{0,1}^j](t)\\ [S_{1,0}^j](t) \\ [S_{0,2}^j](t) \\ [S_{1,1}^j](t) \\ [S_{2,0}^j](t)\\ [S_{0,3}^j](t)\\ \vdots\\ [S_{M-2,0}^j](t)\\ [S_M^j](t)
			\end{pmatrix}, ~~
			\dot{[\boldsymbol{S}^j]}= 
			\begin{pmatrix}
				\dot{[S_{0,0}^j]}(t)\\ \dot{[S_{0,1}^j]}(t)\\ \dot{[S_{1,0}^j]}(t) \\ \dot{[S_{0,2}^j]}(t) \\ \dot{[S_{1,1}^j]}(t) \\\dot{[S_{2,0}^j]}(t)\\ \dot{[S_{0,3}^j]}(t)\\ \vdots\\\dot{[S_{M-2,0}^j]}(t)\\ \dot{[S_M^j]}(t)
			\end{pmatrix},\\
			\boldsymbol{A}^j&=
			\left(
			\begin{matrix}
				a_1 & q_{j+1,j} & q_{j-1,j} & 0&\hdots&\hdots & \hdots & \hdots & 0 \\
				0 & a_2 & 0 & q_{j+2,j+1}&q_{j-1,j}&0 & \ddots & \ddots &\ddots \\
				0&0& a_3 & 0&q_{j+1,j}&q_{j-2,j-1} & 0 & \ddots &\ddots\\
				\vdots   & \ddots        & 0       &\ddots&\hdots&\ddots &\ddots   & \ddots &0\\
				\vdots&\ddots&0&\ddots&a_4&\hdots & q_{j+b+1,j+b} & q_{j-a-1,j-a} &0\\
				\vdots   & \ddots  & \ddots  &\ddots&\ddots  &\ddots   & \ddots & \ddots &a_5\\
				0 &    \ddots     &\ddots & \hdots&\hdots&\hdots    &0        &0      &-\sum_{i=1}^{M}p_i
			\end{matrix}
			\right)
		\end{align*}
	\end{small}
	and 
	\begin{align*}
		a_1&=-p_j-q_{j-1,j}-q_{j+1,j}, \quad 		a_2=-p_{j}-p_{j+1}-q_{j-1,j}-q_{j+2,j+1}, \quad
		a_3=-p_{j-1}-p_{j}-q_{j-2,j-1}-q_{j+1,j},\\
		a_4&=\left(-\sum_{i=j-a}^{j+b}p_i\right) -q_{j-a-1,j-a}-q_{j+b+1,j+b},\quad
		a_5=q_{j-M+1,j-M+2}+q_{j+1,j}.
	\end{align*}
	together with the initial conditions 
	\begin{equation}
		\label{eq:two_sided_initial}
		[\boldsymbol{S}^j]_{|t=0}=
		\begin{pmatrix}
			1 \\ \vdots \\ 1
		\end{pmatrix}.
	\end{equation}
\end{subequations}
The eigenvalues of $A^j$ are its diagonal elements
\begin{equation*}
	\lambda_{a,b}^j = \left(-\sum_{i=j-a}^{j+b}p_i\right) -q_{j-a-1,j-a}-q_{j+b+1,j+b},  \qquad a,b=0,1,\dots,M-2,\qquad a+b\leq M-2,
\end{equation*}
and
$\lambda_M=-\sum_{i=1}^{M}p_i$.

We have not yet found a way to explicitly solve for the eigenvectors of $A^j$ for a general $M$. Under the assumption that all eigenvalues are unique, for $M=2$, the 2 eigenvectors are
\begin{equation*}
	v_{0,0}^j=
	\begin{pmatrix}
		1\\0
	\end{pmatrix}, \qquad
	v_2^j=
	\begin{pmatrix}
		-\frac{q_{j-1,j}}{p_{j-1}-q_{j-1,j}}\\1
	\end{pmatrix}.
\end{equation*}
When $M=3$, the 4 eigenvectors are
\begin{align*}
	v_{0,0}^j=
	\begin{pmatrix}
		1\\0\\0\\0
	\end{pmatrix},&\qquad
	v_{0,1}^j=
	\begin{pmatrix}
		-\frac{q_{j+1,j}}{p_{j+1}+q_{j+2,j+1}-q_{j+1,j}}\\1\\0\\0
	\end{pmatrix}, \qquad
	v_{1,0}^j=
	\begin{pmatrix}
		-\frac{q_{j-1,j}}{p_{j-1}+q_{j-2,j-1}-q_{j-1,j}}\\0\\1\\0
	\end{pmatrix},\\
	&v_3^j=
	\begin{pmatrix}
		\frac{q_{j+1,j}\left(\frac{q_{j-1,j}+q_{j+2,j+1}}{p_{j-1}-q_{j-1,j}-q_{j+2,j+1}}\right)+q_{j-1,j}\left(\frac{q_{j-2,j-1}+q_{j+1,j}}{p_{j+1}-q_{j-2,j-1}-q_{j+1,j}}\right)}{p_{j+1}+p_{j-1}-q_{j-1,j}-q_{j+1,j}}\\-\frac{q_{j-1,j}+q_{j+2,j+1}}{p_{j-1}-q_{j-1,j}-q_{j+2,j+1}}\\-\frac{q_{j-2,j-1}+q_{j+1,j}}{p_{j+1}-q_{j-2,j-1}-q_{j+1,j}}\\1
	\end{pmatrix}.
\end{align*}
After solving for the eigenvectors of $A^j$, we get that
\begin{equation}
	\label{eq:sol_two_sided}
	[\boldsymbol{S}^j]=\sum_{\substack{a,b \in \left\lbrace 0,\dots,M-2\right\rbrace \\ a+b \leq M-1} }c_{a,b}^j\boldsymbol{v}_{a,b}^je^{\lambda_{a,b}^jt},
\end{equation}
where $(c_{a,b},\boldsymbol{v}_{a,b},\lambda_{a,b})=(c_M,\boldsymbol{v}_M,\lambda_M)$ when $a+b=M-1$, and the coefficients $c_{a,b}^j$ are determined by~\eqref{eq:sol_two_sided} and~\eqref{eq:two_sided_initial}:
\begin{equation*}
	\sum_{\substack{a,b \in \left\lbrace 0,\dots,M-2\right\rbrace \\ a+b \leq M-1}}c_{a,b}^j\boldsymbol{v}_{a,b}^j=
	\begin{pmatrix}
		1 \\ \vdots \\ 1
	\end{pmatrix}.
\end{equation*}

\section{Proof of eq.~\eqref{eq:der_dimension_initial}.}
\label{app:lem:dimension_initial}


The equations for $f_D'(0)$ and $f_D''(0)$ follow from~\eqref{eq:n''(0)_p_hom} by letting $M\to\infty$. In order to derive $f_D'''(0)$, we first note that since each node is connected to its $2D$ neighbors, then  by~\eqref{eq:masterHeterofc} and~\eqref{eq:cart_hom},
\begin{equation*}
	\frac{d}{dt}[S^{{\bf i}}](t)=-\left(p+\sum_{j=1}^{2D}\frac{q}{2D}\right)[S^{{\bf i}}](t)+\sum_{j=1}^{D}\frac{q}{2D}[S^{\bf i}, S^{{\bf i}+{\bf e}_j}](t)+\sum_{j=1}^{D}\frac{q}{2D}[S^{\bf i}, S^{{\bf i}-{\bf e}_j}](t),
\end{equation*}
where ${\bf e}_j(i)$ is the unit vector in the $j$th coordinate.
By translational invariance, $[S^{\bf i}]$ and $[S^{\bf i}, S^{{\bf i}\pm {\bf e}_j}]$ are independent of ${\bf i}$, ${\bf e}_j$, and $\pm$, and so
\begin{equation}
	\label{eq:masterHeterodD}
	\frac{d}{dt}[S^{{\bf i}}](t)=-\left(p+q\right)[S^{{\bf i}}](t)+q[S_2](t).
\end{equation} 
{\rev Let $(S_3^-)(t)$ denote the event that any configuration in which 3 adjacent non-adopters are colinear appears at time $t$, e.g. $(S^{{\bf i}},S^{{\bf i}+{\bf e}_j},S^{{\bf i}+2{\bf e}_j})(t)$, and let $(S_3^L)(t)$ denote the event that any configuration in which 3 adjacent non-adopters form an L-shape appears at time $t$, e.g. $(S^{{\bf i}},S^{{\bf i}+{\bf e}_j},S^{{\bf i}+{\bf e}_j+{\bf e}_k})(t)$, where $j\neq k$. Let $[S_3^-](t)$ and $[S_3^L](t)$ denote the probabilities of these events, respectively.} By translational invariance, $[S_3^-](t)$ is the same for any such configuration, and similarly for $[S_3^L](t)$. Therefore, by~\eqref{eq:masterHeterofc} and~\eqref{eq:cart_hom},
\begin{equation}
	\label{eq:masterHeterodD2}
	\frac{d}{dt}[S_2](t)=-\left(2p+2\left(\frac{q}{2D}\right)+\left(4D-4\right)\left(\frac{q}{2D}\right)\right)[S_2](t)+2\left(\frac{q}{2D}\right)[S_3^-](t)+\left(4D-4\right)\left(\frac{q}{2D}\right)[S_3^L](t),
\end{equation}
since each node in the configuration $(S_2)$ is potentially influenced by $2D-1$ nodes, $2D-2$ of which form an L-shape when added to $(S_2)$, and 1 of which forms a line.
Differentiating~\eqref{eq:masterHeterodD} and using~\eqref{eq:masterHeterofcInitial} and~\eqref{eq:s'_0} gives
\begin{subequations}
	\label{eqs:cart_initial}
	\begin{equation}
		\frac{d^2}{dt^2}[S](0)=-(p+q)[S]'(0)+q[S_2]'(0)=p(p-q).
	\end{equation}
	Furthermore, by~\eqref{eq:masterHeterodD2},
	\begin{equation}
		\frac{d^2}{dt^2}[S_2](0)=\left(2p+\left(4D-2\right)\left(\frac{q}{2D}\right)\right)(2p)-\left(4D-2\right)\left(\frac{q}{2D}\right)\left(3p\right)=p\left(4p-\frac{2D-1}{D}q\right).
	\end{equation}
\end{subequations}
Therefore, differentiating~\eqref{eq:masterHeterodD} twice and substituting in the right hand sides of equations~\eqref{eqs:cart_initial} gives
\begin{equation*}
	\frac{d^3}{dt^3}[S_1](0)=-\left(p+2D\left(\frac{q}{2D}\right)\right)p(p-q)+\left(2D\right)\left(\frac{q}{2D}\right)p\left(4p-\frac{2D-1}{D}q\right)=-p\left(p^2-4pq+\frac{D-1}{D}q^2\right),
\end{equation*}
and so we get the desired result.
%

\section{Proof of Lemma~\ref{lem:order}.}
\label{app:lem:order}
For any finite $t$, nodes whose distance from the interface between $p_1$ and $p_2$ is $\gg qt$ only ``see" a homogeneous environment. Therefore, as $M\to \infty$, the interaction between nodes with different $p_i$ values becomes negligible, and so
\begin{equation*}
	f_A^{\rm het}(t)\sim \frac{f_{\rm line}^{\rm 1-sided}(t;p_1,q,\frac{M}{2})+f_{\rm line}^{\rm 1-sided}(t;p_2,q,\frac{M}{2})}{2}.
\end{equation*}
Since $\lim\limits_{M\to \infty}f_{\rm line}^{\rm 1-sided}(t;p,q,M)=\lim\limits_{M\to \infty}f_{\rm circle}^{\rm 1-sided}(t;p,q,M)$, see~\cite{Bass-boundary-18}, and $\lim\limits_{M\to \infty}f_{\rm circle}^{\rm 1-sided}(t;p,q,M)=f_{\rm 1D}(t;p,q)$, see~\eqref{eq:f_1D}, result~\eqref{eq:f_A} for $f_A$ follows.
$\\\\$
\noindent In circle $B$, as $M\to \infty$, by translation symmetry, $[S_k^i](t)\equiv [S_k^{i+2}](t)$ for any $i$ and $k$. Therefore,
\begin{equation*}
	f_B^{\rm het}(t)=1-\frac{[S_1^1](t)+[S_1^2](t)}{2}.
\end{equation*}
Let $p:=\frac{p_1+p_2}{2}$. By~\eqref{eq:masterHetero},
\begin{equation*}
	\dot{[S_k^1]}(t)=
	\begin{cases}
		-((k-1)p+p_1+q)[S_k^1](t)+q[S_{k+1}^1](t),&\qquad $k$ \ {\rm odd,}\\
		-(kp+q)[S_k^1](t)+q[S_{k+1}^1](t),&\qquad $k$ \ {\rm even,}
	\end{cases}
\end{equation*}
where $[S_k^1](0)=1$ for all $k$. The substitution
\begin{equation*}
	[S_k^1](t)=
	\begin{cases}
		e^{-(k-1)pt}T_1(t),&\qquad $k$ \ {\rm odd,}\\
		e^{-kpt}R_1(t),&\qquad $k$ \ {\rm even,}
	\end{cases}
\end{equation*}
reduces this infinite system of ODEs to the two coupled ODEs
\begin{align*}
	\dot{R_1}(t)&=-qR_1(t)+qT_1(t),\qquad \qquad \;\;\;\;\qquad R_1(0)=1,\\
	\dot{T_1}(t)&=-(p_1+q)T_1(t)+qe^{-2pt}R_1(t),\qquad T_1(0)=1.
\end{align*} 
Furthermore, the substitutions $R(t)=e^{-qt}U_1(t)$ and $T_1(t)=e^{-(p_1+q)t}V_1(t)$
yield system~\eqref{eq:v1_system} for $U_1(t)$ and $V_1(t)$, and so
\begin{equation*}
	[S_1^1](t)=T_1(t)=e^{-(p_1+q)t}V_1(t)=\frac{1}{q}e^{-qt}\dot{U_1}(t).
\end{equation*}
Repeating this procedure for the infinite system $\{[S^2_k](t)\}_k$ yields
$	[S_1^2](t)=\frac{1}{q}e^{-qt}\dot{U_2}(t)$,
where $U_2(t)$ is the solution of 
$$
\dot{U_2}(t)=qe^{-p_2t}V_2(t), \quad \dot{V_2}(t) =qe^{-p_1t}U_2(t) \qquad U_2(0)= V_2(0)=1.
$$
Comparing this system with~\eqref{eq:v1_system} shows that $V_1(t)\equiv U_2(t)$, and so
$
[S_1^2](t)=\frac{1}{q}e^{-qt}\dot{V_1}(t).
$
Hence,
\begin{equation*}
	f_B(t)=1-\frac{1}{2}\left([S_1^1](t)+[S_1^2](t)\right)=1-\frac{1}{2q}e^{-qt}(\dot{V_1}(t)+\dot{U_1}(t)).
\end{equation*}

\section*{Acknowledgments}
\label{sec:acknowledgments}
We thank Eilon Solan for help with the proof of Theorem~\ref{thm:CDF-dominance}, and Tomer Levin for useful discussions.

\bibliographystyle{plain}
\bibliography{Heterogeneity_MOR_Final_arXiv}

\begin{thebibliography}{10}

\bibitem{Albert-00}
R.~Albert, H.~Jeong, and A.L. Barab\'asi.
\newblock Error and attack tolerance of complex networks.
\newblock {\em Nature}, 406:378--382, 2000.

\bibitem{Anderson-92}
R.M. Anderson and R.M. May.
\newblock {\em Infectious Diseases of Humans}.
\newblock Oxford University Press, Oxford, 1992.

\bibitem{Bass-69}
F.M. Bass.
\newblock A new product growth model for consumer durables.
\newblock {\em Management Sci.}, 15:1215--1227, 1969.

\bibitem{Bulte2007NewPD}
C.~Bulte and Y.~V. Joshi.
\newblock New product diffusion with influentials and imitators.
\newblock {\em Marketing Science}, 26:400--421, 2007.

\bibitem{Chatterjee-Eliashberg-90}
R.~Chatterjee and J.~Eliashberg.
\newblock The innovation diffusion process in a heterogeneous population: A
micromodeling approach.
\newblock {\em Management Science}, 36(9):1057--1079, 1990.

\bibitem{de1903laws}
G.~De~Tarde.
\newblock {\em The laws of imitation}.
\newblock H. Holt, 1903.

\bibitem{Bass-SIR-model-16}
G.~Fibich.
\newblock Bass-{SIR} model for diffusion of new products in social networks.
\newblock {\em Phys. Rev. E}, 94:032305, 2016.

\bibitem{Bass-SIR-analysis-17}
G.~Fibich.
\newblock Diffusion of new products with recovering consumers.
\newblock {\em SIAM J. Appl. Math.}, 2017.

\bibitem{PNAS-12}
G.~Fibich, A.~Gavious, and E.~Solan.
\newblock Averaging principle for second-order approximation of heterogeneous
models with homogeneous models.
\newblock {\em Proc. Natl. Acad. Sci. USA}, 109:19545--19550, 2012.

\bibitem{OR-10}
G.~Fibich and R.~Gibori.
\newblock Aggregate diffusion dynamics in agent-based models with a spatial
structure.
\newblock {\em Oper. Res.}, 58:1450--1468, 2010.

\bibitem{FIBICH2020123055}
G.~Fibich and T.~Levin.
\newblock Percolation of new products.
\newblock {\em Physica A}, 540:123055, 2020.

\bibitem{Bass-boundary-18}
G.~Fibich, T.~Levin, and O.~Yakir.
\newblock Boundary effects in the discrete {B}ass model.
\newblock {\em SIAM J. Appl. Math.}, 79:914--937, 2019.

\bibitem{GLM-01}
J.~Goldenberg, B.~Libai, and E.~Muller.
\newblock Using complex systems analysis to advance marketing theory
development. (special issue on emergent and co-evolutionary processes in
marketing.).
\newblock {\em Acad. Market. Sci. Rev.}, 9:1--19, 2001.

\bibitem{Jackson-08}
M.O. Jackson.
\newblock {\em Social and Economic Networks}.
\newblock Princeton University Press, Princeton and Oxford, 2008.

\bibitem{Mahajan-93}
V.~Mahajan, E.~Muller, and F.M. Bass.
\newblock New-product diffusion models.
\newblock In J.~Eliashberg and G.L. Lilien, editors, {\em Handbooks in
Operations Research and Management Science}, volume~5, pages 349--408.
North-Holland, Amsterdam, 1993.

\bibitem{Niu-02}
S.C. Niu.
\newblock A stochastic formulation of the {B}ass model of new product
diffusion.
\newblock {\em Math. Problems Engrg.}, 8:249--263, 2002.

\bibitem{Pastor-Satorras-01}
R.~Pastor-Satorras and A.~Vespignani.
\newblock Epidemic spreading in scale-free networks.
\newblock {\em Phys. Rev. Lett.}, 86:3200--3203, 2001.

\bibitem{Rand-11}
W.~Rand and R.T. Rust.
\newblock Agent-based modeling in marketing: Guidelines for rigor.
\newblock {\em Intern. J. of Research in Marketing}, 28:181--193, 2011.

\bibitem{Rogers-03}
E.M. Rogers.
\newblock {\em Diffusion of Innovations}.
\newblock Free Press, New York, fifth edition, 2003.

\bibitem{Strang-98}
D.~Strang and S.A. Soule.
\newblock Diffusion in organizations and social movements: From hybrid corn to
poison pills.
\newblock {\em Annu. Rev. Sociol.}, 24:265--290, 1998.

\bibitem{Hopp-04}
ed. W.J.~Hopp.
\newblock Ten most influential papers of management science's first fifty
years.
\newblock {\em Management Sci.}, 50:1763--1893, 2004.

\end{thebibliography}

\end{document}